\documentclass[12pt]{amsart}
\usepackage{amsmath,amsthm,amscd,color}
\usepackage{amssymb}
\usepackage{amsfonts}
\usepackage{cite}
\allowdisplaybreaks
\usepackage{geometry}

\newtheorem{theorem}[equation]{Theorem}
\newtheorem{lemma}[equation]{Lemma}
\newtheorem{definition}[equation]{Definition}
\newtheorem{proposition}[equation]{Proposition}

\newtheorem{remark}[equation]{Remark}

\numberwithin{equation}{section}
\numberwithin{table}{section}

\title{The Shifted convolution L-Function for Maass forms}
\author{Dorian Goldfeld,\\ Gerhardt Hinkle,\\Jeffrey Hoffstein}
\address{Dorian Goldfeld: Department of Mathematics \\ Columbia University \\ New York \\NY 10027 \\USA\\
}
\email{goldfeld@columbia.edu}

\address{Gerhardt Hinkle: Department of Mathematics \\ Brown University \\ Providence \\RI 02912 \\USA}
\email{gerhardt.hinkle@gmail.com}

\address{Jeffrey Hoffstein: Department of Mathematics \\ Brown University \\ Providence \\RI 02912 \\USA}
\email{jhoff@math.brown.edu}
\thanks{Dorian Goldfeld is partially supported by Simons Travel Support for Mathematicians MP-TSM-00001990.}

\keywords{Shifted convolution sums, L-functions, Maass forms, Appell hypergeometric function, Picard hypergeometric function}

\subjclass[2020]{Primary 11F30, 11F72; Secondary 11F12, 11M41, 33C65}

\begin{document}

\maketitle

\begin{abstract}
  Let $\Phi_1,\Phi_2$ be Maass forms for $\textup{SL}(2,\mathbb Z)$ with Fourier coefficients $C_1(n),C_2(n)$.   
  For a positive integer $h$ the meromorphic continuation and growth in  $s\in\mathbb C$ (away from poles) of the shifted convolution L-function 
$$L_h(s,{\Phi_1,\Phi_2})\, := \sum_{n \neq 0,-h} {C_1(n) C_2(n + h)} \cdot \big|n(n + h)\big|^{-\frac{1}{2}s}$$
is obtained. For ${\rm Re}(s) > 0$ it is shown that the only poles  are possible simple poles at $\frac{1}{2} \pm ir_k$, where $\tfrac14+r_k^2$ are eigenvalues of the Laplacian. As an application we obtain, for $T\to\infty$, the asymptotic formula 
{\tiny \begin{align*}
   & \underset{n \neq 0,-h}{\sum_{\sqrt{|n (n + h)|}<T} } \hskip-5pt{C_1(n) C_2(n + h)} \left(\textup{log}\Big(\tfrac{T}{\sqrt{|n (n + h)|}}\,\Big)\right)^{\frac{3}{2} + \varepsilon} \hskip-7pt =\; f_{{\mathfrak r_1,\mathfrak r_2,}h,\varepsilon}(T) \cdot T^{\frac{1}{2}} \; + \; \mathcal O\left( h^{1-\varepsilon} T^\varepsilon + h^{1 + \varepsilon} T^{-1 - \varepsilon}  \right),
  \end{align*}}\noindent
 where the function $f_{{\mathfrak r_1,\mathfrak r_2,}h,\varepsilon}(T)$ is given as an explicit spectral sum that satisfies the bound $f_{{\mathfrak r_1,\mathfrak r_2,}h,\varepsilon}(T) \ll h^{\theta + \varepsilon}$. We also obtain a sharp bound for the  above shifted convolution sum with sharp cutoff, i.e., without the smoothing weight $\log(*)^{\frac32+\varepsilon}$ with uniformity in the $h$ aspect. Specifically, we show that for $h < x^{\frac{1}{2} - \varepsilon}$,
  \[
    {\sum_{\sqrt{|n (n + h)|} < x} C_1(n) C_2(n + h)} \ll h^{\frac{2}{3}\theta + \varepsilon}x^{\frac{2}{3} (1 + \theta) + \varepsilon} + h^{\frac{1}{2} + \varepsilon}x^{\frac{1}{2} + 2\theta + \varepsilon}.
  \]
\end{abstract}

\section{\large\bf Introduction}

 Let $f, g$ be cusp forms for a congruence subgroup of ${\rm SL}(2,\mathbb Z)$. Let $c_f(n), c_g(n)$ denote the $n^{th}$ Fourier coefficients  (for $n=1,2,3\ldots$) of $f,g$, respectively. The Rankin-Selberg L-function (see \cite{MR0001249}, \cite{MR0002626}), 
 \begin{equation}\label{Rankin-SelbergL}
 L(s, f\otimes g) := \sum_{n=1}^\infty c_f(n) c_g(n)\, n^{-s}, \qquad\quad (s\in\mathbb C,\; {\rm Re}(s)\gg 1)
 \end{equation} 
 has played a major role in analytic number theory. Let $h$ be a positive integer. In a breakthrough paper, Selberg \cite{MR0182610} introduced for the first time the shifted convolution L-functions of the form
 \begin{equation} \label{SelbergLfunction} L_h(s, f\otimes g) := \sum_{n=1}^\infty c_f(n) c_g(n+h)\, (2n+h)^{-s}
 \end{equation}
 for holomorphic cusp forms $f,g$ and pointed out that these series (when suitably normalized) have meromorphic continuation to all of $\mathbb C$ and are holomorphic for ${\rm Re}(s)>\tfrac12$ except possibly for a finite number of simple poles in the segment $\tfrac12<s\le 1$, and that a pole at $s=1$ can only occur if $f, g$ are of the same type. 
 
 Unfortunately the L-function in  (\ref{SelbergLfunction}) does not satisfy a functional equation so cannot be considered as a natural generalization of the Rankin-Selberg L-function (\ref{Rankin-SelbergL}). In \cite{MR0556667}  the Dirichlet series 
  $$Z_h(s) = \frac{2^{-s}\Gamma(s)}{\Gamma\left(s-\tfrac32-k\right)}\sum_{n=1}^\infty c_f(n) c_g(n+h) \left(\frac{n}{2n+h}   \right)^s F\left(\tfrac{s}{2},\tfrac{s+1}{2}, s+\tfrac32-k; \tfrac{h^2}{(2n+h)^2}\right)$$
  is introduced for holomorphic modular forms of weight $k$ belonging to the modular group. Here $F\left(\tfrac{s}{2},\tfrac{s+1}{2}, s+\tfrac32-k; \tfrac{h^2}{(2n+h)^2}\right)$ is the Gauss hypergeometric function. It is shown that $Z_h(s)$ can be continued as a meromorphic function over $\mathbb C$ which is regular for ${\rm Re}(s) \ge k-\tfrac12$ with the exception of simple poles on the line 
  ${\rm Re}(s) = k-\tfrac12$ arising from the eigenvalues of the Laplacian. Moreover $Z_h(s)$ satisfies a functional equation which expresses $Z_h(s)-Z_h(2k-1-s)$ as a product of Gamma functions, zeta functions, and the Rankin-Selberg L-function $L(s, f\otimes g)$.

 The shifted convolution problem is the problem of obtaining bounds for sums of the form 
\begin{equation} \label{ShiftedConvProblem} \sum_{n\le x} c_f(n) c_g(n+h)\end{equation}
 as $x\to\infty.$ One of the earliest results  goes back to Ingham \cite{MR1574426} who showed that
 $$\sum_{n\le x} d(n) d(n+h) \sim \frac{6}{\pi^2} \sigma_{-1}(h) x (\log x)^2,$$
 where $d(n)$ denotes the number of positive divisors of $n$ and $\sigma_{-1}(n) =\sum_{d|n} d^{-1}$ with the sum again going over positive divisors.
 
 The shifted convolution problem for cusp forms $f,g$  was first considered in \cite{MR0556667} for the case of cusp forms $f,g$ of weight $k$ for the modular group where the shifted convolution sum contains a weight function.
 It was proved in \cite{MR0556667} that
 $$\sum_{n\le x} c_f(n) c_g(n+h) e^{-\frac{n}{x}} \ll x^{k-\tfrac12+\varepsilon}.$$
 
   Several years later Anton Good \cite{MR0701361} proved that  the unsmoothed shifted convolution sum for the Ramanujan tau function (Fourier coefficient of the weight 12 cusp form for ${\rm SL}(2,\mathbb Z)$) satisfies
 $$\sum_{1\le n\le x} \frac{\tau(n) \tau(n+h)}{n^{11}} \ll x^{\frac23+\varepsilon}.$$
Here we divide $\tau(n)$ by $n^{\frac{11}{2}}$ to normalize the Fourier coefficient so it behaves like a bounded function on average.  As pointed out in Math Reviews, Good's proof generalizes to holomorphic cusp forms on finitely generated discrete groups of the first kind containing translation by one. 

In the case when $f,g$ are holomorphic cusp forms for a congruence subgroup of ${\rm SL}(2,\mathbb Z)$, Hoffstein and Hulse \cite{MR3435737} 
 introduced the shifted convolution L-functions
\begin{equation}\label{HHseries}
 L_h(s,f\otimes g) := \sum_{n=1}^\infty  c_f(n) c_g(n+h)\, n^{-s}.
\end{equation}
 Notice that in Selberg's definition (\ref{SelbergLfunction}) of the shifted convolution L-function, he has $(2n+h)^{-s}$ while Hoffstein and Hulse have the simpler $n^{-s}$ which doesn't depend on $h$.
 In fact, Selberg's construction can be modified very slightly to get the meromorphic continuation of 
 $$L_h(s,f\otimes g) := \sum_{n=1}^\infty  c_f(n) c_g(n-h)\, n^{-s},$$
 where $h$ is a positive integer.  As the functions involved are holomorphic this limits the sum to $n>h$.  Curiously, it is quite hard to modify Selberg's construction to achieve the continuation of the series
 \eqref{HHseries}, where $h$ is positive.   This is something which it is sometimes desirable to do if one wants to average over $h$ and $x$ in different ranges in the shifted convolution problem (\ref{ShiftedConvProblem}).   However Hoffstein and Hulse were unable to use their methods to find a meromorphic continuation of the corresponding series when both of $f,g$ were Maass forms.
 
 The shifted convolution problem (with smooth weights) for the case when $f,g$ are holomorphic or Maass cusp forms was investigated by Blomer and Harcos in \cite{MR2437682} using spectral methods introduced by Selberg \cite{MR0182610} (see also \cite{MR4405745}). Let $$W_1,W_2:\mathbb R^\times\to\mathbb C$$ be smooth compactly supported weight functions. Assume that $f,g$ are holomorphic or Maass cusp forms normalized so that the Fourier coefficients $c_f(n), c_g(n)$ are bounded on average. Then they prove for all $x>0$ that
$$\sum_{m+n=h} c_f(|m|) c_g(|n|) W_1\left(\frac{m}{x}\right) W_2\left(\frac{n}{x}\right) \ll h^{\theta+\varepsilon}x^{\frac12+\varepsilon},$$
where $\theta$ is the best progress toward the Ramanujan--Petersson conjecture for Maass cusp forms for $\textup{SL}(2,\mathbb Z)$. Currently, $\theta = \frac{7}{64}$ (see Appendix 2 of \cite{MR1937203}); if the Ramanujan--Petersson conjecture holds, then $\theta = 0$. In addition they found a spectral decomposition for the Dirichlet series
$$
\sum_{\substack{m,n \ge
1\\m-n=h}}\frac{c_f(m)c_g(n)}{(m+n)^s}\left(\frac{\sqrt{mn}}{m+n}\right)^{
100}
$$
back to $\textup{Re}(s) >\frac12$, with polynomial growth on vertical lines in the $s$ aspect and uniformity in the $h$ aspect.

In this paper we initially restrict attention to the case when $f=g =\phi$ is a fixed Maass cusp form for ${\rm SL}(2,\mathbb Z)$. Curiously, one key part of our proof can be greatly simplified by an identity that depends on the fact that the Maass forms are equal. See (\ref{CuriousIdentity}) and Lemma \ref{identity1}, as well as Section \ref{Different} in which we obtain an analogous novel, more complicated identity that does not require that assumption and resolve the additional complications that arise in the proof as a result.

We shall define a shifted convolution L-function for this situation and prove an asymptotic formula for the associated shifted convolution problem. The asymptotic formula is a new result that as far as we know has not previously appeared in the literature. We now make some definitions and state the main results of this paper.

\begin{definition} {\bf(The fixed Maass cusp forms $\Phi_1,\Phi_2$}) 
 \label{FixedMaassForm}
For the rest of this paper we fix two Maass cusp forms $\Phi_1,\Phi_2$ for ${\rm SL}(2,\mathbb Z)$ with Laplace eigenvalue {$\Lambda_i = \tfrac14+\mathfrak r_i^2$ (where $\mathfrak r_1,\mathfrak r_2>0$)} which have the Fourier expansion
 $$
  {\Phi_i(z) = \sum_{m \neq 0} C_i(m)y^{\frac{1}{2}}K_{i\mathfrak r_i}(2\pi|m|y) e^{2\pi imx} \textup{ } (i = 1,2)}.
  $$
  \end{definition}

\begin{definition} {\bf (Shifted convolution L-function associated to $\Phi_1,\Phi_2$}).   For any positive integer $h$ and $s\in\mathbb C$ with $\text{\rm Re}(s)>1$, we define the shifted convolution L-function associated to the fixed Maass cusp forms $\Phi_1$ and $\Phi_2$ by
$$L_h(s, \Phi_1,\Phi_2)\, := \sum_{n \neq 0,-h} C_1(n) C_2(n + h) \cdot \big|n(n + h)\big|^{-\frac{1}{2}s},$$
which converges absolutely for $\text{\rm Re}(s) > 1.$
\end{definition}  

\begin{definition} {\bf (The constant $\theta$)} \label{theta} Let $\theta$ denote the best progress toward the Ramanujan--Petersson conjecture for Maass forms for $\textup{SL}(2,\mathbb Z)$. Currently, $\theta = \frac{7}{64}$ (see Appendix 2 of {\rm \cite{MR1937203}}); if the Ramanujan--Petersson conjecture holds, then $\theta = 0$.
\end{definition} 

The main object of this paper is to obtain the meromorphic continuation of $L_h(s,\Phi_1,\Phi_2)$ to $\text{\rm Re}(s)> 0$ as well as its growth as $|s|\to\infty.$ We shall prove the following theorem.
\vskip 6pt
Let $\phi_1, \phi_2, \phi_3,\ldots$,  denote an orthonormal basis of Maass cusp forms for $\textup{SL}(2,\mathbb Z)$ where each $\phi_k$ has Laplace eigenvalue $\lambda_k = \frac{1}{4} + r_k^2$ with $r_k> 0$ and $r_1\le r_2\le r_3\le\cdots.$

\begin{theorem} \label{MainTheorem} {\bf (Meromorphic continuation and bounds for $L_h(s,\Phi_1,\Phi_2)$   )} Fix $\varepsilon > 0$. Let $h$ be a positive integer. The shifted convolution L-function $L_h(s,\Phi_1,\Phi_2)$ has meromorphic continuation to $\text{\rm Re}(s)> \varepsilon$ with possible simple poles at $$s = \tfrac{1}{2} \pm ir_k \qquad\quad (\text{for} \; k=1,2,3,\ldots)$$ and no other poles in this region. Let $s=\sigma+it$ with $\sigma>\varepsilon$ and $|t|\to\infty.$ Then for $|s-\rho_k|>\varepsilon$ {(where as an abuse of notation we let $\rho_k$ range over all poles $\frac{1}{2} \pm ir_k$)}, we have the bound
  \[
    L_h(s, \Phi_1,\Phi_2) \ll
    \begin{cases}
      h^{\frac{1}{2} -\sigma + \theta + \varepsilon} |s|^{\frac{3}{2} - \sigma + \varepsilon} + h^{1 - \sigma + \varepsilon} |s|^{1 - \sigma + \varepsilon} & \varepsilon < \sigma \le \frac{1}{2}, \\
      h^{(2\theta + \varepsilon) (1 - \sigma + \varepsilon)} |s|^{2 (1 - \sigma + \varepsilon)} + h^{1 - \sigma + \varepsilon} |s|^{1 - \sigma + \varepsilon} & \frac{1}{2} \le \sigma \le 1 + \varepsilon, \\
      1 & 1 + \varepsilon\le \sigma.
    \end{cases}
  \]

 Furthermore, the residues at the poles $s =\tfrac12 \pm ir_k$ are given by
 
 \begin{align*}
    \underset{s = \frac{1}{2} \pm ir_k}{\textup{Res}}L_h(s,\Phi_1,\Phi_2) & = \frac{2\pi^{\frac{1}{2}} h^{\mp ir_k} \cdot\Gamma(\pm ir_k) \Gamma\left(\frac{1}{2} \pm ir_k - i (\mathfrak r_1 - \mathfrak r_2)\right)\cdot c_k(h) \left\langle \phi_k,\overline{\Phi_1} \Phi_2 \right\rangle}
    {\Gamma\left(\frac{1}{4} \pm \frac{1}{2} ir_k\right)^2 \Gamma\left(\frac{1}{4} \pm \frac{1}{2} ir_k - i\mathfrak r_1\right) \Gamma\left(\frac{1}{4} \pm \frac{1}{2} ir_k + i\mathfrak r_2\right)}.
 \end{align*}
\end{theorem}

\begin{remark} The method of this paper could be used to meromorphically continue $L_h(s,\Phi_1,\Phi_2)$ to the whole complex plane, but such an extension would not improve the following two applications. 
\end{remark}

Our main theorem \ref{MainTheorem} allows us to obtain an asymptotic formula for a certain smoothed shifted convolution sum and a bound for the unsmoothed shifted convolution sum. In particular, we have the following results. Let SCS denote: {\it ``shifted convolution sums.''}
\begin{theorem} {\bf (Asymptotic formula for smoothed SCS)}
 \label{AsymptoticFormula}
  Fix $0 < \varepsilon < \tfrac12$. Let $h$ be a positive integer.  Then for $T \rightarrow \infty$, we have
  \begin{align*}
   & \underset{n \neq 0,-h}{\sum_{\sqrt{|n (n + h)|}<T} }
   \hskip-12pt {C_1(n) C_2(n + h)} \Big(\textup{log} \tfrac{T}{\sqrt{|n (n + h)|}}\,\Big)^{\frac{3}{2} + \varepsilon} \hskip-10pt = f_{{\mathfrak r_1,\mathfrak r_2,}h,\varepsilon}(T) T^{\frac{1}{2}} +  \mathcal O\Big(h^{1-\varepsilon}\, T^{\varepsilon} + h^{1+\varepsilon}\, T^{-1 - \varepsilon}\Big).
  \end{align*}
  Here $f_{{\mathfrak r_1,\mathfrak r_2,}h,\varepsilon}{(T)} \ll h^{\theta + \varepsilon}$, and more precisely,

  \begin{align*}
    &f_{\mathfrak r_1,\mathfrak r_2,h,\varepsilon}(T) = 4\pi^{\frac{1}{2}} \Gamma\left(\tfrac{5}{2} + \varepsilon\right) \sum_{k = 1}^{\infty} c_k(h) \left\langle \phi_k,\overline{\Phi_1} \Phi_2 \right\rangle
    \\
    & 
  \hskip 50pt
       \cdot \textup{Re}\left(  \frac{(T/h)^{i r_k}}{\left(\frac12+ir_k\right)^{\frac{5}{2}+\varepsilon}}\;
      \cdot \frac{\Gamma(ir_k) \Gamma\left(\frac{1}{2} + ir_k - i (\mathfrak r_1 - \mathfrak r_2)\right)}{\Gamma\left(\frac{1}{4} + \frac{1}{2} ir_k\right)^2 \Gamma\left(\frac{1}{4} + \frac{1}{2}ir_k - i\mathfrak r_1\right) \Gamma\left(\frac{1}{4} + \frac{1}{2}ir_k + i\mathfrak r_2\right)}\right),
  \end{align*}
  which converges for all $T$ and satisfies $f_{{\mathfrak r_1,\mathfrak r_2,}h,\varepsilon}(T) \ll h^{\theta + \varepsilon}$ where the $\ll$-constant is absolute. 
\end{theorem}
\begin{remark} It seems likely that for any fixed positive integer $h$ and $0<\varepsilon<\tfrac12$, the function $f_{{\mathfrak r_1,\mathfrak r_2,}h,\varepsilon}(T)$  is never identically zero.
\end{remark}

\begin{theorem} {\bf (Upper bound for unsmoothed SCS)} \label{unsmoothedSCS}
 Fix $\,0<\varepsilon<\tfrac12.$ Let $x\to\infty.$  Then for any positive integer $h<x^{\frac12-\varepsilon}$ we have
  \[
    {\sum_{\sqrt{|n (n + h)|} < x} C_1(n) C_2(n + h)}\; \ll \;h^{\frac{2}{3}\theta + \varepsilon}x^{\frac{2}{3} (1 + \theta) + \varepsilon} + h^{\frac{1}{2} + \varepsilon}x^{\frac{1}{2} + 2\theta + \varepsilon}.
  \]
\end{theorem}

\begin{remark}
  The above results are similar to those found by Jutila in \cite{MR1417854} \cite{MR1466405}, which obtains meromorphic continuation and bounds for
  \[
    \sum_{n \geq 1} c(n) c(n + h) (n + h)^{-s}
  \]
  and uses those results to show that
  \[
    \sum_{1 \leq n \leq x} c(n) c(n + h) \ll x^{\frac{2}{3} + \varepsilon}
  \]
  uniformly for $1 \leq h \ll x^{\frac{2}{3}}$. However, Jutila's results only apply to shifted convolutions of the same Maass form, while ours apply to shifted convolutions of different Maass forms as well. Even when the two Maass forms are equal, our results are distinct from Jutila's. In addition to the different forms of the objects under consideration in the statements, our meromorphic continuation and bounds for the Dirichlet series are stronger: Jutila meromorphically continues an approximation of the Dirichlet series to $\sigma > \frac{1}{2}$ and shows that it is less than a constant times $h^{\frac{1}{2} - \sigma + \theta + \varepsilon} |s|^A$ for some positive constant $A$, where the error term resulting from this approximation is less than a constant times $h^{1 - \sigma + \varepsilon} |s|^B$ for some positive constant $B$, while we continue the Dirichlet series to $\sigma > 0$ and bound it in that entire region, with our bounds for $\sigma > \frac{1}{2}$ being stronger than Jutila's.
\end{remark}

\section{\large\bf Basic Notation}
Let $\frak h^2$ denote the upper half plane and $\Gamma = {\rm SL}(2,\mathbb Z)$. Also set $\Gamma_\infty := \left\{ \left(\begin{smallmatrix}1&m\\0&1\end{smallmatrix} \right) \big| \; m\in\mathbb Z  \right\}$ to be the stabilizer of $\infty$ in $\Gamma$.
Let $\phi_1, \phi_2, \phi_3,\ldots$,  denote an orthonormal basis of Maass cusp forms for $\Gamma$ where each $\phi_k$ has Laplace eigenvalue $\lambda_k = \frac{1}{4} + r_k^2$ with $r_k> 0$ and $r_1\le r_2\le r_3\le\cdots.$  For $z =x+iy$ in the upper half plane ($x\in\mathbb R,\; y>0$), each  $\phi_k$ (for $k=1,2,3\ldots$) has a Fourier expansion of the form (see \cite{Goldfeld2006})
  \begin{equation} \label{MaassFormFourierExp}
  \phi_k(z) = \sum_{m \neq 0} c_k(m)y^{\frac{1}{2}}K_{ir_k}(2\pi|m|y) e^{2\pi imx},
\end{equation}
where $c_k(m)\in\mathbb C$ and
$K_{it}(y) = \tfrac12\int\limits_0^\infty e^{-\tfrac12 y\left(u+u^{-1}\right)} u^{it} \;\frac{du}{u}$ (with $t\in\mathbb R, \, y>0$) is the K-Bessel function.

\begin{definition} {\bf (The Picard function $\mathcal F_{b,a}(s)$)}
For $b\in\mathbb R$, $a\ge 0$, and $s\in\mathbb C$ with $\text{\rm Re}(s) > 0$ we define  
\[
  \mathcal F_{b,a}(s) = \int\limits_0^1 \textup{cos}(b \cdot \textup{log}t) t^{\frac{1}{2}s - 1} \left(at + \frac{(1 - t)^2}{2}\right)^{-\frac{1}{2}s} dt.
\]
\end{definition}

\begin{definition} {\bf (The L-function $L_h^\#(s)$)} \label{LhSharp} For $s\in\mathbb C$ with ${\rm Re}(s) >0$, we define
$$L_h^{\#}(s,\phi) = \mathcal F_{r,2}(s) L_h(s,\phi).$$
\end{definition}
 
\begin{definition} {\bf (The Poincar\'e series $P_h(z,s)$)} For $z\in\mathfrak h$, $s\in\mathbb C$ with ${\rm Re}(s)>1$ and $h$ a positive integer, we define the Poincar\'e series
\[
  P_h(z,s) = \sum_{\gamma \in \Gamma_{\infty} \backslash \Gamma} \textup{Im}(\gamma z)^s e^{2\pi ih \cdot \gamma z}.
\]
\end{definition}

\section{\large\bf Overview of the proofs }
For a reason discussed later in this section, it turns out that assuming that the two spectral parameters are the same makes one step significantly easier (in addition to making the notation less cumbersome throughout) while otherwise having no significant effect on the arguments or results. For ease of exposition we thus assume throughout most of this paper that the two Maass forms are equal; the additional complications that arise in the case in which they have different spectral parameters are dealt with in Section \ref{Different}.

Let $\phi$ denote a fixed Maass form with Laplace eigenvalue $\lambda = \tfrac14+r^2.$ Theorem  \ref{MainTheorem} gives the meromorphic continuation and sharp bounds for $L_h(s,\phi)$ in the region ${\rm Re}(s) > 0$ provided the complex variable $s$ is away  (at a distance $>\varepsilon$) from the poles of $L_h(s,\phi)$. Our proof of Theorem \ref{MainTheorem} follows the traditional course, initiated by Selberg \cite{MR0182610}, of considering the inner product 
 $$\Big\langle P_h(*,s),|\phi|^2 \Big\rangle  = \int\limits_{\Gamma \backslash \mathfrak h^2} P_h(z,s) \left|\phi(z)\right|^2\, \frac{dxdy}{y^2}.$$
 This approach is a natural generalization of the Rankin--Selberg method which is based on the inner product
 \begin{equation}\label{RankinSelbergConvolution}
 \Big\langle E(*,s), \, |\phi|^2 \Big\rangle = \frac{\pi^{\frac12-s}}{2^{2+s}} \frac{\Gamma\left( \frac{s}{2}-ir  \right) \Gamma\left( \frac{s}{2}+ir  \right) }{\Gamma\left( \frac{s+1}{2}  \right)  }\cdot L\big(s, \,\phi\otimes \overline{\phi}   \big).
 \end{equation}

 In (\ref{PhPhiSquaredInnerProduct}) we show that
 
\begin{equation} \label{innerProductIntro}
\Big\langle P_h(*,s), \, |\phi|^2 \Big\rangle  =  \sum_{n \neq 0,-h} c(n) c(n + h) \int\limits_0^{\infty} y^{s}\,e^{-2\pi hy}K_{ir}\big(2\pi|n|y\big) K_{ir}\big(2\pi|n + h|y\big) \,\frac{dy}{y}.
\end{equation}
The main difficulty in proving Theorem \ref{MainTheorem} this way is that it is not at all clear how to express the right side of (\ref{innerProductIntro}) as a product of $L_h(s, \phi)$ with a function that has properties similar to a ratio of gamma functions as is the case in (\ref{RankinSelbergConvolution}). 

The key ingredient in our new approach is the appearance of the Picard hypergeometric function 
 $\mathcal F_{b,a}(s)$ (see \S \ref{picard}) where for $r>0$ 
 $$\mathcal F_{r,2}(s) = 2^{\frac{s}{2}-1}\,\frac{\Gamma\left(\frac{s}{2}+ir   \right) \Gamma\left(\frac{s}{2}-ir   \right)}{\Gamma(s)}.$$
 The Picard function $\mathcal F_{r,2}(s)$ amazingly arises after a  long  computation of the integral involving K-Bessel functions on the right side of (\ref{innerProductIntro}).

 In particular, one step in those computations relies on using an identity that we have previously found (restated in this paper as Lemma \ref{identity1}) to rewrite the integral
\begin{equation}\label{CuriousIdentity}
  \int\limits_0^{\infty} y^se^{-2\pi hy} K_{ir}(2\pi |n| y) K_{ir}(2\pi |n + h| y) \frac{dy}{y}
\end{equation}
in a more useful form. That identity requires that the subscripts of the Bessel functions be equal. This is the only part of the proof in which we use the fact that the two Maass forms in the shifted convolution are equal; in Section \ref{Different} we find a more complicated identity for the analogous integral in which the Bessel functions have subscripts $i\mathfrak r_1$ and $i\mathfrak r_2$ and deal with the additional complications introduced by that identity.

 In Proposition \ref{Prop:InnerProductBound} we prove an asymptotic formula of the form
 \begin{equation} \label{AsymptoticFormulaIntro}
 \Big\langle P_h(*,s), \,|\phi|^2 \Big\rangle \sim \frac{\Gamma\left(\frac{s}{2}\right)^2}{ 2^{2+\frac12s}\,\pi^{s} } \mathcal F_{r,2}(s)\cdot L_h(s,\phi)
 \end{equation}
 as $|{\rm Im}(s)|\to\infty.$ This asymptotic formula  allows us to derive properties of $L_h(s, \phi)$
  in a manner analogous to the Rankin--Selberg method except for the additional complication of keeping track of the error terms that arise.
 Our method is capable of wide generalization, but it will only be in certain very special cases (such as the case of Maass forms of level one we are considering here) that the Picard function $\mathcal F_{b,a}(s)$ appearing in the asymptotic formula for the inner product of a Poincar\'e series with the absolute value squared of an automorphic form takes such a simple form. 

\vskip 10pt
The meromorphic continuation and sharp bounds for $L_h(s, \phi)$ are obtained from the spectral expansion of $\big\langle P_h(*,s),|\phi|^2 \big\rangle$ given in \S \ref{spectralexpansion},  \S \ref{spectralsideDiscrete}, and \S
\ref{spectralsideContinuous}. It is shown in \S \ref{spectralsideDiscrete}  that the poles of $L_h(s, \phi)$ come from the discrete spectrum contribution given by
$$\mathcal D(s) := \sum_{k=1}^\infty \Big\langle P_h(*,s),\phi_k\Big\rangle \cdot\big\langle \phi_k, \, |\phi|^2 \big\rangle$$ 
and it is shown in \S
\ref{spectralsideContinuous} that the continuous spectrum contribution
$$ \mathcal C(s) =\frac{1}{4\pi}\int\limits_{-\infty}^\infty \Big\langle P_h(*,s), E\left(*,\tfrac12+iu\right) \Big\rangle \Big\langle E\left(*, \tfrac12+iu\right),\,|\phi|^2\Big\rangle \, du$$
just contributes to the error terms in Theorem \ref{MainTheorem}.
\vskip 10pt
In \S \ref{smoothedsum}, we obtain an asymptotic formula for a smoothed shifted convolution sum by considering the integral
\[
  \frac{1}{2\pi i} \int\limits_{a - iT}^{a + iT} L_h(s,\phi) \frac{T^s}{s^{\frac{5}{2} + \varepsilon}} ds
\]
for $1 + \frac{1}{10}\varepsilon < a < 1 + \frac{1}{5}\varepsilon$ as $T \rightarrow \infty$. The infinite integral equals the smoothed shifted convolution sum, and the ``tail end'' integrals (the difference between the infinite and finite integrals) can be bounded. We then shift the line of integration to $\frac{1}{10}\varepsilon$ and bound the left vertical integral and the horizontal integrals. All of the bounds are straightforward to compute except the bounds for the section of the horizontal integrals near the poles. In this case, an additional complication arises because of the poles at $\frac{1}{2} \pm ir_k$, as the $r_k$ become denser as $k \rightarrow \infty$. (In particular, as $T \rightarrow \infty$, the number of $r_k$ that are in $[T,T + 1]$ is approximately a constant times $T$.) Addressing this issue occupies a significant portion of \S \ref{smoothedsum} and in particular imposes a constraint on the permissible values of $T$. We also compute the residues at all of the poles $\frac{1}{2} \pm ir_k$ and show that the sum of the residues converges. The sum of the residues yields the main term, while the bounds for the ``tail end'' integrals, the left vertical integral, and the horizontal integrals, are all smaller, yielding the error term.

\vskip 10pt
In \S \ref{unsmoothedsum}, we obtain an upper bound for the unsmoothed shifted convolution sum by considering the integral
\[
  \frac{1}{2\pi i} \int\limits_{1 + \varepsilon - iT}^{1 + \varepsilon + iT} L_h(s,\phi) \frac{x^s}{s} ds
\]
as $T\rightarrow \infty$ and $x\rightarrow \infty$. This is done similarly to the preceding case, but with some additional complications. The infinite integral equals the unsmoothed shifted convolution sum. The ``tail end'' can still be bounded, but requires a more complicated argument because the denominator of the integrand is just $s$ and in particular imposes a constraint on the permissible values of $x$. We then shift the line of integration to $\varepsilon$ and bound the left vertical integral and the horizontal integrals, as well as the sum of the residues, using similar techniques to the preceding case. It is shown that setting $T = h^{-\left(\frac{2}{3}\theta + \varepsilon\right)} x^{\frac{1}{3} + \frac{4}{3}\theta}$ yields the claimed bound for the unsmoothed shifted convolution sum.

\section{\large\bf Picard's hypergeometric function}\label{picard}

We now define the Picard hypergeometric function and derive some of its  useful properties. Let $(x)_n =\prod\limits_{k=0}^{n-1} (x-k)$ denote the Pochhammer function.

\'Emile Picard \cite{MR1508705} showed that the Appell $F_1$ hypergeometric function defined for $|z_1|<1, |z_2|<1$ by
$$F_1(\alpha,\beta_1,\beta_2,c; z_1,z_2) := \sum_{m=0}^\infty\sum_{n=0}^\infty \frac{(\alpha)_{m+n} (\beta_1)_m (\beta_2)_n}{(c)_{m+n} m!n!}z_1^m z_2^n,$$ 
 can be written as a one-dimensional Euler type integral:
$$F_1(\alpha,\beta_1,\beta_2,c; z_1,z_2) = \frac{\Gamma(c)}{\Gamma(\alpha)\Gamma(c-\alpha)}\int\limits_0^1 t^{\alpha-1} (1-t)^{c-\alpha-1} (1-z_1t)^{-\beta_1}(1-z_2t)^{-\beta_2}\; dt,$$
 which is valid for $\text{Re}(c) > \text{Re}(\alpha) > 0.$ 
 The proof of Theorem \ref{MainTheorem}  utilizes the following special case of Picard's hypergeometric function. 
\begin{definition} {\bf (The function $\mathcal F_{b,a}(s)$)}
For $b\in\mathbb R$, $a\ge 0$, and $s\in\mathbb C$ with $\text{\rm Re}(s) > 0$ we define  a special case of Picard's hypergeometric function (without the ratio of gamma factors) given by 
\[
  \mathcal F_{b,a}(s) = \int\limits_0^1 \textup{cos}(b \cdot \textup{log}t) t^{\frac{1}{2}s - 1} \left(at + \frac{(1 - t)^2}{2}\right)^{-\frac{1}{2}s} dt.
\]
\end{definition}
 
 The proof of the meromorphic continuation of the shifted convolution L-function $L_h(s,\phi)$ makes use of the meromorphic continuation of  $\mathcal F_{b,0}(s)$ and $\mathcal F_{b,2}(s)$ where the variable $b$ takes the values $r_j$ and $\tfrac14+r_j^2$ is a Laplace eigenvalue of the Maass cusp form $\phi_j.$  
 
 We recall Stirling's asymptotic formula for the gamma function which will be used to obtain bounds for $\mathcal F_{b,a}(s)$.
 
 \begin{proposition} {\bf (Stirling's Asymptotic Formula)} \label{Stirling} Let $\sigma\in\mathbb R$ be fixed. Then for $t\in \mathbb R$ with $|t|\to\infty$ we have
 $$\big|\Gamma(\sigma+it)\big|\sim \sqrt{2\pi}\cdot \lvert t \rvert^{\sigma-\frac12}\, e^{-\frac{\pi}{2}\lvert t\rvert}.$$
\end{proposition}
 
 \begin{proposition} \label{F_{b,2}(s)}
  Let $b\in\mathbb R$, and $s\in\mathbb C$ with $\text{\rm Re}(s) > 0$. Then we have the identity  \[
    \mathcal F_{b,2}(s)  = 2^{\frac{1}{2}s - 1}\frac{\Gamma\left(\frac{1}{2}s + ib\right) \Gamma\left(\frac{1}{2}s - ib\right)}{\Gamma(s)}.
  \]
  Furthermore, for $b$ and $\text{\rm Re}(s)$ fixed we have the asymptotic formula 
  $$\big|\mathcal F_{b,2}(s)\big| \sim C \big(1+|s|\big)^{-\frac12}, \qquad \big(|\text{\rm Im}(s)|\to\infty\big), $$
  where the constant $C$ depends at most on $b$ and  $\text{\rm Re}(s).$
\end{proposition}

\begin{proof}
  Note that $\mathcal F_{b,2}(s)= 2^{\frac12 s}\int\limits_0^1 \textup{cos}(b \cdot \textup{log}\,t) \left(t^{\frac{1}{2}} + t^{-\frac{1}{2}}\right)^{-s} \frac{dt}{t}.$
  
    By the change of variables $t \mapsto t^{-1}$, we have
  \[
     \int\limits_0^1 \textup{cos}(b \cdot \textup{log}\,t) \left(t^{\frac{1}{2}} + t^{-\frac{1}{2}}\right)^{-s} \frac{dt}{t} =
     \int\limits_1^{\infty} \textup{cos}(b \cdot \textup{log}\,t) \left(t^{\frac{1}{2}} + t^{-\frac{1}{2}}\right)^{-s} \frac{dt}{t} 
  \]
  and
  \[\int\limits_0^1 \textup{sin}(b \cdot \textup{log}t) \left(t^{\frac{1}{2}} + t^{-\frac{1}{2}}\right)^{-s} \frac{dt}{t} =
    -\int\limits_1^{\infty} \textup{sin}(b \cdot \textup{log}\,t) \left(t^{\frac{1}{2}} + t^{-\frac{1}{2}}\right)^{-s} \frac{dt}{t}.  
  \]
  so
  \[
    \int\limits_0^1 \textup{cos}(b \cdot \textup{log}\,t) \left(t^{\frac{1}{2}} + t^{-\frac{1}{2}}\right)^{-s} \frac{dt}{t} = \frac{1}{2} \int\limits_0^{\infty} \textup{cos}(b \cdot \textup{log}t) \left(t^{\frac{1}{2}} + t^{-\frac{1}{2}}\right)^{-s} \frac{dt}{t}.
  \]
  Similarly
  $
    \int\limits_0^{\infty} \textup{sin}(b \cdot \textup{log}\,t) \left(t^{\frac{1}{2}} + t^{-\frac{1}{2}}\right)^{-s} \frac{dt}{t} = 0.
  $ It follows that
\begin{align*}\mathcal F_{b,2}(s) = 2^{\frac12s-1}\int_0^\infty t^{ib} \left( t^{\frac{1}{2}} + t^{-\frac{1}{2}} \right)^{-s} \; \frac{dt}{t} = 2^{\frac12s-1}\frac{\Gamma\left(\frac{1}{2}s + ib\right) \Gamma\left(\frac{1}{2}s - ib\right)}{\Gamma(s)}.
\end{align*}
  since the last integral is just a beta function. 
  
  The bound for $\mathcal F_{b,2}(s)$ follows immediately from Stirling's asymptotic formula given in Proposition \ref{Stirling}.
\end{proof}

\begin{proposition}
  For any $b \geq 0$ and $s\in\mathbb C$ with $\text{\rm Re}(s) >0,$ we have the identity
  \[
    \mathcal F_{b,0}(s) =  \frac{\textup{cos}(i\pi b)}{\textup{cos}\left(\frac{\pi}{2}s\right)} \mathcal F_{b,2}(s).
  \]
\end{proposition}

\begin{proof} Note that 
$\mathcal F_{b,0}(s) = 2^{\frac12s} \int\limits_0^1 \textup{cos}(b \cdot \textup{log}\,t)\, t^{\frac12 s-1}\big(1-t\big)^{-s} \,dt.$
  We use the identities
  \[
    \int\limits_0^1 t^{ib}\, t^{\frac{1}{2}s - 1} (1 - t)^{-s} dt + \int\limits_0^1 t^{-ib}t^{\frac{1}{2}s - 1} (1 - t)^{-s} dt = 2 \int\limits_0^1 \textup{cos}(b \cdot \textup{log}\,t) t^{\frac{1}{2}s - 1} (1 - t)^{-s}\, dt,
  \]
  \[
    \int\limits_0^1 t^{\frac{1}{2}s \pm ib - 1} (1 - t)^{-s} dt = \frac{\Gamma\left(\frac{1}{2}s \pm ib\right) \Gamma(-s + 1)}{\Gamma\left(-\frac{1}{2}s \pm ib + 1\right)},
  \]
   together with the identity
  $
    \Gamma(z) \Gamma(1 - z) = \frac{\pi}{\textup{sin}(\pi z)}
  $
  to show that

\begin{align*}
 \mathcal F_{b,0}(s) & =   2^{\frac{1}{2}s} \int\limits_0^1 \textup{cos}(b \cdot \textup{log}\,t)\, t^{\frac{1}{2}s - 1} (1 - t)^{-s}\, dt \\
   &
   \hskip-15pt                                                                                                                                                                                                                                                                        = 2^{\frac12 s-1}\frac{\Gamma(-s + 1) \left(\Gamma\left(\frac{1}{2}s + ib\right) \Gamma\left(-\frac{1}{2}s - ib + 1\right) + \Gamma\left(\frac{1}{2}s - ib\right) \Gamma\left(-\frac{1}{2}s + ib + 1\right)\right)}{\Gamma\left(-\frac{1}{2}s + ib + 1\right) \Gamma\left(-\frac{1}{2}s - ib + 1\right)} \\                                                                                            &
 \hskip-15pt                                                                                                 = 2^{\frac12 s-1}\frac{\frac{\pi}{\Gamma(s) \textup{sin}(\pi s)} \left(\frac{\pi}{\textup{sin}\left(\pi \left(\frac{1}{2}s + ib\right)\right)} + \frac{\pi}{\textup{sin}\left(\pi \left(\frac{1}{2}s - ib\right)\right)}\right)}{\frac{\pi}{\Gamma\left(\frac{1}{2}s - ib\right) \textup{sin}\left(\pi \left(\frac{1}{2}s - ib\right)\right)} \frac{\pi}{\Gamma\left(\frac{1}{2}s + ib\right) \textup{sin}\left(\pi \left(\frac{1}{2}s + ib\right)\right)}} \\
                                                                                             &\\
                                                                                             &
 \hskip-15pt                                                                                                 = 2^{\frac12 s-1}\frac{\Gamma\left(\frac{1}{2}s + ib\right) \Gamma\left(\frac{1}{2}s - ib\right)}{\Gamma(s)} \frac{\textup{sin}\left(\pi \left(\frac{1}{2}s + ib\right)\right) + \textup{sin}\left(\pi \left(\frac{1}{2}s - ib\right)\right)}{ \textup{sin}(\pi s)} \\
                                                                                             &
 \hskip-15pt                                                                                                 = 2^{\frac12 s-1}\frac{\Gamma\left(\frac{1}{2}s + ib\right) \Gamma\left(\frac{1}{2}s - ib\right)}{\Gamma(s)} \frac{2\,\textup{sin}\left(\frac{\pi}{2}s\right) \textup{cos}(i\pi b)}{\textup{sin}(\pi s)} \\
                                                                                             &
 \hskip-15pt                                                                                                 = 2^{\frac12 s-1}\frac{\Gamma\left(\frac{1}{2}s + ib\right) \Gamma\left(\frac{1}{2}s - ib\right)}{\Gamma(s)} \frac{\textup{cos}(i\pi b)}{\textup{cos}\left(\frac{\pi}{2}s\right)} \; = \; \frac{\textup{cos}(i\pi b)}{\textup{cos}\left(\frac{\pi}{2}s\right)} \mathcal F_{b,2}(s).
  \end{align*}
\end{proof}

The following integral representation is also useful.

\begin{proposition} \label{PropFba}
  For any $b \geq 0$, $a \geq 0$, and $s \in \mathbb C$ with $\textup{Re}(s) > 0$, we have the identity
  \[
    \mathcal F_{b,a}(s) = \int\limits_0^{\infty} \textup{cos}(bu)\cdot \Big(a - 1 + \textup{cosh}(u)\Big)^{-\frac{1}{2}s}\, du.
  \]
\end{proposition}

\begin{proof}
  We rewrite
  \begin{align*}
    \int\limits_0^{\infty} \textup{cos}(bu) \cdot \Big(a - 1 + \textup{cosh}(u)\Big)^{-\frac{1}{2}s} du &= \int\limits_0^{\infty} \textup{cos}(bu) \left(a - 1 + \frac{e^u + e^{-u}}{2}\right)^{-\frac{1}{2}s} du \\
                                                                                            &= \int\limits_0^{\infty} \textup{cos}(bu) e^{-\frac{1}{2}su} \left(ae^{-u} + \frac{\left(1 - e^{-u}\right)^2}{2}\right)^{-\frac{1}{2}s} du.
  \end{align*}
  Using the change of variables $t = e^{-u}$ yields the result.
\end{proof}

\section{\large\bf The spectral expansion of $\big\langle P_h(*,s),|\phi|^2 \big\rangle$}\label{spectralexpansion}

We shall show  that the meromorphic continuation and growth in $|s|$ of the convolution L-function $L_h(s,\phi)$ for $\text{\rm Re}(s)>0$ can be obtained from the inner product of  $|\phi|^2$ with the Poincar\'e series
\[
  P_h(z,s) = \sum_{\gamma \in \Gamma_{\infty} \backslash \Gamma} \textup{Im}(\gamma z)^s e^{2\pi ih \cdot \gamma z},
\]
where $s\in\mathbb C$ with $\text{\rm Re}(s)>1$, $z$ is in the upper half-plane, and $h$ is a fixed positive integer.
Our first goal is to obtain the meromorphic continuation and a sharp bound for the inner product $\big\langle P_h(*,s),|\phi|^2 \big\rangle$ in the range ${\rm Re}(s)>0.$ This is given in the following theorem.

\begin{theorem}{\bf (Meromorphic continuation and bound for $\left\langle P_h(*,s),|\phi|^2 \right\rangle$)} \label{Thm:InnerProdBound}
 Fix $\varepsilon > 0$. The inner product $\left\langle P_h(*,s),|\phi|^2 \right\rangle$, which is defined for $\textup{Re}(s) > 1$, has meromorphic continuation to $\textup{Re}(s) > \varepsilon$, with possible simple poles at $s = \frac{1}{2} \pm ir_k$ and no other poles in that region.
 
 Let $s=\sigma+it$ with $\sigma>\varepsilon$ and $|t|\to\infty.$ Then for $|s-\rho_k|>\varepsilon$ (for all poles $\rho_k$), we have the  bound
  \[
    \Big\langle P_h(*,s),|\phi|^2 \Big\rangle 
    \ll
     (1 + |t|)^{\textup{max}\left(\sigma - \frac{1}{2},0\right)} e^{-\frac{\pi}{2}|t|}.
  \]
\end{theorem}
\begin{proof} The proof follows from the spectral expansion (\ref{SpectralExpansion}) and the discrete and continuous spectrum bounds given in Propositions \ref{DiscreteSpectrumBound}, \ref{ContinuousSpectrumBound}, respectively.
\end{proof}

For $s\in\mathbb C$ with $\text{\rm Re}(s)>1$ the inner product is given by
\begin{align*}
  \Big\langle P_h(*,s),|\phi|^2 \Big\rangle & = \int\limits_{\Gamma \backslash \mathfrak h^2} P_h(z,s) \overline{\phi(z)} \phi(z)\, \frac{dxdy}{y^2}  = \int\limits_{\Gamma_{\infty} \backslash \mathfrak h^2} y^se^{-2\pi hy}e^{2\pi ihx} \overline{\phi(z)} \phi(z)\, \frac{dxdy}{y^2}\end{align*}
  after unraveling the Poincar\'e series. It follows that
  \begin{align}
  \Big\langle P_h(*,s),|\phi|^2 \Big\rangle & = {\int\limits_0^{\infty} \int\limits_0^1} y^se^{-2\pi hy}e^{2\pi ihx} \left(\sum_{m \neq 0} c(m) y^{\frac{1}{2}}K_{ir}\big(2\pi|m|y\big) e^{-2\pi imx}\right)\nonumber \\ 
  & 
  \hskip 130pt \cdot \left(\sum_{n \neq 0} c(n) y^{\frac{1}{2}}K_{ir}\big(2\pi|n|y\big) e^{2\pi inx}\right) \frac{dxdy}{y^2}
  \label{PhPhiSquaredInnerProduct}
   \\
  &
  = \sum_{n \neq 0,-h} c(n) c(n + h) \int\limits_0^{\infty} y^{s }e^{-2\pi hy}K_{ir}\big(2\pi|n|y\big) K_{ir}\big(2\pi|n + h|y\big) \,\frac{dy}{y}.
\nonumber\end{align}
It is not at all clear how to extract  the term     $ |n(n+h)|^{-s} $ from the above integral of K-Bessel functions which is necessary to have $L_h(s,\phi)$ appear as a factor. We shall return to this topic in \S \ref{RelatingL_h(s)}. For the moment we shall construct the spectral expansion of the inner product $\left\langle P_h(*,s),|\phi|^2 \right\rangle$.

The Selberg spectral expansion for $\text{\rm SL}(2,\mathbb Z)$ (see \cite{Goldfeld2006}), \S 3.16) states that for $s\in\mathbb C$ with $\text{\rm Re}(s)>1$ we have
\begin{align} \label{SpectralExpansion} P_h(z,s) = \sum_{k=1}^\infty \Big\langle P_h(*,s),\phi_k\Big\rangle \,\phi_k(z) + \frac{1}{4\pi}\int\limits_{-\infty}^\infty \Big\langle P_h(*,s), E(*,\tfrac12+iu \Big\rangle  E(z, \tfrac12+iu)\, du
\end{align}
where the Eisenstein series
$
E(z,s)  := \frac12\sum\limits_{\gamma\in \Gamma_\infty\backslash \Gamma } \text{\rm Im}(\gamma z)^s$
has the Fourier expansion
\begin{align*}E(z,s) & \;=\;  y^s + \sqrt{\pi} \frac{\Gamma(s-\tfrac12)\zeta(2s-1)}{\Gamma(s) \zeta(2s)}y^{1-s} \; + \; \underset{\ell\ne 0}{\sum_{\ell\in\mathbb Z}}  c(\ell,s)y^{\frac12} K_{s-\frac12}\big(2\pi i|\ell | y\big) e^{2\pi i\ell x}.
\end{align*}
with non-constant Fourier coefficients given by
$$c(\ell,s) = \frac{2\pi^s \sigma_{1-2s}(\ell)\cdot |\ell|^{s-\frac12}}{\Gamma(s)\zeta(2s)}$$ for $\ell\ne 0.$
\vskip 5pt

It follows that
\begin{align}\label{SpectralExpansion}
\Big\langle P_h(*,s),|\phi|^2 \Big\rangle & = 
\sum_{k=1}^\infty \Big\langle P_h(*,s),\phi_k\Big\rangle \cdot\Big\langle \phi_k, \, |\phi|^2 \Big\rangle \\
&
\hskip -48pt
 +  \frac{1}{4\pi}\int\limits_{-\infty}^\infty \Big\langle P_h(*,s), E\left(*,\tfrac12+iu\right) \Big\rangle \Big\langle E\left(*, \tfrac12+iu\right),\,|\phi|^2\Big\rangle \, du.\nonumber
\end{align}
The spectral expansion (\ref{SpectralExpansion}) can be computed in two different ways. The first way (termed {\it ``the spectral side"}) is done by computing all the individual inner products appearing on the right hand side of (\ref{SpectralExpansion}). The second way (termed {\it ``the geometric side"}) is to just directly compute the inner product on left hand side.

\section{\large\bf The spectral side (discrete spectrum)}\label{spectralsideDiscrete}

To evaluate the spectral side (discrete spectrum) we begin with a lemma on the growth of the Fourier coefficients of the orthonormal basis of Maass forms $\phi_k$ with Laplace eigenvalue $\lambda_k = \tfrac14+r_k^2\,$ for $k=1,2,3,\ldots$
  \begin{lemma}{\bf (Bounds for Fourier coefficients of Maass forms)} \label{c_k(h)Bound} Let $c_k(h)$ denote the $h^{th}$ Fourier coefficient of the Maass form $\phi_k$.   Then
  $$c_k(h) \ll h^{\theta+\varepsilon}\,\log(r_k)^{\frac12}\cdot e^{\frac{\pi}{2}r_k}$$ for $r_k\to\infty.$ Here $\theta$ is the best progress toward the Ramanujan--Petersson conjecture for Maass forms for $\textup{SL}(2,\mathbb Z)$.
  \end{lemma}
  \begin{proof}
  It is proved in \cite{MR4271357} that for $h\ge 1$, we have $c_k(h) = c_k(1)\cdot \lambda_k(h)$ where  \[
    |c_k(1)|^2 = \frac{1}{L(1,\textup{Ad}(\phi_k)) \Gamma\left(\frac{1}{2} + ir_k\right) \Gamma\left(\frac{1}{2} - ir_k\right)} = \frac{\textup{cosh}(\pi r_k)}{\pi L(1,\textup{Ad}(\phi_k))}
  \]
 and $\lambda_k(h)$ is the $h^{th}$ Hecke eigenvalue of the Maass form $\phi_k$ which satisfies the bound $|\lambda_k(h)|\ll h^{\theta+\varepsilon}$.
 
 The proof of the lemma then follows from  Stirling's asymptotic formula   
(Proposition \ref{Stirling}) and the bound $L(1,\textup{Ad}(\phi_k))\gg \log(r_k)^{-1}$  proved in the appendix of \cite{HL1994}.\end{proof}
\begin{proposition} \label{SpectralInnerProductDiscrete}
{\bf (Bound for the inner product $\langle P_h(*,s),\phi_k  \rangle$)}
Let $s=\sigma+it$ with $\sigma>0$ and $t\in\mathbb R.$
Then
$$\Big\langle P_h(*,s),\phi_k \Big \rangle = \frac{2\pi h^\frac12\, c_k(h)}{(4\pi h)^s}\, \frac{\Gamma\left(s - \frac{1}{2} + ir_k\right) \Gamma\left(s - \frac{1}{2} - ir_k\right)}{\Gamma(s)},$$
 \begin{align*} \Big\langle P_h(*,s),\phi_k \Big \rangle 
  &\ll h^{\frac12 - \sigma +\theta+\varepsilon}\log(r_k)^{\frac12} \frac{ \left(\big(1 + |t - r_k|\big)  \big(1 + |t + r_k|\big)\right)^{\sigma - 1}}{  \big(1 + |t|\big)^{\sigma - \frac12} }
   e^{-\frac{\pi}{2} \big(|t - r_k| + |t +r_k| - |t|-r_k\big)}.
\end{align*}

\end{proposition}

\begin{proof}
    For $\text{\rm Re}(s)>1$ we compute with Rankin-Selberg unfolding:
  \begin{align*}
   \Big \langle P_h(*,s),\phi_k \Big\rangle & = \int\limits_{\Gamma \backslash \mathfrak h^2} \sum_{\gamma \in \Gamma_{\infty} \backslash \Gamma} (\textup{Im}(\gamma z))^s e^{2\pi ih \cdot \gamma z} \sum_{n \neq 0} \overline{c_k(n) y^{\frac{1}{2}}K_{ir_k}(2\pi|n|y) e^{2\pi inx}}\; \frac{dxdy}{y^2} \\
                                 & = \int\limits_{\Gamma_{\infty} \backslash \mathfrak h^2} y^{s}e^{-2\pi hy}e^{2\pi ihx} \sum_{n \neq 0} c_k(n) y^{\frac{1}{2}}K_{ir_k}(2\pi|n|y) e^{-2\pi inx} \;\frac{dxdy}{y^2} \\
                                 & = \sum_{n \neq 0} c_k(n) \int\limits_0^1 e^{2\pi i (h - n) x} dx \int\limits_0^{\infty} y^{s - \frac{3}{2}}e^{-2\pi hy}K_{ir_k}(2\pi|n|y) dy \\
                                 & = c_k(h) \int\limits_0^{\infty} y^{s - \frac{3}{2}}e^{-2\pi hy}K_{ir_k}(2\pi hy) dy \\
                                 & = \frac{2\pi h^\frac12\, c_k(h)}{(4\pi h)^s}\, \frac{\Gamma\left(s - \frac{1}{2} + ir_k\right) \Gamma\left(s - \frac{1}{2} - ir_k\right)}{\Gamma(s)}.
  \end{align*}

  The bound for the inner product  follows from Stirling's asymptotic formula  for the gamma function  given in Proposition \ref{Stirling} together with the bound given in Lemma \ref{c_k(h)Bound}.
\end{proof}

 A crucial ingredient in obtaining sharp bounds for the discrete spectrum contribution is the following proposition  first proved by Bernstein and Reznikov. 
 \begin{proposition}\label{Prop:B-R} Define  $b_k := \big\langle \phi_k,|\phi|^2 \big\rangle^2 \cdot e^{\pi r_k}.$ Then for all $T\gg 1$ the following bound holds:
  \[
    \sum_{T < r_k < 2T}  b_k \; \ll1.
  \]
  Since $\mathfrak r_1 >1$ we can obtain (with a dyadic decomposition) that 
  $\sum\limits_{r_k\le T} b_k \ll \log T.$
  \end{proposition}
  \begin{proof}
  See \cite{BR1999}.
  \end{proof}
   This implies that $b_k$ is bounded for $\phi$  fixed and $k \rightarrow \infty$, so $\big\langle \phi_k,|\phi|^2 \big\rangle \ll e^{-\frac{\pi}{2}r_k}$. It is conjectured that $b_k \ll r_k^{-2 + \varepsilon}$; if so, then  $\big\langle \phi_k,|\phi|^2 \big\rangle \ll r_k^{-1 + \varepsilon}e^{-\frac{\pi}{2}r_k}$  (see Remark 1 of \cite{BR1999}). This conjecture is consistent with the Lindel\"of hypothesis.

\begin{proposition} {\bf (Discrete spectrum bound)} \label{DiscreteSpectrumBound}
Fix $\varepsilon > 0$. Let 
  $$\mathcal D(s) := \sum_{k=1}^\infty \Big\langle P_h(*,s),\phi_k\Big\rangle \cdot\big\langle \phi_k, \, |\phi|^2 \big\rangle$$ 
  denote the discrete part of the spectral expansion (\ref{SpectralExpansion})
  which  is defined for $\textup{Re}(s) > 1$. Then $\mathcal D(s)$ has meromorphic continuation to $\textup{Re}(s) > \varepsilon$, with possible simple poles at $s = \frac{1}{2} \pm ir_k$ and no other poles in that region.
  
  Let $s=\sigma+it$.  For $\varepsilon<\sigma$ fixed with $|s-\rho_k| >\varepsilon$ (for all poles $\rho_k$) and $|t| \rightarrow \infty$ we have  the  bound
  \[
    \mathcal D(s) \ll \; h^{\frac12-\sigma+\theta+\varepsilon} \big(1 + |t|\big)^{\textup{max}\left(\sigma - \frac{1}{2},0\right)} \textup{log}(1 + |t|) e^{-\frac{\pi}{2}|t|}.
  \]
\end{proposition}

\begin{proof}
  By Proposition \ref{SpectralInnerProductDiscrete} we have
  \begin{align*} \Big\langle P_h(*,s),\phi_k \Big \rangle 
  &\ll h^{\frac12-\sigma+\theta+\varepsilon}|c_k(1)|^{\frac12} \frac{ \left(\big(1 + |t - r_k|\big)  \big(1 + |t + r_k|\big)\right)^{\sigma - 1}}{  \big(1 + |t|\big)^{\sigma - \frac12} }
   e^{-\frac{\pi}{2} \big(|t - r_k| + |t +r_k| - |t|-r_k\big)}.
\end{align*}

  Next, consider $t \rightarrow \infty$ and assume $t > 0$; the computations for $t \rightarrow -\infty$ are analogous.
We split the sum for $\mathcal D(s)$ 
     into the sums over $r_k \le t$ and $r_k > t$ (with a dyadic decomposition) and then apply Proposition  \ref{Prop:B-R}.

     We have
     \[
       \sum_{|t - r_k| \leq \Delta} \left|c_k(1) e^{-\frac{\pi}{2}r_k}\right|^2 \ll t\Delta \textup{log} t
     \]
     and
     \[
       \sum_{|t - r_k| \leq 1} \left|\left\langle \phi_k,|\phi|^2 \right\rangle e^{\frac{\pi}{2}r_k}\right|^2 \ll 1 \textup{ and } \sum_{A \leq |t - r_k| \leq 2A} \left|\left\langle \phi_k,|\phi|^2 \right\rangle e^{\frac{\pi}{2}r_k}\right|^2 \ll 1.
     \]
     Thus, by the Cauchy--Schwarz inequality,
     \begin{align*}
       \sum_{0 \leq t - r_k \leq 1} c_k(1) \left\langle \phi_k,|\phi|^2 \right\rangle & \leq \left(\sum_{0 \leq t - r_k \leq 1} \left|c_k(1) e^{-\frac{\pi}{2}r_k}\right|^2\right)^{\frac{1}{2}} \left(\sum_{0 \leq t - r_k \leq 1} \left|\left\langle \phi_k,|\phi|^2\right\rangle e^{\frac{\pi}{2}r_k}\right|^2\right)^{\frac{1}{2}} \\
  &\\                                                                                    & \ll t^{\frac{1}{2}} \textup{log} t
     \end{align*}
     and
     \begin{align*}
       \sum_{A \leq t - r_k \leq 2A} c_k(1) \left\langle \phi_k,|\phi|^2 \right\rangle & \leq \left(\sum_{A \leq t - r_k \leq 2A} \left|c_k(1) e^{-\frac{\pi}{2}r_k}\right|^2\right)^{\frac{1}{2}} \left(\sum_{0 \leq t - r_k \leq 1} \left|\left\langle \phi_k,|\phi|^2\right\rangle e^{\frac{\pi}{2}r_k}\right|^2\right)^{\frac{1}{2}} \\
 &\\                                                                                      & \ll A^{\frac{1}{2}}t^{\frac{1}{2}} \textup{log} t.
     \end{align*}
    
 For $r_k \le t$, we thus have
 \begin{align*}
  & \sum_{0 \leq t - r_k \leq 1} \langle P_h(*,s),\phi \rangle \left\langle \phi_k,|\phi|^2 \right\rangle \;  \ll \; h^{\frac{1}{2} - \sigma + \theta + \varepsilon} (1 + t)^{-\sigma + \frac{1}{2}} e^{-\frac{\pi}{2}t}  \cdot \sum_{0 \leq t - r_k \leq 1} c_k(1)\\
   &
   \hskip 220pt \cdot \left\langle \phi_k,|\phi|^2\right\rangle (1 + t - r_k)^{\sigma - 1} (1 + t + r_k)^{\sigma - 1} \\
                                                                                                         &                                                                                                         \hskip 70pt \leq \; h^{\frac{1}{2} - \sigma + \theta + \varepsilon} (1 + t)^{-\frac{1}{2}} e^{-\frac{\pi}{2}t} \sum_{0 \leq t - r_k \leq 1} c_k(1) \left\langle \phi_k,|\phi|^2\right\rangle \\
                                                                                                         &
                                                                                                         \hskip 70pt \ll\; h^{\frac{1}{2} - \sigma + \theta + \varepsilon} \textup{log}(1 + t) e^{-\frac{\pi}{2}t}.
 \end{align*}
 Similarly
 \[
   \sum_{A \leq t - r_k \leq 2A} \langle P_h(*,s),\phi \rangle \left\langle \phi_k,|\phi|^2 \right\rangle \ll h^{\frac{1}{2} - \sigma + \theta + \varepsilon}A^{\sigma - \frac{1}{2}} \textup{log}(1 + t) e^{-\frac{\pi}{2}t},
 \]
 so
 \begin{align*}
   \sum_{t - r_k > 1} \langle P_h(*,s),\phi_k \rangle \left\langle \phi_k,|\phi|^2 \right\rangle & = \sum_{\ell = 0}^{\lfloor \textup{log}_2 t \rfloor} \sum_{2^{\ell} < t - r_k \leq 2^{\ell + 1}} \langle P_h(*,s),\phi_k \rangle \left\langle \phi_k,|\phi|^2 \right\rangle \\
                                                                                                 & \ll h^{\frac{1}{2} - \sigma + \theta + \varepsilon}e^{-\frac{\pi}{2}t} \sum_{\ell = 0}^{\lfloor \textup{log}_2 t \rfloor} 2^{\ell \left(\sigma - \frac{1}{2}\right)} \\
                                                                                                 & \ll h^{\frac{1}{2} - \sigma + \theta + \varepsilon} (1 + t)^{\textup{max}\left(\sigma - \frac{1}{2},0\right)} \textup{log}(1 + t) e^{-\frac{\pi}{2}t}.
 \end{align*}

\noindent
  Next, consider the case $r_k > t$. Again by Cauchy--Schwarz and Proposition \ref{Prop:B-R} we get 
  
  \pagebreak
  
   \begin{align*}\sum_{r_k\le t} \Big\langle P_h(*,s),\phi_k\Big\rangle \cdot\big\langle \phi_k, \, |\phi|^2 \big\rangle 
  & \ll h^{\frac12-\sigma+\theta+\varepsilon} \;\sum_{\ell = 0}^{\infty} \sum_{2^{\ell}t < r_k < 2^{\ell + 1}t} 
     \hskip-4pt\left|\left\langle \phi_k,|\phi|^2 \right\rangle\right|  \; e^{-\frac{\pi}{2} (r_k - t)}
     \\
     &
     \hskip 60pt\cdot
     \frac{(\textup{log}(1 + r_k))^{\frac{1}{2}} ((1 + r_k - t) (1 + t + r_k))^{\sigma - 1}}{(1 + t)^{\sigma - \frac{1}{2}}}\\
   & 
   \hskip-100pt
    \ll \; h^{\frac12-\sigma+\theta+\varepsilon}\;\frac{(\textup{log}(1 + t))^{\frac{1}{2}}}{(1 + t)^{\sigma - \frac{1}{2}}}  \sum_{\ell = 0}^{\infty} \left(\sum_{2^{\ell}t < r_k < 2^{\ell + 1}t} \left|\left\langle \phi_k,|\phi|^2 \right\rangle\right|^2 e^{\pi r_k}\right)^\frac12 
    \\
    &
    \hskip-10pt\cdot \left(\sum_{2^{\ell}t < r_k < 2^{\ell + 1}t} ((1 + r_k - t) (1 + t + r_k))^{2\sigma - 2} e^{-\pi (2r_k - t)}\right)^\frac12\\
    &
    \hskip-100pt
     \ll \; h^{\frac12-\sigma+\theta+\varepsilon}\;\frac{(\textup{log}(1 + t))^{\frac{1}{2}}}{(1 + t)^{\sigma - \frac{1}{2}}} 
     \sum_{\ell = 0}^{\infty} \sum_{2^{\ell}t < r_k < 2^{\ell + 1}t} ((1 + r_k - t) (1 + t + r_k))^{\sigma - 1} e^{-\frac{\pi}{2} (2r_k - t)} \\
    &
    \hskip -100pt
      \ll  h^{\frac12-\sigma+\theta+\varepsilon}\; \frac{(\textup{log}(1 + t))^{\frac{1}{2}}}{(1 + t)^{\sigma - \frac{1}{2}}} 
 e^{\frac{\pi}{2}t} ((1 + r_{k_0(t)} - t) (1 + t + r_{k_0(t)}))^{\sigma - 1} e^{-\pi r_{k_0(t)}} 
    \\
    &
    \hskip -100pt
    \leq h^{\frac12-\sigma+\theta+\varepsilon}\; \frac{(\textup{log}(1 + t))^{\frac{1}{2}}}{(1 + t)^{\sigma - \frac{1}{2}}} 
 e^{\frac{\pi}{2}t} (1 + 2t)^{\sigma - 1} e^{-\pi t} 
    \\
    &
    \hskip-100pt
    \ll h^{\frac12-\sigma+\theta+\varepsilon}\;(\textup{log}(1 + t))^{\frac{1}{2}} (1 + t)^{-\frac{1}{2}} e^{-\frac{\pi}{2}t},
     \end{align*}
     where $r_{k_0(t)}$ is the least value of $r_k$ that is greater than $t$. We are using the fact that the summands decay exponentially, so that the sum is dominated by its first term. Note that the computation is done assuming that $\sigma < 1$ (so that $\sigma - 1 < 0$), as that is the case that we are most interested in.

  Therefore, the overall bound is
  $
    \mathcal D(s) \ll h^{\frac12-\sigma+\theta+\varepsilon}\;(1 + |t|)^{\left|\sigma - \frac{1}{2}\right|} \textup{log}(1 + |t|)\, e^{-\frac{\pi}{2}|t|}.
 $
\end{proof}

\section{\large\bf The spectral side (continuous spectrum)}\label{spectralsideContinuous}

\begin{lemma}{\bf (Bounds for Fourier coefficients of the Eisenstein series)} \label{c_k(h,s)Bound} For $s\in\mathbb C$ let $$c(h, s) = \frac{  \sigma_{1-2s}(h)\, |h|^{s-\tfrac12}}{(2\pi)^{-s}\,\Gamma(s) \zeta(2s)}$$ denote the $h^{th}$ Fourier coefficient of the Eisenstein series $E(*,s)$.   Then for $s=\tfrac12+iu$ with $u\in\mathbb R$ we have the bound
  $$c\big(h, \tfrac12+iu\big)  \ll h^{\varepsilon} \log (1+|u|)\, e^{\frac{\pi}{2} |u|}.$$ 
  \end{lemma}
\begin{proof} This follows from Stirling's asymptotic formula  (Proposition \ref{Stirling}), and the lower bound $\zeta(1-2iu) \gg \log (1+|u|)^{-1}.$
\end{proof}

\begin{proposition} \label{SpectralInnerProductCont}
{\bf (Bound for the inner product $\big\langle P_h(*,s),E\left(*,\tfrac{1}{2} + iu\right) \big\rangle$)}
Let $s=\sigma+it$ with $\sigma>0$ and $t\in\mathbb R.$ Then
\begin{equation} \label{InnerProductContIdentity}
\left\langle P_h(*,s),E\left(*,\tfrac{1}{2} + iu\right) \right\rangle  = \frac{2\pi\sqrt{h}}{(4\pi h)^s}\cdot \overline{c\left(h,\tfrac{1}{2} + iu\right)}\, \frac{\Gamma\left(s - \frac{1}{2} + iu\right)  \Gamma\left(s - \frac{1}{2} - iu\right)}{\Gamma(s)}
\end{equation}
and
\begin{multline}\label{intstart}
 \Big\langle P_h(*,s),E\left(*,\tfrac{1}{2} + iu\right) \Big\rangle  \\   
    \ll h^{\frac12 -\sigma + \varepsilon}e^{-\frac{\pi}{2} \left|\vphantom{\frac{}{}} |u|-|t|\right|}    (1+|t + u|)^{\sigma -1}(1+|t -u|)^{\sigma -1}\log (1+|u|).
\end{multline}
\end{proposition}
\begin{proof} The proof of (\ref{InnerProductContIdentity}) is essentially the same as the proof of the analogous identity in Proposition \ref{SpectralInnerProductDiscrete}. The bound (\ref{intstart}) follows immediately from Lemma \ref{c_k(h,s)Bound} and Stirling's asymptotic formula given in Proposition \ref{Stirling}.

\end{proof}

\begin{remark}
  It is well known (by a similar Rankin--Selberg unfolding) that
  \begin{align*}
    \left\langle E(*,s),|\phi|^2 \right\rangle & = \sum_{n \neq 0} c(n)^2 \int\limits_0^{\infty} y^{s - 1} K_{i r}(2 \pi |n| y)^2 dy \\
                                               & = 2 \left(\int\limits_0^{\infty} y^{s - 1} K_{i r}(2 \pi y)^2 dy\right) \left(\sum_{n = 1}^{\infty} c(n)^2 n^{-s}\right) \\
                                               & = 2 \cdot 2^{s - 3} (2 \pi)^{-s} \frac{\Gamma\left(\frac{1}{2} s + i r\right) \Gamma\left(\frac{1}{2} s - i r\right) \Gamma\left(\frac{1}{2} s\right)^2}{\Gamma(s)} \cdot \frac{1}{\zeta(2 s)} L\left(s,\phi \otimes \overline{\phi}\right) \\
                                               & = 2^{-2} \pi^{-s} \frac{\Gamma\left(\frac{1}{2} s + i r\right) \Gamma\left(\frac{1}{2} s - i r\right) \Gamma\left(\frac{1}{2} s\right)^2}{\zeta(2 s) \Gamma(s)} L\left(s,\phi \otimes \overline{\phi}\right).
  \end{align*}
  Then we have the bound
  \begin{align}
    & \left\langle E\left(*,\tfrac{1}{2} + i u\right),|\phi|^2 \right\rangle \\ \nonumber & \hskip 25pt = 2^{-2} \pi^{-\frac{1}{2} - i u} \frac{\Gamma\left(\frac{1}{4} + i \left(\frac{1}{2} u + r\right)\right) \Gamma\left(\frac{1}{4} + i \left(\frac{1}{2} u - r\right)\right) \Gamma\left(\frac{1}{4} + \frac{1}{2} i u\right)^2}{\zeta(1 + 2 i u) \Gamma\left(\frac{1}{2} + i u\right)} L\left(\tfrac{1}{2} + i u,\phi \otimes \overline{\phi}\right) \\
    \nonumber & \hskip 25pt \ll (1 + |u|)^{-\frac{1}{2} + \eta + \varepsilon} (1 + |2 r + u|)^{-\frac{1}{4}} (1 + |2 r - u|)^{-\frac{1}{4}} e^{-\frac{\pi}{4} (|2r + u| + |2r - u|)},
  \end{align}
  where $0 \leq \eta \leq \frac{1}{2}$ represents the current best bound for the Rankin--Selberg L-function on the critical line (with $\eta = \frac{1}{2}$ being the convexity bound and any $\eta < \frac{1}{2}$ being a subconvexity bound).
\end{remark}

\begin{proposition} {\bf (Continuous spectrum bound)}\label{ContinuousSpectrumBound}
  Fix $\varepsilon > 0$ and set $s=\sigma+it$. Let
 \begin{equation}\label{contbound}
    \mathcal C(s) =\frac{1}{4\pi}\int\limits_{-\infty}^\infty \Big\langle P_h(*,s), E\left(*,\tfrac12+iu\right) \Big\rangle \Big\langle E\left(*, \tfrac12+iu\right),\,|\phi|^2\Big\rangle \, du,
\end{equation}
  denote the continuous part of the spectral expansion (\ref{SpectralExpansion}). 
  which  is defined for $\sigma > 1$.
  Then $ \mathcal C(s)$ has holomorphic continuation to $\sigma > 0$,  and for $\varepsilon < \sigma\leq 1$ and $|t| \rightarrow \infty$, it satisfies the  bound
  \[
    \mathcal C(s) \ll  h^{\frac12 -\sigma + \varepsilon} |t|^{2 (\sigma - 1) + \varepsilon} e^{-\frac{\pi}{2} |t|}.
  \]
\end{proposition}
\begin{proof}
  By applying the preceding bounds for the two inner products in the integrand, we have
  \begin{align*}
    & \mathcal C(s) \ll h^{\frac{1}{2} - \sigma + \varepsilon} \int\limits_{-\infty}^{\infty} (1 + |u|)^{-\frac{1}{2} + \eta + \varepsilon} (1 + |u + 2 r|)^{-\frac{1}{4}} (1 + |u - 2 r|)^{-\frac{1}{4}} \\ & \hskip 125pt \cdot (1 + |u + t|)^{\sigma - 1} (1 + |u - t|)^{\sigma - 1} e^{-\frac{\pi}{2} \left(\left|\vphantom{\frac{}{}} |u|-|t|\right| + \frac{1}{2} (|u + 2 r| + |u - 2 r|)\right)} du.
  \end{align*}
  Because we seek to analyze as $|t| \rightarrow \infty$ and the integrand is unchanged by $t \mapsto -t$, we can without loss of generality consider the case $t \rightarrow \infty$ and assume $t > 0$, and in fact we can assume $t > 2 r$. We then divide the interval of integration $(-\infty,\infty)$ into the subintervals
  \[
    (-\infty,-t] \cup [-t,-2 r] \cup [-2 r,0] \cup [0,2 r] \cup [2 r,t] \cup [t,\infty)
  \]
  and bound each of them separately. Note as well that the integrand is unchanged by $u \mapsto -u$, so the integrals over the opposite pairs of intervals are equal and it suffices to bound the integrals over the intervals $[0,2 r]$, $[2 r,t]$, and $[t,\infty)$. We have the bounds
  \begin{align*}
    & \int\limits_t^{\infty} (1 + u)^{-\frac{1}{2} + \eta + \varepsilon} (1 + u + 2 r)^{-\frac{1}{4}} (1 + u - 2 r)^{-\frac{1}{4}} (1 + u + t)^{\sigma - 1} (1 + u - t)^{\sigma - 1} e^{-\frac{\pi}{2} (2 u - t)} du \\ & \hskip 50pt \ll (1 + t)^{\sigma - \frac{3}{2} + \eta + \varepsilon} (1 + t + 2 r)^{-\frac{1}{4}} (1 + t - 2 r)^{-\frac{1}{4}} e^{-\frac{\pi}{2} t} \\
    & \hskip 50pt \ll (1 + t)^{\sigma - 2 + \eta + \varepsilon} e^{-\frac{\pi}{2} t},
  \end{align*}
  \begin{align*}
    & \int\limits_{2 r}^t (1 + u)^{-\frac{1}{2} + \eta + \varepsilon} (1 + u + 2 r)^{-\frac{1}{4}} (1 + u - 2 r)^{-\frac{1}{4}} (1 + u + t)^{\sigma - 1} (1 - u + t)^{\sigma - 1} e^{-\frac{\pi}{2} t} du \\ & \hskip 50pt \ll e^{-\frac{\pi}{2} t} \int\limits_1^t u^{-1 + \eta + \varepsilon} \left(t^2 - u^2\right)^{\sigma - 1} du \\
    & \hskip 50pt \ll t^{2 (\sigma - 1) + \eta + \varepsilon} e^{-\frac{\pi}{2} t},
  \end{align*}
  and
  \begin{align*}
    & \int\limits_0^{2 r} (1 + u)^{-\frac{1}{2} + \eta + \varepsilon} (1 + u + 2 r)^{-\frac{1}{4}} (1 - u + 2 r)^{-\frac{1}{4}} (1 + u + t)^{\sigma - 1} (1 - u + t)^{\sigma - 1} e^{-\frac{\pi}{2} (t - u + 2 r)} du \\ & \hskip 50pt \ll (1 + r)^{\frac{3}{4}} (1 + t)^{\sigma - 1} (1 + t - 2 r)^{\sigma - 1} e^{-\frac{\pi}{2} t} \\
    & \hskip 50pt \ll (1 + t)^{2 (\sigma - 1)} e^{-\frac{\pi}{2} t}.
  \end{align*}
  Note that we are assuming $\sigma \leq 1$ in the above computations, and we allow the implicit constants to depend on $r$ (i.e. on $\phi$). These bounds immediately yield the final result. We can remove the $\eta$ in the exponent as a result of the fact that $L\left(s,\phi \otimes \overline{\phi}\right)$ is ``Lindel\"of on average'' (see \cite{BR1999}) by using a dyadic division of the interval of integration, similarly to what was done in the discrete spectrum case. However, as the original asymptotic upper bound for the continuous spectrum part is already smaller than the asymptotic upper bound for the discrete spectrum part, this, and any other improvement to the continuous spectrum bound, would make no difference for the subsequent arguments and results.
\end{proof}

\section{\large \bf The geometric side}\label{geometricside}

\begin{proposition} \label{Prop:GeomSide}  Let $h$ be a positive integer.
 Then for  $s\in\mathbb C$ with $\text{\rm Re}(s) > 1,$
  we have 
  \[
    \Big\langle P_h(*,s),|\phi|^2 \Big\rangle = \sum_{n \neq 0,-h} c(n) c(n + h) \cdot \mathcal I_r(n,s).
  \]
  where
  $$\mathcal I_r(n,s) := \int\limits_0^{\infty} y^{s}e^{-2\pi hy}K_{ir}(2\pi|n|y) K_{ir}(2\pi|n + h|y)\, \frac{dy}{y}.$$
\end{proposition}

\begin{proof} See (\ref{PhPhiSquaredInnerProduct}).
\end{proof}

\begin{theorem} \label{Thm:GeomSide}
  Fix $\varepsilon> 0$  and let $s\in\mathbb C$ with $\textup{Re}(s) > 1 + 2\varepsilon$. Then
  $$
  \Big\langle P_h(*,s),|\phi|^2 \Big\rangle =  \frac{\Gamma(s)}{2\,(2\pi h)^s}\cdot \Big(\mathcal G(s) +\mathcal T(s)   \Big)
  $$
  where
\begin{align*}
\mathcal G(s) & = \frac{1}{2\pi i}\int\limits_{\textup{Re}(w) = -\frac{1}{2} - \varepsilon} 
 \hskip-10pt 
  \frac{\Gamma\left(\frac{1}{2}s + \frac{1}{2} + w\right) \Gamma\left(\frac{1}{2}s + w\right) \Gamma(-w)}{2^{-w} h^{2w}\,\Gamma\left(s + \frac{1}{2} + w\right)} \;\mathcal F_{r,2}(-2w)
  \\
  &
  \hskip 170pt
  \cdot
   \sum_{n < -h \textup{ or } n > 0} c(n) c(n + h) \,\big |n(n + h)\big|^w \; dw.
  \end{align*}
  and
\begin{align*} \mathcal T(s) & = \frac{1}{2\pi i} \int\limits_{\textup{Re}(w) = -\frac{1}{2} + \varepsilon}\hskip-10pt 
   \frac{\Gamma\left(\frac{1}{2}s + \frac{1}{2} + w\right) \Gamma\left(\frac{1}{2}s + w\right) \Gamma(-w)}{2^{-w} h^{2w} \,\Gamma\left(s + \frac{1}{2} + w\right)} \,\mathcal F_{r,0}(-2w) \\
   &
   \hskip 180pt
   \cdot
     \sum_{-h < n < 0} c(n) c(n + h)\, \big|n(n + h)\big|^w \,dw.
   \end{align*}

\end{theorem}

\begin{proof} 
The function 
$$
 \mathcal I_r(n,s) = (2\pi h)^{-s}\int\limits_0^{\infty} y^{s}e^{-y}K_{ir}\left(\tfrac{|n|}{h}y\right) K_{ir}\left(\tfrac{|n + h|}{h}y\right) \frac{dy}{y},$$
 defined in Proposition \ref{Prop:GeomSide},  can be evaluated by the following lemma.
  
  \begin{lemma}\label{identity1} Let $m,n\in \mathbb R_{>0}$ and $r\in\mathbb R$. Then for  $s\in\mathbb C$ with $\text{\rm Re}(s)>0$ we have
   \begin{align*} \int\limits_0^{\infty} K_{ir}(my) K_{ir}(ny) e^{-y} y^{s}\, \frac{dy}{y} 
   &\; = \; \frak g(s)\cdot \lim_{\delta\to 0^+} \int\limits_\delta^{\infty} F\Big(\tfrac{s+1}{2},\tfrac{s}{2};s + \tfrac{1}{2};1 - \alpha(u)^2\Big)\cos(ru)\,du
   \end{align*} 
where $\displaystyle\frak g(s) = \frac{\sqrt{\pi}}{2^{s}} \frac{\Gamma(s)^2}{\Gamma\left(s + \frac{1}{2}\right)}$ and $\alpha(u) = \Big(m^2 + n^2 + 2mn\cosh(u)\Big)^{\frac{1}{2}}$.
  \end{lemma}
  
 \begin{proof}
 See (\cite{MR1319517}, Lemma 1) and (\cite{MR4678127}, Proposition 9.6).
 \end{proof}

 It follows that for $\textup{Re}(s) > 0,$  we have
\begin{align} \label{I_r(n,s)}
\mathcal I_r(n,s)  & = \frac{\frak g(s)}{(2\pi h)^s}\cdot  \lim_{\delta\to 0^+} \;\int\limits_\delta^{\infty} F\left(\tfrac{s+1}{2},\tfrac{s}{2};s + \tfrac{1}{2};1 - \tfrac{|n|^2 + |n + h|^2}{h^2} - \tfrac{2|n(n + h)|}{h^2}\,\textup{cosh}(u)\right)\\
&
\hskip 330pt \cdot \textup{cos}(ru)\; du.
\nonumber
\end{align}

We introduce the limit as $\delta\to 0$ in the above identity  because in order to evaluate the inner product $ \left\langle P_h(*,s),|\phi|^2 \right\rangle$ we need to multiply $\mathcal I_r(n,s)$ by $c(n) c(n + h)$ and sum over $n\ne 0,h.$ It will turn out that the above identity for $\mathcal I_r(n,s)$ needs to be slightly modified in order to show that the interchange of summation and integration is justified for the sum over $n$ with $-h<n<0.$

\vskip 5pt
To evaluate $\mathcal I_r(n,s)$ we use the following integral representation of the Gaussian hypergeometric function. 
\begin{lemma}\label{Barnes} Fix $\rho>0.$ Then
  \[
  F(\alpha,\beta;\gamma;z) = \frac{\Gamma(\gamma)}{\Gamma(\alpha) \Gamma(\beta)} \cdot \frac{1}{2\pi i} \int\limits_{\text{\rm Re}(w) = -\rho} \frac{\Gamma(\alpha + w) \Gamma(\beta + w) \Gamma(-w)}{\Gamma(\gamma + w)} (-z)^w dw,
\]
where $|\textup{arg}(-z)| < \pi$  and $\text{\rm Re}(\alpha),\text{\rm Re}(\beta) > \rho$.
\end{lemma}
\begin{proof}
See \cite{MR1575118}.
\end{proof}

It  immediately follows from Lemma \ref{Barnes} that for $u\ge \delta >0$, $\textup{Re}(s) > 1 + 2\varepsilon$ and $n\ne 0,-h$, we have
\begin{align}
  & 
  F\left(\tfrac{s+1}{2},\tfrac{s}{2};s + \tfrac{1}{2};1 - \tfrac{|n|^2 + |n + h|^2}{h^2} - \tfrac{2|n(n + h)|}{h^2}\,\textup{cosh}(u)\right) 
  \nonumber
  \\
  &
  \hskip 25pt 
  = \frac{\Gamma\left(s + \frac{1}{2}\right)}{\Gamma\left(\frac{1}{2}s + \frac{1}{2}\right) \Gamma\left(\frac{1}{2}s\right)} \int\limits_{\textup{Re}(w) = -\frac{1}{2} - \varepsilon}
  \hskip-13pt \frac{\Gamma\left(\frac{1}{2}s + \frac{1}{2} + w\right) \Gamma\left(\frac{1}{2}s + w\right) \Gamma(-w)}{2\pi i\cdot \Gamma\left(s + \frac{1}{2} + w\right)} 
  \nonumber
  \\ 
  & 
  \hskip 170pt 
  \cdot \left(\frac{|n|^2 + |n + h|^2 - h^2}{h^2} + \frac{2|n(n + h)|}{h^2}\,\textup{cosh}(u)\right)^w dw 
 \nonumber
  \\
  & 
  \hskip 25pt
  = \frac{\Gamma\left(s + \frac{1}{2}\right)}{\Gamma\left(\frac{1}{2}s + \frac{1}{2}\right) \Gamma\left(\frac{1}{2}s\right)}\int\limits_{\textup{Re}(w) = -\frac{1}{2} - \varepsilon}
  \hskip-13pt \frac{\Gamma\left(\frac{1}{2}s + \frac{1}{2} + w\right) \Gamma\left(\frac{1}{2}s + w\right) \Gamma(-w)}{2\pi i\cdot \Gamma\left(s + \frac{1}{2} + w\right)} \left(\frac{2|n(n + h)|}{h^2}\right)^w 
  \label{Irns}
  \\ 
  & 
  \hskip 315pt 
  \cdot (a_h(n) + \textup{cosh}(u))^w dw,
  \nonumber 
  \end{align}
  where for $n \neq 0,-h$ a rational integer, 
$$
  a_h(n) := \frac{|n|^2 + |n + h|^2 - h^2}{2|n||n + h|} = \frac{n (n + h)}{|n (n + h)|} = \textup{sgn}(n (n + h)).
$$
which implies
\begin{equation} \label{a_h(n)}
a_h(n) = \begin{cases} \;\;1 &\text{if}\; n<-h \;\text{or} \; n>0,\\
-1 &\text{if}\; -h<n<0.\end{cases}.
\end{equation}
Furthermore, by (\ref{I_r(n,s)}) and (\ref{Irns})  we have
\begin{align}
\mathcal I_r(n,s) &  = \frac{\Gamma(s)}{2\,(2\pi h)^s}\cdot\lim_{\delta\to 0}\int\limits_\delta^{\infty} \int\limits_{\textup{Re}(w) = -\frac{1}{2} - \varepsilon} \hskip-10pt\frac{\Gamma\left(\frac{1}{2}s + \frac{1}{2} + w\right) \Gamma\left(\frac{1}{2}s + w\right) \Gamma(-w)}{2\pi i\cdot \Gamma\left(s + \frac{1}{2} + w\right)}
\label{SecondIrns}\\
&
\hskip 125pt
\cdot  \left(\frac{2|n(n + h)|}{h^2}\right)^w  \big(a_h(n) + \textup{cosh}(u)\big)^w \textup{cos}(ru)\;dw du.
\nonumber
\nonumber\end{align}

\noindent
$\underline{\text{{\bf Case 1: $n<-h$ or $n>0$:}}}$
\vskip 8pt
  It follows from (\ref{a_h(n)}) that in this case $a_h(n) = 1$ and by Proposition \ref{PropFba}, we have
  \begin{equation} \label{coshIntegral1}
  \lim_{\delta\to 0} \int\limits_\delta^\infty \big(a_h(n) + \textup{cosh}(u)\big)^w \textup{cos}(ru)\; du = \mathcal F_{r,2}(-2w).
  \end{equation}
  Note that the above integral converges absolutely when $\text{\rm Re}(w) = -\tfrac12-\varepsilon.$
 Combining (\ref{coshIntegral1}) and (\ref{SecondIrns}) establishes that
 $$\sum_{n < -h \textup{ or } n > 0} c(n) c(n + h)\cdot \mathcal I_r(n,s) =  \frac{\Gamma(s)}{2\,(2\pi h)^s}\cdot\mathcal G(s).$$

 \noindent
$\underline{\text{{\bf Case 2: $-h<n<0$:}}}$
\vskip 8pt
 In this case $a_h(n) = -1$ and the term $ \big(a_h(n) + \textup{cosh}(u)\big)^w$ in (\ref{SecondIrns}) becomes problematic as $u\to 0$ since $\text{\rm Re}(w) = -\frac12-\varepsilon.$ As long as $\delta > 0$, however, we can shift the line of integration in the $w$ integral in 
 (\ref{SecondIrns}) to $\text{\rm Re}(w) = -\frac12+\varepsilon.$ In this case the integral
 $\lim\limits_{\delta\to 0} \int\limits_\delta^\infty \big(-1 + \textup{cosh}(u)\big)^w \textup{cos}(ru)\; du$
 converges to  $\mathcal F_{r,0}(-2w)$ (by Proposition  \ref{PropFba})
 and
 we obtain the following identity which completes the proof of Theorem \ref{Thm:GeomSide}

 $$ \sum_{-h<n<0} c(n) c(n + h)\cdot \mathcal I_r(n,s) =  \frac{\Gamma(s)}{2\,(2\pi h)^s}\cdot\mathcal T(s). \hskip 70pt \qedhere$$
\end{proof}

\section{\large \bf Relating $\left\langle P_h(*,s),|\phi|^2 \right\rangle$ to $L_h(s,\phi)$} \label{RelatingL_h(s)}

Relating the shifted convolution L-function $L_h(s,\phi)$ to the inner product $\left\langle P_h(*,s),|\phi|^2 \right\rangle$ requires several convoluted steps via an intermediate shifted convolution L-function, denoted $L_h^\#(s,\phi)$, defined in Definition \ref{LhSharp}.

We now prove the following proposition relating $\left\langle P_h(*,s),|\phi|^2 \right\rangle$ to $L_h^{\#}(s,\phi)$.

\begin{proposition} \label{Prop:InnerProductBound}
  The function $L_h^\#(s,\phi)$ has meromorphic continuation to $s\in\mathbb C$ with $\textup{Re}(s) > 0$. Fix $\varepsilon > 0$ and set $s=\sigma+it$. Then for $5\varepsilon<\sigma<1+3\varepsilon$ and $t\in\mathbb R$ we have
  \[
    \Big\langle P_h(*,s),|\phi|^2 \Big\rangle = \frac{\Gamma\left(\frac{s}{2}\right)^2}{ 2^{2+\frac12s}\,\pi^{s} } L_h^{\#}(s,\phi) + \mathcal O\left(h^{1 - \sigma + \varepsilon} (1+|t|)^{-\frac12 + \varepsilon} e^{-\frac{\pi}{2}|t|}\right).
  \]
\end{proposition}

\begin{proof}
Recall Theorem \ref{Thm:GeomSide} which states that for $\textup{Re}(s) > 1 + 2\varepsilon$,
\begin{equation} \label{GeomSideIdentity}
  \Big\langle P_h(*,s),|\phi|^2 \Big\rangle =  \frac{\Gamma(s)}{2\,(2\pi h)^s}\cdot  \Big(\mathcal G_h(s,\phi) +\mathcal T_h(s,\phi)   \Big)
  \end{equation}
  where
\begin{align*}
\mathcal G_h(s,\phi) & = \frac{1}{2\pi i}\int\limits_{\textup{Re}(w) = -\frac{1}{2} - \varepsilon} 
 \hskip-10pt 
  \frac{\Gamma\left(\frac{1}{2}s + \frac{1}{2} + w\right) \Gamma\left(\frac{1}{2}s + w\right) \Gamma(-w)}{2^{-w} h^{2w}\;\Gamma\left(s + \frac{1}{2} + w\right)} \;  \,\mathcal F_{r,2}(-2w) 
  \\
  &
  \hskip 150pt
  \cdot  \sum_{n < -h \textup{ or } n > 0} c(n) c(n + h) \,\big |n(n + h)\big|^w \; dw.
  \end{align*}
  and

\begin{align*} \mathcal T_h(s,\phi) & = \frac{1}{2\pi i} \int\limits_{\textup{Re}(w) = -\frac{1}{2} + \varepsilon}\hskip-10pt 
   \frac{\Gamma\left(\frac{1}{2}s + \frac{1}{2} + w\right) \Gamma\left(\frac{1}{2}s + w\right) \Gamma(-w)}{2^{-w} h^{2w}\;\Gamma\left(s + \frac{1}{2} + w\right)} \,\mathcal F_{r,0}(-2w) \\
   &
   \hskip 150pt
   \cdot
     \sum_{-h < n < 0} c(n) c(n + h) \,\big |n(n + h)\big|^w \; dw.
   \end{align*}
   
 Let $-1 \le\text{\rm Re}(w) \le  -\tfrac12+\varepsilon.$ By Proposition  \ref{F_{b,2}(s)} we have 
 $$
\mathcal F_{r,2}(-2w)  = \frac{\Gamma\left(-w + ir\right) \Gamma\left(-w - ir\right)}{2^{w+1}\Gamma(-2w)}\; \sim \;C_r \big(1+|w|\big)^{-\frac12}
 $$ 
 as $|\text{\rm Im}(w)|\to\infty$ for a constant $C_r$ depending at most on $r$. It follows that the ratio of gamma functions multiplied by $\mathcal F_{r,2}(-2w)$ which appears in the integrand of $\mathcal G_h(s,\phi)$ has exponential decay in the variable $w$ which allows us
  to evaluate $\mathcal G_h(s,\phi)$ by shifting the line of integration  in $\mathcal G_h(s,\phi)$ to the left to $\text{\rm Re}(w) = -\frac{1}{2} - 2\varepsilon$. It is convenient to adopt the following notation.

   \begin{definition}{\bf (Notation for the shifted $\mathcal G$ integral)} For  $\beta \le -\tfrac12-\varepsilon$ we define
 \begin{align*}
\mathcal G_{h,\beta}(s,\phi) & = \frac{1}{2\pi i}\int\limits_{\textup{Re}(w) = \beta} 
 \hskip-10pt 
  \frac{\Gamma\left(\frac{1}{2}s + \frac{1}{2} + w\right) \Gamma\left(\frac{1}{2}s + w\right) \Gamma(-w)}{\Gamma\left(s + \frac{1}{2} + w\right)} \;  2^w h^{-2w}\,\mathcal F_{r,2}(-2w)
  \\
  &
  \hskip 150pt
  \cdot
   \sum_{n < -h \textup{ or } n > 0} c(n) c(n + h) \,\big |n(n + h)\big|^w \; dw.
  \end{align*}
   \end{definition}
   
 Assume $1+2\varepsilon<\text{\rm Re}(s)<1+4\varepsilon$. We can then evaluate $\mathcal G_h(s,\phi)$ by shifting the line of integration  to $\textup{Re}(w) = -\tfrac12-2\varepsilon$. This crosses over a simple pole at $w = -\tfrac12 s$ and no other poles. It follows by Cauchy's residue theorem that
\[
  \mathcal G_h(s,\phi) = \mathcal R_h^*(s,\phi) + \mathcal G_{h,-\frac12-2\varepsilon}(s,\phi)
\]
where the residue term from the pole at $w=-\tfrac12 s$ is given by
\begin{align*}
  \mathcal R_h^*(s,\phi) &=\frac{\sqrt{\pi} h^s \,\Gamma\left(\tfrac12 s\right)}{2^{\frac{s}{2}} \,\Gamma\left(\tfrac12s + \frac{1}{2}\right)}\;  \mathcal F_{r,2}(s) \hskip-5pt\sum_{n < -h \textup{ or } n > 0} c(n) c(n + h)\, \big|n(n + h)\big|^{-\frac{1}{2}s} 
  \\
  &
  =\frac{\sqrt{\pi} h^{s}\,\Gamma\left(\tfrac12 s\right)}{2^{\frac{s}{2}} \, \Gamma\left(\tfrac12s + \frac{1}{2}\right)}\; \left(L_h^{\#}(s,\phi) \; - \; \mathcal F_{r,2}(s) \sum_{-h < n < 0} \frac{c(n) c(n + h)}{ \big|n(n + h)\big|^{\frac{1}{2}s}   }\, \right).
\end{align*}

\vskip 8pt
Let $\mathfrak g_h(s) =  {\displaystyle \frac{\Gamma(s)}{2\,(2\pi h)^s}}.$
It immediately follows from (\ref{GeomSideIdentity}) and the above that
\begin{align*}
  \Big\langle P_h(*,s),|\phi|^2 \Big\rangle &=  \mathfrak g_h(s)\cdot\Big(\mathcal R_h^*(s,\phi) + \mathcal G_{h,-\frac12-2\varepsilon}(s,\phi) + \mathcal T_h(s,\phi)\Big) 
  \\
     &
     \hskip-60pt
     = \mathfrak g_h(s)\cdot  \frac{\sqrt{\pi} h^{s}\,\Gamma\left(\tfrac12 s\right)}{2^{\frac{s}{2}} \, \Gamma\left(\tfrac12s + \frac{1}{2}\right)}\; \left(L_h^{\#}(s,\phi) \; - \; \mathcal F_{r,2}(s) \sum_{-h < n < 0} \frac{c(n) c(n + h)}{ \big|n(n + h)\big|^{\frac{1}{2}s}   }\, \right)
     \\
&
\hskip 15pt
+ \mathfrak g_h(s)\cdot \Big(\mathcal G_{h,-\frac12-2\varepsilon}(s,\phi) + \mathcal T_h(s,\phi)\Big).
     \end{align*}

Consequently, for $1+2\varepsilon<\text{\rm Re}(s)<1+4\varepsilon$, we have shown that

\begin{align} 
  \Big\langle P_h(*,s),|\phi|^2 \Big\rangle & \;= \;\frac{\Gamma\left(\tfrac12s\right)^2}{ 2^{2+\frac12s}\,\pi^{s} } \left(L_h^{\#}(s,\phi) \, - \, \mathcal F_{r,2}(s) \sum_{-h < n < 0} \frac{c(n) c(n + h)}{ \big|n(n + h)\big|^{\frac{1}{2}s}   }\, \right)
\label{InnerProductExpansion}\\
&
\hskip 70pt
+\frac{\Gamma(s)}{2\,(2\pi h)^s}\cdot \Big(\mathcal G_{h,-\frac12-2\varepsilon}(s,\phi) + \mathcal T_h(s,\phi)\Big).
\nonumber
     \end{align}

Next, we will show that all the terms in the identity (\ref{InnerProductExpansion}), except the one involving $L_h^\#(s,\phi)$, have meromorphic continuation to $s\in\mathbb C$ \;\text{with}\;$5\varepsilon<\textup{Re}(s)\le 1+3\varepsilon$. We will then obtain bounds for these terms when  $5\varepsilon<\textup{Re}(s)\le 1+3\varepsilon$ is fixed and $|\text{\rm Im}(s)|\to\infty.$ This will allow us to obtain the meromorphic continuation of $L_h^{\#}(s,\phi)$ and its growth away from poles in the region $\text{\rm Re}(s) > 5\varepsilon$.

\vskip 5pt

\begin{lemma}
 Let $s=\sigma+it$ with $\varepsilon < \sigma < 1 + \varepsilon$ and $t\in\mathbb R$. Then
  \[
    \sum_{-h < n < 0} \frac{c(n) c(n + h)}{|n (n + h)|^{\frac{1}{2}s}} \ll h^{1 - \sigma + \varepsilon}.
  \]
\end{lemma}

\begin{proof}
  We have
  \[
    \sum_{-h < n < 0} \frac{c(n) c(n + h)}{|n (n + h)|^{\frac{1}{2}s}} \ll \sum_{0 < n < h} \frac{|c(n)|^2}{n^{\sigma}},
  \]
  so it suffices to bound the sum on the right side. This can be done by standard techniques by considering the integral
  \[
    \frac{1}{2\pi i} \int\limits_{\textup{Re}(w) = 1 - \textup{Re}(s) + \varepsilon} L\left(s + w,\phi \otimes \overline{\phi}\,\right) \frac{x^w}{w^3}\, dw,
  \]
  where $L\left(s,\phi \otimes \overline{\phi}\,\right)$ is the Rankin--Selberg L-function defined by
  \[
    L\left(s,\phi \otimes \overline{\phi}\,\right) = \sum_{n \neq 0} \frac{|c(n)|^2}{|n|^s}
  \]
  for $\textup{Re}(s) > 1$, and shifting the line of integration to $\textup{Re}(w) = -\textup{Re}(s) + \varepsilon$. That gives the bound
  \[
    \sum_{0 < n < x} \frac{|c(n)|^2}{n^s} \left(\textup{log}\left(\tfrac{x}{n}\right)\right)^2 \ll x^{1 - \sigma + \varepsilon};
  \]
  considering the case where $s$ is real and letting $x = h$ yields the final result.
\end{proof}

It immediately follows from the bound above, Stirling's asymptotic formula (Proposition \ref{Stirling}), and the bound for $\mathcal F_{r,2}(s)$ (Proposition \ref{F_{b,2}(s)}) that for $5\varepsilon<\sigma <1+3\varepsilon$ we have
\begin{equation} \label{F_(r,2)(s)-Bound}
 \frac{\Gamma\left(\tfrac12s\right)^2}{ 2^{2+\frac12s}\,\pi^{s} } \cdot \mathcal F_{r,2}(s)\sum_{-h < n < 0} \frac{c(n) c(n + h)}{ \big|n(n + h)\big|^{\frac{1}{2}s}   } \; \ll \;  h^{1 - \sigma + \varepsilon} \big(1+|t|\big)^{\sigma-\frac32} e^{-\frac{\pi}{2}|t|}.
\end{equation}

\begin{lemma} \label{(G(s)+T(s))-Bound}
Let $s=\sigma+it$ with $5\varepsilon<\sigma<1+3\varepsilon$ and $t\in\mathbb R$.  Then
 $$ \left|\frac{\Gamma(s)}{2\,(2\pi h)^s}\cdot \Big(\mathcal G_{h,-\frac12-2\varepsilon}(s,\phi)+\mathcal T_h(s,\phi)\Big)\right| 
  \;\ll \;
 h^{-\sigma}\, (1 + |t|)^{\sigma - 1} \textup{log}(1 + |t|) e^{-\frac{\pi}{2}|t|}.$$
\end{lemma}

\begin{proof} First of all by Stirling's asymptotic formula (Proposition \ref{Stirling}) we have
\begin{equation} \label{GammaBound}
\Big| \frac{\Gamma(s)}{2\,(2\pi h)^s}\Big| \ll h^{-\sigma} (1+|t|)^{\sigma-\frac12}\, e^{-\frac{\pi}{2}|t|}.
\end{equation}
Since $\sum\limits_{n < -h \textup{ or } n > 0} c(n) c(n + h)\, \big|n(n + h)\big|^w \ll 1$ for
$\textup{Re}(w) = -\tfrac12-2\varepsilon$ and we also have $\mathcal F_{r,2}(-2w)\ll (1+|\textup{Im}(w)|)^{-\frac12}$ by Proposition \ref{F_{b,2}(s)},  we see that
\[
  \mathcal G_{h,-\frac12-2\varepsilon}(s,\phi) \ll h^{1 + 4 \varepsilon} \int\limits_{\textup{Re}(w) = -\tfrac12-2\varepsilon} \big(1+|\textup{Im}(w)|   \big)^{-\frac12} \left| \frac{\Gamma\left(\frac{1}{2}s + \frac{1}{2} + w\right) \Gamma\left(\frac{1}{2}s + w\right) \Gamma(-w)}{\Gamma\left(s + \frac{1}{2} + w\right)}\right|\;  dw.
\]
  
 Note that for $s=\sigma+it\in\mathbb C$ varying in the range $5\varepsilon<\sigma<1+3\varepsilon$ we don't hit any poles of $\Gamma\left(\frac{1}{2}s + \frac{1}{2} + w\right) \Gamma\left(\frac{1}{2}s + w\right)$, and, in fact, maintain a distance of at least $\varepsilon$ between $s$ and any pole. It follows that if the above integral converges absolutely in this range then $\mathcal G_{-\frac12-2\varepsilon}(s,\phi)$ is a holomorphic function of $s$ for $5\varepsilon <\sigma<1+3\varepsilon$.

\vskip 10pt
Next apply Stirling's formula (Proposition \ref{Stirling}) and let  $w = \left(-\frac{1}{2} - 2\varepsilon\right) + iv$  in the above integral. This gives
\begin{align*}
 \mathcal G_{h,-\frac{1}{2} - 2\varepsilon}(s,\phi)
 &
  \ll 
   h^{1 + 4 \varepsilon} \int\limits_{-\infty}^{\infty} G(t,v) \;dv
  \end{align*}
  where
  $$G(t,v) :=  \frac{\left(1 + \left|\frac{1}{2}t + v\right|\right)^{\sigma - \frac{3}{2} - 4\varepsilon} (1 + |v|)^{-\frac{1}{2} + 2\varepsilon}}{(1 + |t + v|)^{\sigma - \frac{1}{2} - 2\varepsilon}} \cdot 
  \exp\left(-\tfrac{\pi}{2}(|t+2v   | + |v| - |t+v|  )\right).$$
  To estimate the growth of $\mathcal G_{h,-\frac{1}{2} - 2\varepsilon}(\sigma+it,\phi)$ as $|t|\to\infty$ we first get a bound for  $G(t,v)$. We consider the case $t \rightarrow \infty$ and assume $t > 0$; this immediately yields the same asymptotic bound for $t \rightarrow -\infty$ because $G(-t,-v) = G(t,v)$. We split the interval of integration into four smaller intervals and separately bound each of them.

  \vskip 15pt
  \noindent
  \underline{{\bf Case 1:} $v > 0$.}
  \vskip 5pt
  In this case,
 $
    G(t,v) = \frac{\left(1 + \frac{1}{2}t + v\right)^{\sigma - \frac{3}{2} - 4\varepsilon} (1 + v)^{-\frac{1}{2} + 2\varepsilon}}{(1 + t + v)^{\sigma - \frac{1}{2} - 2\varepsilon}} e^{-\pi v}.
  $
  Because of the exponential decay, we have
  $
    \int\limits_0^{\infty} G(t,v) dv \ll (1 + t)^{-1 - 2\varepsilon}.
  $

  \begin{lemma} \label{IntegralBound} Let $t>0$ and $a,b\in\mathbb R.$ Then
 $$\int\limits_{-\tfrac12t}^0 \frac{\left(1+\tfrac12t+v\right)^a} 
 {\phantom{\big(}(1-v)^{b}} \, dv  \ll_{a,b} \,(1+t)^{a-b+1} \, \log(1+t).$$
  \end{lemma} 
  \begin{proof} 
   We first break the interval of integration into two pieces and then make a change of variables in each integral. This gives
  \begin{align*}
  \int\limits_{-\tfrac12t}^0 \frac{\left(1+\tfrac12t+v\right)^a}{  (1-v)^b }\,dv & = \int\limits_{-\tfrac12t}^{-\tfrac14t}\frac{\left(1+\tfrac12t+v\right)^a}{  (1-v)^b }\,dv\; + \; \int\limits_{-\tfrac14t}^{0} \frac{\left(1+\tfrac12t+v\right)^a}{  (1-v)^b }\,dv\\
  & \ll (1+t)^{-b} \int\limits_0^{\tfrac14t} (1+v)^a\;dv \; + \; (1+t)^a \int_0^{\tfrac14t}(1+v)^{-b}\;dv\\
  & \ll \begin{cases} (1+t)^{a-b+1}  & \text{if} \; a\ne -1\; \text{and}\; b\ne 1,\\
 (1+t)^{a-b+1}\log(1+t) & \text{if} \; a=-1 \;\text{or}\; b = 1.
 \end{cases}
  \end{align*}
  \end{proof}

 \pagebreak 
  \noindent
  \underline{{\bf Case 2:} $-\frac{1}{2}t < v < 0$.}
  \vskip 5pt
  In this case,
  \[
    G(t,v) = \frac{\left(1 + \frac{1}{2}t + v\right)^{\sigma - \frac{3}{2} - 4\varepsilon} (1 - v)^{-\frac{1}{2} + 2\varepsilon}}{(1 + t + v)^{\sigma - \frac{1}{2} - 2\varepsilon}}.
  \]
  Because $\frac{1}{2}t < t + v < t$, we have
  \[
    G(t,v) \ll (1 + t)^{-\sigma + \frac{1}{2} + 2\varepsilon} \left(1 + \tfrac{1}{2}t + v\right)^{\sigma - \frac{3}{2} - 4\varepsilon} (1 - v)^{-\frac{1}{2} + 2\varepsilon}.
  \]
  It immediately follows from Lemma \ref{IntegralBound} that
  \[
    \int\limits_{-\frac{1}{2}t}^0 G(t,v) dv 
    \ll  (1 + t)^{-\frac{1}{2}} \,\textup{log}(1 + t).
  \]

  \noindent
  \underline{{\bf Case 3:} $-t < v < -\frac{1}{2}t$.}
  \vskip 5pt
  In this case,
  \[
    G(t,v) = \frac{\left(1 - \frac{1}{2}t - v\right)^{\sigma - \frac{3}{2} - 4\varepsilon} (1 - v)^{-\frac{1}{2} + 2\varepsilon}}{(1 + t + v)^{\sigma - \frac{1}{2} - 2\varepsilon}} e^{\pi (t + 2v)}.
  \]
  Because $t + 2v < 0$, we have
  \[
    G(t,v) < \frac{\left(1 - \frac{1}{2}t - v\right)^{\sigma - \frac{3}{2} - 4\varepsilon} (1 - v)^{-\frac{1}{2} + 2\varepsilon}}{(1 + t + v)^{\sigma - \frac{1}{2} - 2\varepsilon}}.
  \]
  It follows by an argument analogous to the one in Case 2 that
  \[
    \int\limits_{-t}^{-\frac{1}{2}t} G(t,v) dv 
    \ll  (1 + t)^{-\frac{1}{2}} \,\textup{log}(1 + t).
  \]

  \noindent
  \underline{{\bf Case 4:} $v < -t$.}
  \vskip 5pt
  In this case,
 $
    G(t,v) = \frac{\left(1 - \frac{1}{2}t - v\right)^{\sigma - \frac{3}{2} - 4\varepsilon} (1 - v)^{-\frac{1}{2} + 2\varepsilon}}{(1 + t + v)^{\sigma - \frac{1}{2} - 2\varepsilon}} e^{\pi v}.
 $
  Thus
  \[
    \int\limits_{-\infty}^{-t} G(t,v) dv \ll (1+t) e^{-\pi t}.
  \]

  \vskip 12pt
  We can analyze $\mathcal T_h(s,\phi)$ by the same method because
\[
  \sum_{-h < n < 0} c(n) c(n + h)\, \big|n(n + h)\big|^w \ll h^{3\varepsilon}
\]
for $\textup{Re}(w) =-\tfrac12+\varepsilon$. Consequently, we obtain a smaller bound for $\mathcal T_h(s,\phi)$ than the bound for $\mathcal G_{h,-\frac{1}{2} - 2\varepsilon}(s,\phi)$ because we have
\[
  \left|\mathcal F_{r,0}(-2w)\right| = \left|\frac{\textup{cos}(i\pi r)}{\textup{cos}(\pi w)} \mathcal F_{r,2}(-2w)\right| \ll 
e^{-\pi|\text{\rm Im}(w)|}\cdot \left| \mathcal F_{r,2}(-2w)\right|.\]

The proof of Lemma \ref{(G(s)+T(s))-Bound} immediately follows from the bounds obtained in the four cases above and  the bound in (\ref{GammaBound}).

\end{proof}
The proof of Proposition \ref{Prop:InnerProductBound}   follows from (\ref{InnerProductExpansion}), (\ref{F_(r,2)(s)-Bound}), and Lemma \ref{(G(s)+T(s))-Bound}. \end{proof}

\section{\large \bf Proof of Theorem \ref{MainTheorem}}\label{lfunctionbound}

Recall that our main Theorem \ref{MainTheorem} gives the meromorphic continuation of $L_h(s,\phi)$ to ${\rm Re}(s) > 0$ and gives sharp bounds for its growth in this region away from poles. 

\begin{proof}
The key ingredient for the proof of Theorem \ref{MainTheorem} is Proposition \ref{Prop:InnerProductBound}  which gives the meromorphic continuation of $\big\langle P_h(*,s),|\phi|^2 \big\rangle$ as well as the asymptotic formula 
\[
  \Big\langle P_h(*,s),|\phi|^2 \Big\rangle = \frac{\Gamma\left(\frac{s}{2}\right)^2}{ 2^{2+\frac12s}\,\pi^{s} } L_h^{\#}(s,\phi) + \mathcal O\left(h^{1 - \sigma + \varepsilon} (1+|t|)^{-\frac12 + \varepsilon} e^{-\frac{\pi}{2}|t|}\right),
\]
which relates $\big\langle P_h(*,s),|\phi|^2 \big\rangle$ with 
$$L_h^\#(s) = \mathcal F_{r,2}(s) L_h(s,\phi) = 2^{\frac{1}{2}s - 1}\, \frac{\Gamma\left(\frac{1}{2}s + ir\right) \Gamma\left(\frac{1}{2}s - ir\right)}{\Gamma(s)}\, L_h(s,\phi)$$
where $\phi$ is the fixed Maass cusp form with Laplace eigenvalue $\tfrac14+r^2$.

Multiplying both sides of the above equation by
$
    \frac{2^{2 + \frac{1}{2}s}\pi^s}{\Gamma\left(\frac{1}{2}s\right)^2 \mathcal F_{r,2}(s)}
 $
  gives
   \begin{align} \label{LhFunction}
    & L_h(s,\phi) = \frac{2^{2 + \frac{1}{2}s}\pi^s \Big\langle P_h(*,s),|\phi|^2 \Big\rangle}{\Gamma\left(\frac{1}{2}s\right)^2 \mathcal F_{r,2}(s)}  + \mathcal O\Big(h^{1 - \sigma + \varepsilon} (1 + |t|)^{1 - \sigma + \varepsilon}\Big).
  \end{align}
  Combining this with the bound for $\big\langle P_h(*,s),|\phi|^2 \big\rangle$ given in Theorem \ref{Thm:InnerProdBound}, the formula for $\mathcal F_{r,2}(s)$ in Proposition \ref{F_{b,2}(s)}, and Stirling's asymptotic formula for the gamma function \ref{Stirling}, gives the claimed meromorphic continuation and sharp bound
  \[
    L_h(s,\phi) \ll h^{\frac{1}{2} - \sigma + \theta + \varepsilon} (1 + |t|)^{\textup{max}\left(\frac{3}{2} - \sigma,\;1\right)} + h^{1 - \sigma + \varepsilon} (1 + |t|)^{1 - \sigma + \varepsilon}
  \]
  for the region $\frak R_\epsilon$ consisting of all  $s=\sigma+it\in\mathbb C$ satisfying $\sigma >\varepsilon$,   $|t| \rightarrow \infty,$ and $|s - \rho_k| > \varepsilon$ for all poles $\rho_k$.
  
  \vskip 5pt
  By using a convexity argument for $\sigma > \frac{1}{2}$ (on which $L_h(s,\phi)$ is holomorphic), we obtain the improved bound for $s\in\frak R_\varepsilon$ given by
  \[
    L_h(s,\phi) \ll
    \begin{cases}
      h^{\frac{1}{2} - \sigma + \theta + \varepsilon} |s|^{\frac{3}{2} - \sigma + \varepsilon} + h^{1 - \sigma + \varepsilon} |s|^{1 - \sigma + \varepsilon}  & \varepsilon < \sigma \le \frac{1}{2}, \\
      h^{(2\theta + \varepsilon) (1 - \sigma + \varepsilon)} |s|^{2 (1 - \sigma + \varepsilon)} + h^{1 - \sigma + \varepsilon} |s|^{1 - \sigma + \varepsilon} & \frac{1}{2} \le \sigma \le 1 + \varepsilon, \\
      1 & 1 + \varepsilon \le \sigma.
    \end{cases}
  \]

  To compute $\underset{s = \frac{1}{2} \pm ir_k}{\textup{Res}} L_h(s,\phi)$ we first compute $\underset{s = \frac{1}{2} \pm ir_k}{\textup{Res}} \left\langle P_h(*,s),|\phi|^2 \right\rangle$. From the proof of Theorem \ref{Thm:InnerProdBound}, such a pole comes entirely from the summand $\langle P_h(*,s),\phi_k \rangle \left\langle \phi_k,|\phi|^2 \right\rangle$. It follows from Proposition \ref{SpectralInnerProductDiscrete} that
 \begin{align*}
    \Big\langle P_h(*,s),\phi_k \Big\rangle \left\langle \phi_k,|\phi|^2 \right\rangle =  c_k(h)\cdot\frac{2\pi\sqrt{h}}{(4\pi h)^s} \frac{\Gamma\left(s - \frac{1}{2} + ir_k\right) \Gamma\left(s - \frac{1}{2} - ir_k\right)}{\Gamma(s)} \left\langle \phi_k,|\phi|^2 \right\rangle,
 \end{align*}
  so we have
  \begin{align*}
    \underset{s = \frac{1}{2} \pm ir_k}{\textup{Res}} \Big\langle P_h(*,s),|\phi|^2 \Big\rangle &= \underset{s = \frac{1}{2} \pm ir_k}{\textup{Res}} \Big\langle P_h(*,s),\phi_k \Big\rangle \left\langle \phi_k,|\phi|^2 \right\rangle\\
                                                                                                   &= 2^{\mp 2ir_k}\pi^{\frac{1}{2} \mp ir_k}h^{\mp ir_k}c_k(h) \frac{\Gamma(\pm 2ir_k)}{\Gamma\left(\frac{1}{2} \pm ir_k\right)} \left\langle \phi_k,|\phi|^2 \right\rangle \\
                                                                                                   &= 2^{-1}\pi^{\mp ir_k}h^{\mp ir_k}c_k(h) \Gamma(\pm ir_k) \left\langle \phi_k,|\phi|^2 \right\rangle \\
                                                                                                   &= \frac{1}{2} (\pi h)^{\mp ir_k} \Gamma(\pm ir_k) c_k(h) \left\langle \phi_k,|\phi|^2 \right\rangle
  \end{align*}
  and thus
  \begin{align*}
    \underset{s = \frac{1}{2} \pm ir_k}{\textup{Res}} L_h(s,\phi) &= \frac{2^{\frac{9}{4} \pm \frac{1}{2}ir_k}\pi^{\frac{1}{2} \pm ir_k}}{\Gamma\left(\frac{1}{4} \pm \frac{1}{2}ir_k\right)^2 \mathcal F_{r,2}\left(\frac{1}{2} \pm ir_k\right)} \;\underset{s = \frac{1}{2} \pm ir_k}{\textup{Res}} \Big\langle P_h(*,s),|\phi|^2 \Big\rangle 
    \\
   &
   \hskip-48pt
   = 2^{\frac{5}{4} \pm \frac{1}{2}ir_k}\pi^{\frac{1}{2}}h^{\mp ir_k} \frac{\Gamma(\pm ir_k)}{\Gamma\left(\frac{1}{4} \pm \frac{1}{2}ir_k\right)^2} \left(2^{-\frac{3}{4} \pm \frac{1}{2}ir_k} \frac{\Gamma\left(\frac{1}{4} \pm \frac{1}{2}ir_k + ir\right) \Gamma\left(\frac{1}{4} \pm \frac{1}{2}ir_k - ir\right)}{\Gamma\left(\frac{1}{2} \pm ir_k\right)}\right)^{-1} \\ &\hskip 200pt \cdot c_k(h) \left\langle \phi_k,|\phi|^2 \right\rangle 
   \\
&
\hskip-48pt
= 4\pi^{\frac{1}{2}}h^{\mp ir_k} \frac{\Gamma(\pm ir_k) \Gamma\left(\frac{1}{2} \pm ir_k\right)}{\Gamma\left(\frac{1}{4} \pm \frac{1}{2}ir_k\right)^2 \Gamma\left(\frac{1}{4} \pm \frac{1}{2}ir_k + ir\right) \Gamma\left(\frac{1}{4} \pm \frac{1}{2}ir_k - ir\right)} c_k(h) \left\langle \phi_k,|\phi|^2 \right\rangle 
\\
&
\hskip-48pt
= 2^{\frac{3}{2} \pm ir_k}h^{\mp ir_k} \frac{\Gamma(\pm ir_k) \Gamma\left(\frac{3}{4} \pm \frac{1}{2}ir_k\right)}{\Gamma\left(\frac{1}{4} \pm \frac{1}{2}ir_k\right) \Gamma\left(\frac{1}{4} \pm \frac{1}{2}ir_k + ir\right) \Gamma\left(\frac{1}{4} \pm \frac{1}{2}ir_k - ir\right)} c_k(h) \left\langle \phi_k,|\phi|^2 \right\rangle.
  \end{align*}
\end{proof}

\pagebreak

\section{\large \bf Proof of Theorem  \ref{AsymptoticFormula}} \label{smoothedsum}

In Theorem  \ref{AsymptoticFormula} we obtain the following asymptotic formula for the smoothed shifted convolution sum:

$$\underset{n \neq 0,-h}{\sum_{\sqrt{|n (n + h)|}<T} }\hskip-7pt c(n) c(n + h) \left(\textup{log}\Big(\tfrac{T}{\sqrt{|n (n + h)|}}\,\Big)\right)^{\frac{3}{2} + \varepsilon} \hskip-8pt = \;f_{{r,}h,\varepsilon}(T) T^{\frac{1}{2}} + \mathcal O\Big(h^{1-\varepsilon}\, T^{\varepsilon} + h^{1+\varepsilon}\, T^{-1 - \varepsilon}\Big).$$

\begin{proof}
The proof of Theorem \ref{AsymptoticFormula} will involve two steps. In the first step we  prove the error term $\mathcal O\left(h^{1-\varepsilon}\, T^{\varepsilon} + h^{1 + \varepsilon}T^{-1 - \varepsilon}\right)$ in the asymptotic formula and in the second step we compute the main term $f_{{r,}h,\varepsilon}(T) T^{\frac{1}{2}}$ arising from a certain sum of residues of $L_h(s,\phi).$

\vskip 10pt
\noindent 
$\underline{\text{\bf Step 1: Computing the error term:}}$
\vskip 5pt
Let $\frak R_\epsilon$ be the region consisting of all  $s=\sigma+it\in\mathbb C$ satisfying $\sigma >\varepsilon$, $|t| \rightarrow \infty,$ and $|s - \rho_k| > \varepsilon$ for all poles $\rho_k$.
It follows from Theorem \ref{MainTheorem} that for $s\in \frak R_\varepsilon$ that
\begin{equation} \label{Lh(s,phi)-Bound}
    L_h(s,\phi) \ll
    \begin{cases}
      h^{\frac{1}{2} - \sigma + \theta + \varepsilon} |s|^{\frac{3}{2} - \sigma + \varepsilon} + h^{1 - \sigma + \varepsilon} |s|^{1 - \sigma + \varepsilon} & \varepsilon < \sigma \le \frac{1}{2}, \\
      h^{(2\theta + \varepsilon) (1 - \sigma + \varepsilon)} |s|^{2 (1 - \sigma + \varepsilon)} + h^{1 - \sigma + \varepsilon} |s|^{1 - \sigma + \varepsilon} & \frac{1}{2} \le \sigma \le 1 + \varepsilon, \\
      1 & 1 + \varepsilon \le \sigma.
    \end{cases}
  \end{equation}
  
   We begin with a lemma expressing the left hand side of the asymptotic formula in Theorem \ref{AsymptoticFormula} as a Perron type integral of $L_h(s,\phi).$

\begin{lemma}\label{FirstLemma} Fix $0 < \varepsilon < \tfrac12$ and let $1+\tfrac{1}{10}\varepsilon <a < 1+\tfrac15\varepsilon.$ Then as $T\to\infty$
$$\frac{1}{2\pi i} \int\limits_{a-iT}^{a+iT} L_h(s,\phi) \frac{T^s}{s^{\frac{5}{2} + \varepsilon}} \, ds 
\; =
 \hskip-6pt
  \underset{n \neq 0,-h}{\sum_{\sqrt{|n (n + h)|}<T} } \frac{c(n) c(n + h) }{\Gamma\left(\frac{5}{2}+\varepsilon\right)}\left(\textup{log}\Big(\tfrac{T}{\sqrt{|n (n + h)|}}\Big)\right)^{\frac{3}{2} + \varepsilon} \hskip-6pt+ \;\mathcal O\left(T^{-\tfrac{1}{20}\varepsilon} \right).$$
\end{lemma}

\begin{proof}
The proof of the lemma makes use of the identity
\[
    \frac{1}{2\pi i} \int\limits_{c-i\infty}^{c+i\infty}\frac{X^s}{s^c}\, ds \; = \; \begin{cases} \frac{(\textup{log}X)^{c - 1}}{\Gamma(c)} & X>1,\\
  \;\ 0 & 0\le X <1,\end{cases}
  \]
  where  $c$ is any real number with $c \geq 1$.
  Since the Dirichlet series $L_h(s,\phi)$ converges absolutely  for $\text{\rm Re}(s) > 1,$ it immediately follows that 
  
  \pagebreak
  
 \begin{align*}
    \frac{1}{2\pi i} \int\limits_{a-i\infty}^{a+i\infty} L_h(s,\phi) \frac{T^s}{s^{\frac{5}{2} + \varepsilon}} \, ds &=\sum_{n \neq 0,-h} c(n) c(n + h) \cdot \frac{1}{2\pi i} \int\limits_{a-i\infty}^{a+i\infty}\;\frac{\Big(\frac{T}{\sqrt{|n (n + h)|}}\Big)^s}{s^{\frac{5}{2} + \varepsilon}}\, ds \\
                                                                                                     &= \underset{n \neq 0,-h}{\sum_{\sqrt{|n (n + h)|}<T} } c(n) c(n + h)\; \frac{\left(\textup{log}\Big(\tfrac{T}{\sqrt{|n (n + h)|}}\Big)\right)^{\frac{3}{2} + \varepsilon}}{ \Gamma\left(\frac{5}{2}+\varepsilon\right)  }.                                                                                                  \end{align*}
  
  Next, we show that the {\it ``tail end''} integrals  given by {\small $$ \frac{1}{2\pi i} \int\limits_{a+iT}^{a+i\infty} L_h(s,\phi) \frac{T^s}{s^{\frac{5}{2} + \varepsilon}} \, ds, \qquad  \frac{1}{2\pi i} \int\limits_{a-i\infty}^{a-iT} L_h(s,\phi) \frac{T^s}{s^{\frac{5}{2} + \varepsilon}} \, ds,$$}
  are small and constitute the error term in the lemma.
  
   Since $L_h(s,\phi)$ is bounded for $\textup{Re}(s)=a$,
  \begin{align*}
    \int\limits_{a + iT}^{a + i\infty} L_h(s,\phi) \frac{T^s}{s^{\frac{5}{2} + \varepsilon}} ds &\ll  \frac{T^a}{a^{\frac{3}{2}}} \int\limits_{T/a}^\infty
    \frac{dt}{\left(t^2+1\right)^{1+\frac12\varepsilon}}
    \ll T^{a-1-\frac14\varepsilon} \ll \;T^{-\tfrac{1}{20}\varepsilon}.
   \end{align*}                                                                                  
 A similar bound may be obtained for the other {\it ``tail end''} integral.
 \end{proof}
 
 To obtain the right hand side of the asymptotic formula in Theorem \ref{AsymptoticFormula} we shift the line of integration of the Perron integral to $\text{\rm Re}(s) =\tfrac{1}{10} \varepsilon.$ This has to be done carefully since $L_h(s,\phi)$ has infinitely many poles on the line $\text{\rm Re}(s) = \tfrac12.$ The key point is to choose $T$ so that $\tfrac12\pm iT$ is not too close to any poles of $L_h(s,\phi).$  Let $N(T)$ denote the number of $r_k$  such that $0<r_k\le T.$ It is known (see \cite{MR3246823}) that
    \begin{equation} \label{N(T)}
    N(T) =\frac{T^2}{12}-\frac{2T\log T}{\pi} + \frac{T}{\pi}\Big(2+ \log(\pi/2)\Big) +\mathcal O\left( \frac{T}{\log T}  \right).
    \end{equation}
    It follows that $N(T+1) -N(T) \,\sim\, \frac{T}{6}.$ Therefore, among the values of $T\le r_k \le T+1$   there must be a gap of length $\gg T^{-1}.$ Hence, by an appropriate choice of $T$ we can ensure that
    $|r_k - T| \gg T^{-1}$ for all $k$. We assume from now on that $T$ is chosen this way.

\vskip 10pt
It follows from Cauchy's residue theorem that
\begin{equation} \label{CauchyResidueTheorem}
\frac{1}{2\pi i} \int\limits_{a-iT}^{a+iT} L_h(s,\phi) \frac{T^s}{s^{\frac{5}{2} + \varepsilon}} \, ds = \frac{-1}{2\pi i}\left(\int\limits_{a+iT}^{\frac{\varepsilon}{10}+iT} +
\int\limits_{\frac{\varepsilon}{10}+iT}^{\frac{\varepsilon}{10}-iT}+
\int\limits_{\frac{\varepsilon}{10}-iT}^{a-iT}\right) L_h(s,\phi) \frac{T^s}{s^{\frac{5}{2} + \varepsilon}} \, ds \; + \; \mathcal R 
\end{equation}
where $\mathcal R$ denotes the sum of the residues.   It is immediate from   (\ref{Lh(s,phi)-Bound})  that the above integrals converge absolutely.  

\vskip 10pt
  To complete the first step in the proof of Theorem \ref{AsymptoticFormula} it remains to obtain bounds for each  of the three integrals on the right side of   (\ref{CauchyResidueTheorem}) which constitutes the error term in the asymptotic formula in Theorem  \ref{AsymptoticFormula}.                                                                           
  This will be established in the following  lemmas.                                                                                                   

 \begin{lemma} \label{SecondLemma}
   Fix $0 < \varepsilon < \tfrac12$. Then \,
    $
\boxed{\frac{1}{2\pi i} \int\limits_{\frac{\varepsilon}{10}+iT}^{\frac{\varepsilon}{10}-iT}  L_h(s,\phi) \frac{T^s}{s^{\frac{5}{2} + \varepsilon}} \, ds \; \ll \;h^{1-\frac{\varepsilon}{20}}  \,T^{\frac{\varepsilon}{10}}.}
 $
  \end{lemma} 
  \begin{proof}
 It follows from (\ref{Lh(s,phi)-Bound}) (where we use $\frac{\varepsilon}{20}$ in place of $\varepsilon$) that
 \begin{align*}
 \frac{1}{2\pi i} \int\limits_{\frac{\varepsilon}{10}+iT}^{\frac{\varepsilon}{10}-iT}  L_h(s,\phi) \frac{T^s}{s^{\frac{5}{2} + \varepsilon}} \, ds &\ll\;
  \int\limits_{\frac{\varepsilon}{10}-iT}^{\frac{\varepsilon}{10}+iT} \Big( h^{\frac{1}{2} - \sigma + \theta +\frac{\varepsilon}{20}} |s|^{\frac{3}{2} - \sigma +\frac{\varepsilon}{20}} + h^{1 - \sigma + \frac{\varepsilon}{20}} |s|^{1 - \sigma + \frac{\varepsilon}{20}}\Big) \frac{T^{\frac{\varepsilon}{10}}}{|s|^{\frac{5}{2} + \varepsilon}} \, |ds|\\
  &\ll\; T^{\frac{\varepsilon}{10}}\int\limits_{-T}^{T} \frac{h^{\frac{1}{2} +\theta -\frac{\varepsilon}{20}} (1 + |t|)^{\frac{3}{2} - \frac{\varepsilon}{20}} + h^{1 - \frac{\varepsilon}{20}} (1 + |t|)^{1 - \frac{\varepsilon}{20}}}{(1+|t|)^{\frac{5}{2} + \varepsilon}} \; |dt|\\
  & \ll\;  h^{1 - \frac{\varepsilon}{20}} \, T^{\frac{\varepsilon}{10}} 
   \end{align*}
 since we may take $\theta =7/64$ (see Definition \ref{theta}).
\end{proof}

It remains to consider the other two integrals along the horizontal line segments with imaginary part $\pm T$ given on the right hand side of  (\ref{CauchyResidueTheorem}).
  To bound these integrals  we first break them up as follows. Let
   \begin{align} \label{HorizontalIntegrals}
   \frac{1}{2\pi i} \int\limits_{\frac{\varepsilon}{10} \pm iT}^{a \pm iT} \hskip-5pt L_h(s,\phi) \frac{T^s}{s^{\frac{5}{2} + \varepsilon}}\, ds 
   &
    \;= \;
  \frac{1}{2\pi i} \left(\;  \int\limits_{\frac{\varepsilon}{10} \pm iT}^{\frac12-\varepsilon \pm iT}+ \int\limits_{\frac12-\varepsilon \pm iT}^{\frac12+\varepsilon \pm iT} +\int\limits_{\frac12+\varepsilon \pm iT}^{a \pm iT}\;\right) L_h(s,\phi) \frac{T^s}{s^{\frac{5}{2} + \varepsilon}}\, ds. 
   \end{align}
   
   The first and third integrals on the right hand side of (\ref{HorizontalIntegrals}) consist of line segments that are a distance greater than $\varepsilon$ from any pole $\tfrac12 \pm ir_k$ of $L_h(s,\phi).$ We bound them in the following two lemmas.
   
   \begin{lemma} \label{ThirdLemma} Fix $0<\varepsilon <\tfrac12.$ Then    $$\boxed{\frac{1}{2\pi i} \int\limits_{\frac{\varepsilon}{10} \pm iT}^{\frac12-\varepsilon+ \pm iT}\hskip-5pt L_h(s,\phi) \frac{T^s}{s^{\frac{5}{2} + \varepsilon}}\; ds \;\ll \;  h^{\frac12+\theta+\varepsilon} \,T^{-\frac12-\varepsilon} + h^{1+\varepsilon}\, T^{-1 - \varepsilon}
.}$$
   \end{lemma}
   
   \pagebreak
   
   \begin{proof} It follows from (\ref{Lh(s,phi)-Bound}) that
   \begin{align*}
   \frac{1}{2\pi i} \int\limits_{\frac{\varepsilon}{10} \pm iT}^{\frac12-\varepsilon+ \pm iT}\hskip-5pt L_h(s,\phi) \frac{T^s}{s^{\frac{5}{2} + \varepsilon}}\; ds & \ll \; \int\limits_{\frac{\varepsilon}{10} \pm iT}^{\frac12-\varepsilon+ \pm iT}\hskip-5pt \Big(  h^{\frac{1}{2} - \sigma + \theta + \varepsilon} |s|^{\frac{3}{2} - \sigma + \varepsilon} + h^{1 - \sigma + \varepsilon} |s|^{1 - \sigma + \varepsilon}\Big) \frac{T^{\frac12-\varepsilon}}{|s|^{\frac{5}{2} + \varepsilon}}\; |ds|\\
   & \ll \; \int\limits_{\frac{\varepsilon}{10}}^{\frac12-\varepsilon}
   \Big(  h^{\frac{1}{2} - \sigma + \theta + \varepsilon} \,T^{\frac{3}{2} - \sigma + \varepsilon} + h^{1 - \sigma + \varepsilon} T^{1 - \sigma + \varepsilon}\Big) \frac{T^{\frac12-\varepsilon}}{T^{\frac{5}{2} + \varepsilon}}\; |d\sigma|\\
   & \ll \; h^{\frac12+\theta+\varepsilon} \,T^{-\frac12-\varepsilon} + h^{1+\varepsilon}\, T^{-1 - \varepsilon}.
   \end{align*}
   \end{proof}

\begin{lemma} \label{FourthLemma} Fix $0<\varepsilon <\tfrac12.$ Then for $1+\tfrac{1}{10}\varepsilon <a < 1+\tfrac15\varepsilon$, we have
$$
\boxed{\frac{1}{2\pi i} \int\limits_{\frac12+\varepsilon \pm iT}^{a \pm iT}\hskip-5pt L_h(s,\phi) \frac{T^s}{s^{\frac{5}{2} + \varepsilon}}\, ds \;\ll \;  h^{\theta+\frac12\varepsilon}\,     T^{-\frac{1}{2} - \frac{4}{5} \varepsilon}\; + \; h^\frac12 \, T^{-1 - \frac{4}{5} \varepsilon}.}
$$
\end{lemma} 
\begin{proof} We can  estimate this integral with the bound for $L_h(s,\phi)$ given in (\ref{Lh(s,phi)-Bound}). Since $1+\tfrac{1}{10}\varepsilon <a < 1+\tfrac15\varepsilon$ we have
\begin{align*}
\frac{1}{2\pi i} \int\limits_{\frac12+\varepsilon \pm iT}^{a \pm iT}\hskip-5pt L_h(s,\phi) \frac{T^s}{s^{\frac{5}{2} + \varepsilon}}\, ds \;
& \ll \; 
\int\limits_{\frac12+\varepsilon}^{1+\frac15\varepsilon} \left(h^{(2\theta+\varepsilon)(1-\sigma+\varepsilon)} T^{2 (1 - \sigma + \varepsilon)}+ \;h^{1-\sigma+\varepsilon} T^{1 - \sigma + \varepsilon}\right) \frac{T^{1 + \frac{1}{5} \varepsilon}}{T^{\frac{5}{2} + \varepsilon}}\; d\sigma\\
& \ll \; h^{\theta+\frac12\varepsilon}\,     T^{-\frac{1}{2} - \frac{4}{5} \varepsilon}\; + \; h^\frac12 \, T^{-1 - \frac{4}{5} \varepsilon} .
\end{align*}
\end{proof}

Finally, we consider the second integral on the right hand side of  equation (\ref{HorizontalIntegrals}).

 \begin{lemma} \label{FifthLemma} Fix $0<\varepsilon <\tfrac12.$ Then we have the bound
   \[
     \boxed{\frac{1}{2\pi i} \int\limits_{\frac12-\varepsilon \pm iT}^{\frac12+\varepsilon \pm iT} \hskip-5pt L_h(s,\phi) \frac{T^s}{s^{\frac{5}{2} + \varepsilon}}\, ds\; \ll \;h^{\theta + \varepsilon} T^{-\frac{1}{2} + \varepsilon} + {h^{\frac{1}{2} + \varepsilon} T^{-\frac{3}{2} + \varepsilon}}.}
   \]
 \end{lemma}   
 \begin{proof}
   Because we can only guarantee $|T - r_k| \gg T^{-1}$ as $T \rightarrow \infty$, for any $\varepsilon > 0$, we have that for all sufficiently large $T$, the line of integration is not contained entirely in $\mathfrak R_{\varepsilon}$, so we cannot use the bound for $L_h(s,\phi)$ found before. Instead, we work as follows.

   From the spectral expansion of $\left\langle P_h(*,s),|\phi|^2 \right\rangle$ given in (\ref{SpectralExpansion}), the bound for the continuous spectrum part (Proposition \ref{ContinuousSpectrumBound}), and the asymptotic formula relating it to $L_h(s,\phi)$ (Proposition \ref{Prop:InnerProductBound}), we have
\begin{align*}
L_h(s,\phi) & =  \frac{2^{2 + \frac{1}{2}s}\pi^s}{\Gamma\left(\frac{1}{2}s\right)^2 \mathcal F_{r,2}(s)}  \sum\limits_{k = 1}^{\infty} \Big\langle P_h(*,s),\phi_k \Big \rangle \left\langle \phi_k,|\phi|^2 \right\rangle \;+\; \mathcal O\left(h^{\frac{1}{2} - \sigma + \varepsilon} (1+|t|)^{\frac{1}{2}+\varepsilon}  \,\right)\\
&
 \hskip 200pt+  \mathcal O\left(h^{1 - \sigma + \varepsilon} (1 + |t|)^{1 - \sigma+\varepsilon}  \right).
\end{align*}
It immediately follows that the integral $\frac{1}{2\pi i} \int\limits_{\frac12-\varepsilon \pm iT}^{\frac12+\varepsilon \pm iT} \hskip-5pt L_h(s,\phi) \frac{T^s}{s^{\frac{5}{2} + \varepsilon}}\, ds$ is given by
\begin{align} \label{PerronIntegral}
\frac{1}{2\pi i} \int\limits_{\frac12-\varepsilon \pm iT}^{\frac12+\varepsilon \pm iT} \frac{2^{2 + \frac{1}{2}s}\pi^s}{\Gamma\left(\frac{1}{2}s\right)^2 \mathcal F_{r,2}(s)}  \sum\limits_{k = 1}^{\infty} \Big\langle P_h(*,s),\phi_k \Big \rangle \left\langle \phi_k,|\phi|^2 \right\rangle \frac{T^s}{s^{\frac{5}{2}+\varepsilon}}\, ds  \;+\; {\mathcal O\left(h^{\frac{1}{2} + \varepsilon} T^{-\frac{3}{2} + \varepsilon}\right)}
\end{align}

because

\begin{align*}
& \int\limits_{\frac{1}{2} - \varepsilon \pm iT}^{\frac{1}{2} + \varepsilon \pm iT} h^{\frac{1}{2} - \sigma + \varepsilon} (1 + |t|)^{\frac{1}{2} + \varepsilon} \frac{T^s}{s^{\frac{5}{2} + \varepsilon}} ds \; \ll \; h^{\frac{1}{2} + \varepsilon} T^{\frac{1}{2} + \varepsilon} \int\limits_{\frac{1}{2} - \varepsilon}^{\frac{1}{2} + \varepsilon} \left(\tfrac{T}{h}\right)^{\sigma} |\sigma \pm iT|^{-\frac{5}{2} - \varepsilon} d\sigma \\
                                                                                                                                                                                                              &
                                                                                                                                                                                                              \hskip 30pt \ll h^{\frac{1}{2} + \varepsilon} T^{-2} \int\limits_{\frac{1}{2} - \varepsilon}^{\frac{1}{2} + \varepsilon} \left(\tfrac{T}{h}\right)^{\sigma} d\sigma                                                                                                                                                                                                       \;=\;                                                                                                                                                                                                               h^{\frac{1}{2} + \varepsilon} T^{-2} \left(\textup{log}\left(\tfrac{T}{h}\right)\right)^{-1} \left(\left(\tfrac{T}{h}\right)^{\frac{1}{2} + \varepsilon} - \left(\tfrac{T}{h}\right)^{\frac{1}{2} - \varepsilon}\right)\\
                                                                                                                                                                                                              &\hskip30pt \; \ll \; h^{\varepsilon} T^{-\frac{3}{2} + \varepsilon},
\end{align*}
and
\begin{align*}
  \int\limits_{\frac{1}{2} - \varepsilon \pm iT}^{\frac{1}{2} + \varepsilon \pm iT} h^{1 - \sigma + \varepsilon} (1 + |t|)^{1 - \sigma + \varepsilon} \frac{T^s}{s^{\frac{5}{2} + \varepsilon}} ds & \ll  h^{1 + \varepsilon} T^{1 + \varepsilon} \int\limits_{\frac{1}{2} - \varepsilon}^{\frac{1}{2} + \varepsilon} h^{-\sigma} |\sigma \pm iT|^{-\frac{5}{2} - \varepsilon} d\sigma \; \ll \; h^{1 + \varepsilon} T^{-\frac{3}{2}} \int\limits_{\frac{1}{2} - \varepsilon}^{\frac{1}{2} + \varepsilon} h^{-\sigma} d\sigma \\& = T^{-\frac{3}{2}} (\textup{log} h)^{-1} \left(h^{-\frac{1}{2} + \varepsilon} - h^{-\frac{1}{2} - \varepsilon}\right)\\
  &
   \; \ll \; h^{\frac{1}{2} + \varepsilon} T^{-\frac{3}{2}}.
\end{align*}

 \pagebreak

   Next, we have
 \begin{align*}
 &
 \frac{2^{2 + \frac{1}{2}s}\pi^s}{\Gamma\left(\frac{1}{2}s\right)^2 \mathcal F_{r,2}(s)} \sum_{k = 1}^{\infty} \Big\langle P_h(*,s),\phi_k \Big\rangle \big\langle \phi_k,|\phi|^2 \big\rangle \\
 &\hskip40pt= \; 
  \sum_{k=1}^\infty 2\pi^{s - 1}c_k(h) \left\langle \phi_k,|\phi|^2 \right\rangle \frac{\Gamma\left(\frac{1}{2}s + \frac{1}{2}\right)^2 \Gamma\left(s - \frac{1}{2} + ir_k\right) \Gamma\left(s - \frac{1}{2} - ir_k\right)}{\Gamma(s)^2 \Gamma\left(\frac{1}{2}s + ir\right) \Gamma\left(\frac{1}{2}s - ir\right)}.
 \end{align*} 
 We now analyze the gamma function terms in the summand for a single $k$.   Let 
 $s = \sigma + iT$ (with $\tfrac12-\varepsilon\le \sigma\le\tfrac12-\varepsilon $). The case $s = \sigma-iT$ is similar and is omitted. Because $T \rightarrow \infty$ and $r$ is fixed, we can assume that $\frac{T}{2} > r$. We rewrite the  $k^{th}$ gamma function term as 
   \begin{align} \label{GammaTerms}
   G_k(\sigma,T) & := \frac{\Gamma\left(\frac{\sigma+1}{2} + i \frac{T}{2}\right)^2 \Gamma\big(\sigma - \frac{1}{2} + i (T + r_k)\big) \Gamma\big(\sigma - \frac{1}{2} + i (T - r_k)\big)}{\Gamma\big(\sigma + iT\big)^2 \Gamma\left(\frac{\sigma}{2} + i \left(\frac{T}{2} + r\right)\right) \Gamma\left(\frac{\sigma}{2} + i \left(\frac{T}{2} - r\right)\right)}\\
   &\; \ll T^{2-2\sigma} e^{\pi T}\cdot\Big|\Gamma\Big(\sigma - \tfrac{1}{2} + i (T + r_k)\Big) \Gamma\Big(\sigma - \tfrac{1}{2} + i (T - r_k)\Big)\Big|.
   \nonumber
  \end{align}
 after applying Stirling's asymptotic formula  as given in Proposition \ref{Stirling}.

\vskip 8pt
In order to evaluate the integral in (\ref{PerronIntegral}) we shall consider three cases involving the sum over $r_k$.
\vskip 15pt
\noindent
$\underline{\text{{\bf Case 1:} Assume $|T - r_k| < 1$.}}$
\vskip 10pt
 Note that $|T - r_k| < 1$  occurs for a constant times $T$ values of $r_k$. To compute $G_k(\sigma,T)$ for $|T - r_k| < 1$, we use the fact that $\Gamma(z) = z^{-1} + \mathcal O(1)$ (for $|z| \le 2$) for the term $\Gamma\Big(\sigma - \frac{1}{2}\ + i (T - r_k)\Big)$ and Stirling's asymptotic formula given in Proposition \ref{Stirling} for the term $\Gamma\big(\sigma - \tfrac{1}{2} + i (T + r_k)\big)$.  This yields
\begin{align*}
\left|G_k(\sigma,T)\right| & \ll \frac{ T^\sigma \big(T+r_k\big)^{\sigma-1} \left(\sigma-\tfrac12 +i(T-r_k)   \right)^{-1}  }{ T^{2\sigma-1} \left(\frac{T}{2}+r  \right)^{\frac{\sigma-1}{2}}  \left(\frac{T}{2}-r  \right)^{\frac{\sigma-1}{2}}   }\cdot e^{-\frac{\pi}{2} (r_k - T)}\\
& \ll \left(\sigma-\tfrac12 +i(T-r_k)   \right)^{-1} T^{1-\sigma}\end{align*}  
 since   $e^{\frac{\pi}{2} |r_k - T|} \ll 1.$

 \vskip 8pt
    It follows from (\ref{N(T)}) that $|T - r_k| > cT^{-1}$ for some positive constant $c$. This yields the  upper bound 
  \[\left|G_k(\sigma,T)\right| \ll
    \left|\sigma - \tfrac{1}{2} + icT^{-1}\right|^{-1} T^{1-\sigma}.  \]

 As $k \rightarrow \infty$, $c_k(h) \ll h^{\theta + \varepsilon} (\log r_k)\, e^{\frac{\pi}{2}r_k}$ (by Lemma \ref{c_k(h)Bound}) and $\left\langle \phi_k,|\phi|^2 \right\rangle \ll e^{-\frac{\pi}{2}r_k}$ (by  \cite{BR1999}).  It follows that
 
\pagebreak 
 
  \begin{align}
  &
  \frac{1}{2\pi i} \int\limits_{\frac12-\varepsilon \pm iT}^{\frac12+\varepsilon \pm iT} \frac{2^{2 + \frac{1}{2}s}\pi^s}{\Gamma\left(\frac{1}{2}s\right)^2 \mathcal F_{r,2}(s)} 
  \underset{|T-r_k|<1}{\sum\limits_{k = 1}^{\infty}} \Big\langle P_h(*,s),\phi_k \Big \rangle \left\langle \phi_k,|\phi|^2 \right\rangle \frac{T^s}{s^{\frac{5}{2}+\varepsilon}}\, ds \phantom{xxxxxxxxxxxx}
  \nonumber
  \\
  &
  \hskip 50pt
  \ll \; \int\limits_{\frac12-\varepsilon}^{\frac12+\varepsilon}  \sum\limits_{|T-r_k|<1} \left|c_k(h) \left\langle \phi_k,|\phi|^2 \right\rangle\right|\cdot \big|G_k(\sigma,T)\big| \,T^{-2}\, d\sigma
  \label{Case1Bound}
  \nonumber
  \\
  &
  \hskip 50pt
  \ll \; h^{\theta + \varepsilon} \sum\limits_{|T-r_k|<1} (\log r_k) \,  T^{-\frac{3}{2}+\varepsilon}\int\limits_{\frac12-\varepsilon}^{\frac12+\varepsilon} \left|\sigma - \tfrac{1}{2} + icT^{-1}\right|^{-1}\,d\sigma
  \nonumber
  \\
  &
  \hskip 50pt
  \ll h^{\theta + \varepsilon}T^{-\frac{1}{2} + \varepsilon} (\log T) \int\limits_{\frac12-\varepsilon}^{\frac12+\varepsilon}\left((\sigma - \tfrac{1}{2})^2 +\tfrac{c^2}{T^2}\right)^{-\frac12}\,d\sigma.
  \\
  &
  \hskip 50pt
  \ll h^{\theta + \varepsilon}T^{-\frac{1}{2} + \varepsilon} (\log T)\left(\; \int\limits_{\frac12-\varepsilon}^{\frac12-\frac{c}{T}} + \int\limits_{\frac12-\frac{c}{T}}^{\frac12+\frac{c}{T}}+\int\limits_{\frac12+\frac{c}{T}}^{\frac12+\varepsilon}\;\right)\left((\sigma - \tfrac{1}{2})^2 +\tfrac{c^2}{T^2}\right)^{-\frac12}\,d\sigma 
  \nonumber\\
  &
  \hskip 50pt
 \ll h^{\theta + \varepsilon}T^{-\frac{1}{2} + \varepsilon}(\log T) \left(\;\int\limits_{\frac12-\varepsilon}^{\frac12-\frac{c}{T}} (\tfrac12-\sigma)^{-1}\,d\sigma + \int\limits_{\frac12-\frac{c}{T}}^{\frac12+\frac{c}{T}} \frac{T}{c}\; d\sigma + \int\limits_{\frac12+\frac{c}{T}}^{\frac12+\varepsilon} (\sigma-\tfrac12)^{-1} \,d\sigma  \right)
 \nonumber\\
 &
 \hskip 50pt
 \ll h^{\theta + \varepsilon}T^{-\frac{1}{2} + \varepsilon}\, (\log T)^2.
 \nonumber
  \end{align}

\vskip 15pt  
  
  $\underline{\text{{\bf Case 2:} Assume $0<r_k<T-1$.}}$
\vskip 10pt
 Under the assumption that $0<r_k<T-1$ we use similar methods as were used in case 1. It follows by Stirling's asymptotic formula  that the function $G_k(\sigma,T)$ defined  in (\ref{GammaTerms}) has the following bound:
 \begin{align*}
 G_k(\sigma,T) & \ll T^{2-2\sigma} e^{\pi T}\cdot\Big|\Gamma\big(\sigma - \tfrac{1}{2} + i (T + r_k)\big) \Gamma\big(\sigma - \tfrac{1}{2} + i (T - r_k)\big)\Big|
 \\
 &
 \ll T^{2-2\sigma} (T+r_k)^{\sigma-1} (T-r_k)^{\sigma-1}e^{-\pi T}
  \ll e^{-\pi T}.
 \end{align*}
 Returning to the integral we need to bound to prove Lemma  \ref{FifthLemma} we apply the above bound to obtain
 
 \pagebreak
 \begin{align}
  \frac{1}{2\pi i} \int\limits_{\frac12-\varepsilon \pm iT}^{\frac12+\varepsilon \pm iT}  & \frac{2^{2 + \frac{1}{2}s}\pi^s}{\Gamma\left(\frac{1}{2}s\right)^2 \mathcal F_{r,2}(s)} \;
 \underset{ 0<r_k<T-1}{\sum} \Big\langle P_h(*,s),\phi_k \Big \rangle \left\langle \phi_k,|\phi|^2 \right\rangle \frac{T^s}{s^{\frac{5}{2}+\varepsilon}}\, ds
  \nonumber
  \\
  &
  \hskip 40pt
  \ll \; \int\limits_{\frac12-\varepsilon}^{\frac12+\varepsilon}   \underset{ 0<r_k<T-1}{\sum} \left|c_k(h) \left\langle \phi_k,|\phi|^2 \right\rangle\right|\cdot \big|G_k(\sigma,T)\big| \,T^{-2}\, d\sigma
  \label{Case2Bound}
  \\
  &
  \hskip 40pt
  \ll h^{\theta + \varepsilon} \; \underset{ 0<r_k<T-1}{\sum} (\log r_k) \;T^{-2} e^{-\pi T}
  \ll h^{\theta + \varepsilon} (\log T)\, e^{-\pi T}.
  \nonumber
  \end{align}

 $\underline{\text{{\bf Case 3:} Assume $r_k>T+1$.}}$
 \vskip 10pt
 Finally, we assume $r_k > T+1.$ In this case we have
 \begin{align*}
 G_k(\sigma,T) & \ll T^{2-2\sigma} e^{\pi T}\cdot\Big|\Gamma\big(\sigma - \tfrac{1}{2} + i (T + r_k)\big) \Gamma\big(\sigma - \tfrac{1}{2} + i (T - r_k)\big)\Big|
 \\
 &
 \ll T^{2-2\sigma} r_k^{\sigma-1} (r_k-T)^{\sigma-1}e^{-\pi r_k}.\\
  \end{align*}  
  \vskip-15pt\noindent
  It follows that
  \begin{align*}
  &
  \frac{1}{2\pi i} \int\limits_{\frac12-\varepsilon \pm iT}^{\frac12+\varepsilon \pm iT} \frac{2^{2 + \frac{1}{2}s}\pi^s}{\Gamma\left(\frac{1}{2}s\right)^2 \mathcal F_{r,2}(s)} \;
 \underset{ 0<r_k<T-1}{\sum} \Big\langle P_h(*,s),\phi_k \Big \rangle \left\langle \phi_k,|\phi|^2 \right\rangle \frac{T^s}{s^{\frac{5}{2}+\varepsilon}}\, ds
  \\
  &
  \hskip 90pt
  \ll h^{\theta + \varepsilon} \; \int\limits_{\frac12-\varepsilon}^{\frac12+\varepsilon}   \underset{ r_k>T+1}{\sum} \left|c_k(h) \left\langle \phi_k,|\phi|^2 \right\rangle\right|\cdot \big|G_k(\sigma,T)\big| \,T^{-2}\, d\sigma
  \\
  &
  \hskip90pt
  \ll h^{\theta + \varepsilon}T^{-1+\varepsilon} \int\limits_{\frac12-\varepsilon}^{\frac12+\varepsilon} \underset{ r_k>T+1}{\sum}  (\log r_k) r_k^{\sigma-1} (r_k-T)^{\sigma-1} \,e^{-\pi r_k}\, d\sigma.
  \end{align*}

 Continuing the computation we obtain
 
\vskip-10pt 
 
 \begin{align}\label{Case3Bound}
 &
 \frac{1}{2\pi i} \int\limits_{\frac12-\varepsilon \pm iT}^{\frac12+\varepsilon \pm iT} \frac{2^{2 + \frac{1}{2}s}\pi^s}{\Gamma\left(\frac{1}{2}s\right)^2 \mathcal F_{r,2}(s)} \;
 \underset{ 0<r_k<T-1}{\sum} \Big\langle P_h(*,s),\phi_k \Big \rangle \left\langle \phi_k,|\phi|^2 \right\rangle \frac{T^s}{s^{\frac{5}{2}+\varepsilon}}\, ds\\
  &
  \hskip 60 pt
  \ll h^{\theta + \varepsilon}T^{-1+\varepsilon} \sum_{j=0}^\infty \;\sum_{2^jT<r_k \le 2^{j+1} T}  (\log r_k) r_k^{-\frac12+\varepsilon}  \,e^{-\pi r_k}
  \nonumber\\
   &
  \hskip 60 pt
  \ll h^{\theta + \varepsilon}T^{-1+\varepsilon} \sum_{j=0}^\infty 2^{2j+2} T^2 \cdot (\log 2^{j+1} T) \left(2^j T\right)^{-\frac12+\varepsilon} e^{-\pi 2^j T  }
  \nonumber\\
  &
  \hskip 60 pt
  \ll h^{\theta + \varepsilon}\,T^{-\frac{1}{2}}.
  \nonumber
  \end{align}
  
   The proof of Lemma \ref{FifthLemma} immediately follows from the bounds obtained {in (\ref{PerronIntegral}) and in cases 1, 2, and 3} above explicitly given in (\ref{Case1Bound}), (\ref{Case2Bound}), (\ref{Case3Bound}).
   \end{proof}
 
 This completes  the first step of the proof of the asymptotic formula that is given in Theorem \ref{AsymptoticFormula}. The error term in the  asymptotic formula is the largest of the error terms computed in Lemmas \ref{FirstLemma}, \ref{SecondLemma}, \ref{ThirdLemma}, \ref{FourthLemma}, and \ref{FifthLemma}.

\vskip 10pt
\noindent
$ \underline{\text{\bf Step 2: Computing the sum of residues $\mathcal R$ on the right hand side of  (\ref{CauchyResidueTheorem}).}}$
\vskip 5pt
  Finally, we consider the sum of all residues  of $L_h(s,\phi) \frac{T^s}{s^{\frac{5}{2} + \varepsilon}}$ on the right hand side of  (\ref{CauchyResidueTheorem}). As this consists precisely of the possible simple poles at $s = \frac{1}{2} \pm ir_k$, this sum is
  \begin{align*}
    &\sum_{k = 1}^{\infty} \left(\underset{s = \frac{1}{2} + ir_k}{\textup{Res}}\left(L_h(s,\phi) \frac{T^s}{s^{\frac{5}{2} + \varepsilon}}\right) + \underset{s = \frac{1}{2} - ir_k}{\textup{Res}}\left(L_h(s,\phi) \frac{T^s}{s^{\frac{5}{2} + \varepsilon}}\right)\right) \\
    &\hskip 50pt = T^{\frac{1}{2}} \sum_{k = 1}^{\infty} \left(\frac{T^{ir_k}}{\left(\frac{1}{2} + ir_k\right)^{\frac{5}{2} + \varepsilon}} \underset{s = \frac{1}{2} + ir_k}{\textup{Res}} L_h(s,\phi) + \frac{T^{-ir_k}}{\left(\frac{1}{2} - ir_k\right)^{\frac{5}{2} + \varepsilon}} \underset{s = \frac{1}{2} - ir_k}{\textup{Res}} L_h(s,\phi)\right).
  \end{align*}

  Recall from (\ref{LhFunction}) that $L_h(s,\phi)$ equals $$\frac{2^{2 + \frac{1}{2}s}\pi^s}{\Gamma\left(\frac{1}{2}s\right)^2 \mathcal F_{r,2}(s)} \Big\langle P_h(*,s),|\phi|^2 \Big\rangle$$ plus a holomorphic error term.
To compute $\underset{s = \frac{1}{2} \pm ir_k}{\textup{Res}} L_h(s,\phi)$ we first compute $$\underset{s = \frac{1}{2} \pm ir_k}{\textup{Res}} \Big\langle P_h(*,s),|\phi|^2 \Big\rangle.$$ From the proof of Theorem \ref{Thm:InnerProdBound}, such a pole comes entirely from the summand $$\Big\langle P_h(*,s),\phi_k \Big\rangle \left\langle \phi_k,|\phi|^2 \right\rangle.$$
\vskip 5pt  
  Now, by Proposition \ref{SpectralInnerProductDiscrete},
  \[
   \Big \langle P_h(*,s),\phi_k \Big \rangle \left\langle \phi_k,|\phi|^2 \right\rangle = c_k(h)\cdot\frac{2\pi\sqrt{h}}{(4\pi h)^s} \frac{\Gamma\left(s - \frac{1}{2} + ir_k\right) \Gamma\left(s - \frac{1}{2} - ir_k\right)}{\Gamma(s)} \left\langle \phi_k,|\phi|^2 \right\rangle,
  \]
  so we have
  \begin{align*}
    \underset{s = \frac{1}{2} \pm ir_k}{\textup{Res}} \Big\langle P_h(*,s),|\phi|^2 \Big\rangle &= \underset{s = \frac{1}{2} \pm ir_k}{\textup{Res}} \Big\langle P_h(*,s),\phi_k \Big\rangle \left\langle \phi_k,|\phi|^2 \right\rangle\\
                                                                                                   &= 2^{\mp 2ir_k}\pi^{\frac{1}{2} \mp ir_k}h^{\mp ir_k}c_k(h) \frac{\Gamma(\pm 2ir_k)}{\Gamma\left(\frac{1}{2} \pm ir_k\right)} \left\langle \phi_k,|\phi|^2 \right\rangle \\
                                                                                                   &= 2^{-1}\pi^{\mp ir_k}h^{\mp ir_k}c_k(h) \Gamma(\pm ir_k) \left\langle \phi_k,|\phi|^2 \right\rangle \\
                                                                                                   &= \frac{1}{2} (\pi h)^{\mp ir_k} \Gamma(\pm ir_k) c_k(h) \left\langle \phi_k,|\phi|^2 \right\rangle,
  \end{align*}
  and therefore
  \begin{align*}
    \underset{s = \frac{1}{2} \pm ir_k}{\textup{Res}} L_h(s,\phi) &= \frac{2^{\frac{9}{4} \pm \frac{1}{2}ir_k}\pi^{\frac{1}{2} \pm ir_k}}{\Gamma\left(\frac{1}{4} \pm \frac{1}{2}ir_k\right)^2 \mathcal F_{r,2}\left(\frac{1}{2} \pm ir_k\right)} \underset{s = \frac{1}{2} \pm ir_k}{\textup{Res}} \Big\langle P_h(*,s),|\phi|^2 \Big\rangle \\
                                                                  &
                                                                  \hskip -60pt
                                                                  = 2^{\frac{5}{4} \pm \frac{1}{2}ir_k}\pi^{\frac{1}{2}}h^{\mp ir_k} \frac{\Gamma(\pm ir_k)}{\Gamma\left(\frac{1}{4} \pm \frac{1}{2}ir_k\right)^2} \left(2^{-\frac{3}{4} \pm \frac{1}{2}ir_k} \frac{\Gamma\left(\frac{1}{4} \pm \frac{1}{2}ir_k + ir\right) \Gamma\left(\frac{1}{4} \pm \frac{1}{2}ir_k - ir\right)}{\Gamma\left(\frac{1}{2} \pm ir_k\right)}\right)^{-1} \\ &\hskip 200pt \cdot c_k(h) \left\langle \phi_k,|\phi|^2 \right\rangle \\
                                                                  &
                                                                  \hskip-60pt
                                                                  = 4\pi^{\frac{1}{2}}h^{\mp ir_k} \frac{\Gamma(\pm ir_k) \Gamma\left(\frac{1}{2} \pm ir_k\right)}{\Gamma\left(\frac{1}{4} \pm \frac{1}{2}ir_k\right)^2 \Gamma\left(\frac{1}{4} \pm \frac{1}{2}ir_k + ir\right) \Gamma\left(\frac{1}{4} \pm \frac{1}{2}ir_k - ir\right)} c_k(h) \left\langle \phi_k,|\phi|^2 \right\rangle \\
                                                                  &
                                                                  \hskip-60pt
                                                                  = 2^{\frac{3}{2} \pm ir_k}h^{\mp ir_k} \frac{\Gamma(\pm ir_k) \Gamma\left(\frac{3}{4} \pm \frac{1}{2}ir_k\right)}{\Gamma\left(\frac{1}{4} \pm \frac{1}{2}ir_k\right) \Gamma\left(\frac{1}{4} \pm \frac{1}{2}ir_k + ir\right) \Gamma\left(\frac{1}{4} \pm \frac{1}{2}ir_k - ir\right)} c_k(h) \left\langle \phi_k,|\phi|^2 \right\rangle.
  \end{align*}
  Thus
  \begin{align*}
    &\frac{T^{ir_k}}{\left(\frac{1}{2} + ir_k\right)^{\frac{5}{2} + \varepsilon}} \underset{s = \frac{1}{2} + ir_k}{\textup{Res}} L_h(s,\phi) \;+\; \frac{T^{-ir_k}}{\left(\frac{1}{2} - ir_k\right)^{\frac{5}{2} + \varepsilon}}\; \underset{s = \frac{1}{2} - ir_k}{\textup{Res}} L_h(s,\phi) \\
    &\hskip 10pt = 2^{\frac{5}{2}}c_k(h) \left\langle \phi_k,|\phi|^2 \right\rangle  \cdot 
    \textup{Re}\left(\frac{(T/h)^{ir_k}}{(\frac12+ir_k)^{\frac52+\varepsilon}}\;\tfrac{\Gamma(ir_k) \Gamma\left(\frac{3}{4} + \frac{1}{2}ir_k\right)}{\Gamma\left(\frac{1}{4} + \frac{1}{2}ir_k\right) \Gamma\left(\frac{1}{4} + \frac{1}{2}ir_k + ir\right) \Gamma\left(\frac{1}{4} + \frac{1}{2}ir_k - ir\right)}\right).
  \end{align*}

 This completes the second and final step in the proof of Theorem \ref{AsymptoticFormula}. 
\end{proof}

\section{\large \bf Proof of Theorem \ref{unsmoothedSCS}}\label{unsmoothedsum}

In Theorem  \ref{unsmoothedSCS} (assuming $h < x^{\frac12-\varepsilon}$) we obtain the following bound for the unsmoothed shifted convolution sum:

$$
    \underset{n \neq 0,-h}{\sum_{\sqrt{|n (n + h)|} < x}} c(n) c(n + h) \ll h^{\frac{2}{3}\theta + \varepsilon}x^{\frac{2}{3} (1 + \theta) + \varepsilon} + h^{\frac{1}{2} + \varepsilon}x^{\frac{1}{2} + 2\theta + \varepsilon}.$$

\begin{proof}
The proof of Theorem \ref{unsmoothedSCS} is similar to the proof of Theorem  \ref{AsymptoticFormula} except that we now use the Perron type formula (see Chapter 17 of \cite{Davenport2013-br})
\begin{equation}\label{DavLemma}
\frac{1}{2\pi i}\int\limits_{c-iT}^{c+iT} y^s\; \frac{ds}{s} = \delta(y) \; + \;\begin{cases}
          \mathcal O\left(y^c \textup{min}\left(1,\;T^{-1} |\textup{log}y|^{-1}\right)\right) &\textup{if } y \neq 1, \\
          \mathcal O\left(c\,T^{-1}\right) &\textup{if } y = 1,
        \end{cases}
\end{equation}
where $c,T,y>0,$ and
\[
      \delta(y) =
      \begin{cases}
        1 &\textup{if } y > 1, \\
        \frac{1}{2} &\textup{if } y = 1, \\
        0 &\textup{if } 0 < y < 1.
      \end{cases}
    \]

    We begin with the following lemma which expresses the unsmoothed shifted convolution sum in terms of a Perron type integral.
    \begin{lemma} \label{UnsmoothedPerronIntegral} Fix $\varepsilon>0.$ Assume $h<x^{\frac12-\varepsilon}$ where $x$ is  sufficiently large \footnote{Here sufficiently large means there exists an effectively computable constant $c_0>1$ depending only on $\varepsilon$ for which $x>c_0.$}.  Then for $1<T\ll \frac{x^{1+\varepsilon}}{h}$  we have
   $$
     \frac{1}{2\pi i} \int\limits_{\frac{1 + \varepsilon}{2} - iT}^{\frac{1 + \varepsilon}{2} + iT} L_h(2s,\phi) \frac{x^{2s}}{s}\; ds \;= \hskip-5pt \underset{n \neq 0,-h}{\sum_{\sqrt{|n (n + h)|} < x}} c(n) c(n + h) \; + \; \mathcal O\left(\frac{x^{1 + 2\theta + \varepsilon}}{T}\right)
   $$
    where $\theta$ is the best progress toward the Ramanujan--Petersson conjecture.
    \end{lemma}
\begin{proof} Choose $x=\kappa+\frac13$ with $\kappa \in\mathbb Z$ so that $\frac{x^2}{|n(n+h)|} \ne 1.$ It follows from \ref{DavLemma} that

\begin{align} \label{ErrorTerm}
 \frac{1}{2\pi i} \int\limits_{\frac{1 + \varepsilon}{2} - iT}^{\frac{1 + \varepsilon}{2} + iT} L_h(2s,\phi) \frac{x^{2s}}{s}\; ds & = \sum_{n\ne 0,-h} c(n) c(n + h)\cdot\frac{1}{2\pi i} \int\limits_{\frac{1 + \varepsilon}{2} - iT}^{\frac{1 + \varepsilon}{2} + iT} \left(\frac{x^2}{|n(n+h)|}\right)^s \;\frac{ds}{s}\\
 & \hskip-122pt=  \underset{n \neq 0,-h}{\sum_{\sqrt{|n (n + h)|} < x}} \hskip-10pt
 c(n) c(n + h) \; + \; \mathcal O\left(\sum_{n\ne 0,-h}
\hskip-5pt x^{1+\varepsilon}\,\tfrac{c(n) c(n + h)}{|n(n+h)|^{\frac{1+\varepsilon}{2}}}\cdot \min\left(1,\;\, \frac{1}{T\cdot \left|\log\tfrac{x^2}{|n (n + h)|} \right|  }  \right)\right).\nonumber
 \end{align}
To estimate the error term on the lower right side of (\ref{ErrorTerm}) we take first all the terms for which 
$|n(n+h)| \le \tfrac{x^2}{3}$ or $|n(n+h)| \ge 3x^2.$ In these cases  $\left|\log\tfrac{x^2}{|n (n + h)|} \right| > \log(3) >1.$
Since $T>1$, the contribution of these cases to the error term is bounded by
$$\ll \; \frac{x^{1+\varepsilon}}{T}\sum_{n\ne 0,-h}\frac{c(n) c(n + h)}{ |n(n+h)|^{\frac{1+\varepsilon}{2}}  } \; \ll \; \frac{x^{1+\varepsilon}}{T}.$$
Next, we consider the terms for which $\tfrac{x^2}{3} < |n(n+h)| < x^2.$
Note that
$$\left| \log \left(\frac{x^2}{|n(n+h)| }\right)\right|^{-1} = \left|\log\left(1-\frac{(|n(n+h)|-x^2)}{|n(n+h)|}\right)\right|^{-1}\ll \frac{|n(n+h)| }{x^2-|n(n+h)|}.$$

It follows that the contribution of these cases to the error term is bounded by
\begin{align*}
\sum_{\frac{x^2}{3}<|n(n+h)|<x^2}\hskip-4pt\frac{x^{1+\varepsilon}}{T} \frac{c(n) c(n + h) \, |n(n+h)|^{\frac12+\varepsilon}}{x^2-|n(n+h)|} 
&\ll \;
\frac{x^{1+2\theta+\varepsilon}}{T}\hskip-5pt
\sum_{\frac{x^2}{3}<|n(n+h)|<x^2} \frac{|n(n+h)|^\frac12}{x^2-|n(n+h)|}\\
&
\hskip-130pt
 \ll \frac{x^{1+2\theta+\varepsilon}}{T}\hskip-5pt
\sum_{\frac{x^2}{3}<|n(n+h)|<x^2}  \frac{|n(n+h)|^\frac12}{\Big(x-(|n(n+h)|)^\frac12\Big)\Big(x+(|n(n+h)|)^\frac12\Big)}\\
&
\hskip-130pt
 \ll \frac{x^{1+2\theta+\varepsilon}}{T}\hskip-5pt
\sum_{\frac{x^2}{3}<|n(n+h)|<x^2}  \frac{1}{x-\sqrt{|n(n+h)|}}.
\end{align*}

By the binomial theorem we see that
$$
\sqrt{|n(n+h)|} = |n| + \frac{h}{2} + \mathcal O\left( \frac{h^2}{|n|}  \right).$$
It follows that for $h< x^{\frac12-\varepsilon}$ and $x$ sufficiently large that
\begin{align*}
\sum_{\frac{x}{3}<\sqrt{|n(n+h)|}<x}  \frac{1}{x-\sqrt{|n(n+h)|}} & \ll \sum_{\frac{x}{\sqrt{3}}< |n|+\frac{h}{2}<x} \; \frac{1}{x-\Big(|n|+\frac{h}{2} +  \mathcal O\left( x^{-\varepsilon} \right)\Big)}\\
& \ll \int\limits_{\frac{x}{\sqrt{3}}}^{x - \frac14} \frac{dy}{x - y}\; \ll  \; \log x
\end{align*}
since $\sqrt{|n(n+h)|}$ is always very close to a half-integer while $x=\kappa +\tfrac13$ where $\kappa$ is a positive integer.

The same bound holds for the case when $x^2 < |n(n+h)| < 3x^2$ by a similar argument.
It follows that the contribution of these cases to the error term on the lower right side of (\ref{ErrorTerm}) is given by $\mathcal O\left(  \frac{x^{1+2\theta+\varepsilon}}{T} \right).$ This proves Lemma \ref{UnsmoothedPerronIntegral} when
$x-\tfrac13\in\mathbb Z.$ 

We can extend the proof to all $x,T$ where $1<T\ll \frac{x^{1+\varepsilon}}{h}$ since this change amounts to at most $\ll h$ additional terms in the sum 
$$\underset{n \neq 0,-h}{\sum\limits_{\sqrt{|n (n + h)|} < x}} c(n) c(n + h). $$
\end{proof}

Note that    $\boxed{\;\mathcal I_{T,x}(s) := \frac{1}{2\pi i} \int\limits_{\frac{1 + \varepsilon}{2} - iT}^{\frac{1 + \varepsilon}{2} + iT} L_h(2s,\phi) \,\frac{x^{2s}}{s}\, ds =  \frac{1}{2\pi i} \int\limits_{1 + \varepsilon - i\infty}^{1 + \varepsilon + i\infty} L_h(s,\phi)\, \frac{x^s}{s}\, ds.}$

\vskip 3pt\noindent 
The remainder of the proof of Theorem \ref{unsmoothedSCS} has 3 steps. In the first step we shift the line of integration  in $\mathcal I_{T,x}(s)$ to ${\rm Re}(s) =\varepsilon$ and bound all the line integrals that occur when the shift of integration is performed. In the second step we obtain a bound for the sum of residues. In the third and final step we combine these bounds with Lemma  \ref{UnsmoothedPerronIntegral} which  relates $\mathcal I_{T,x}(s)$ with the unsmoothed shifted convolution sum.

\vskip 10pt
\noindent
$\underline{\text{\bf Step 1: Bounding all the line integrals which occur after shifting $\mathcal I_{T,x}(s)(s)$.}}$
\vskip 5pt
We begin by bounding the left vertical integral.

  \begin{lemma}
    Let $x > 0$. For $T \rightarrow \infty$, we have
    \[
      \frac{1}{2\pi i} \int\limits_{\varepsilon - iT}^{\varepsilon + iT} L_h(s,\phi) \frac{x^s}{s} ds \ll h^{\frac{1}{2} + \theta}x^{\varepsilon}T^{\frac{3}{2}} + hx^{\varepsilon}T.
    \]
  \end{lemma}

  \begin{proof}
   It follows from (\ref{Lh(s,phi)-Bound}) that
    \begin{align*}
      \int\limits_{\varepsilon - iT}^{\varepsilon + iT} L_h(s,\phi) \frac{x^s}{s} ds &\ll h^{\frac{1}{2} + \theta}x^{\varepsilon} \int\limits_0^T (1 + t)^{\frac{3}{2}} \left(\varepsilon^2 + t^2\right)^{-\frac{1}{2}} dt + hx^{\varepsilon} \int\limits_0^T (1 + t) \left(\varepsilon^2 + t^2\right)^{-\frac{1}{2}} dt \\
                                                                                                       &\ll h^{\frac{1}{2} + \theta}x^{\varepsilon}T^{\frac{3}{2}} + hx^{\varepsilon}T.     \end{align*}\end{proof}
   Next, we bound the two horizontal integrals. 
   
  \begin{lemma} \label{LowerBoundLemma}
    Assume $\left|\textup{log}\left(\frac{x}{h}\right)\right| > 1$, $\left|\textup{log}\left(\frac{x}{hT}\right)\right|>1,$ and $\left|\textup{log}\left(\frac{x}{h^{2\theta + \varepsilon}T^2}\right)\right|>1.$  Then    \[
      \frac{1}{2\pi i} \int\limits_{\varepsilon \pm iT}^{1 + \varepsilon \pm iT} L_h(s,\phi) \frac{x^s}{s} ds \ll h^{\theta + \varepsilon}x^{\frac{1}{2} + \varepsilon}T^{\frac{1}{2} + \varepsilon} + h^{\frac{1}{2} + \theta}x^{\varepsilon}T^{\frac{1}{2} - \varepsilon} + x^{1 + \varepsilon}T^{-1} + hx^{\varepsilon}T^{-1}.
    \]
  \end{lemma}

  \begin{proof}
    We need to assume that the horizontal line segments do not pass through or near any poles $\frac{1}{2} \pm ir_k$ of $L_h(s,\phi)$, which can be done for suitable $T \rightarrow \infty$ as before. We break up the integrals as
    \begin{equation}\label{HorizIntegrals}
      \frac{1}{2\pi i} \int\limits_{\varepsilon \pm iT}^{1 + \varepsilon \pm iT} L_h(s,\phi) \frac{x^s}{s} ds = \frac{1}{2\pi i} \left(\int\limits_{\varepsilon \pm iT}^{\frac{1}{2} - \varepsilon \pm iT} + \int\limits_{\frac{1}{2} - \varepsilon \pm iT}^{\frac{1}{2} + \varepsilon \pm iT} + \int\limits_{\frac{1}{2} + \varepsilon \pm iT}^{1 + \varepsilon \pm iT}\right) L_h(s,\phi) \frac{x^s}{s} ds.
\end{equation}

    The first and third integrals are computed over line segments that are a distance greater than $\varepsilon$ from any pole $\frac{1}{2} \pm ir_k$ of $L_h(s,\phi)$, so we can bound them using (\ref{Lh(s,phi)-Bound}) as follows.
\vskip 8pt   
    First of all
   \begin{align*}
     \frac{1}{2\pi i} \int\limits_{\varepsilon \pm iT}^{\frac{1}{2} - \varepsilon \pm iT} L_h(s,\phi) \frac{x^s}{s} ds &\ll \int\limits_{\varepsilon}^{\frac{1}{2} - \varepsilon } \left(h^{\frac{1}{2} - \sigma + \theta + \varepsilon} T^{\frac{3}{2} - \sigma + \varepsilon} + h^{1 - \sigma + \varepsilon} T^{1 - \sigma + \varepsilon}\right)\cdot \frac{x^{\sigma}}{T}\; d\sigma\\
                                                                                                                       & = \left(h^{\frac{1}{2} + \theta + \varepsilon}T^{\frac{1}{2} + \varepsilon} + h^{1 + \varepsilon} T^{\varepsilon}\right) \int\limits_{\varepsilon}^{\frac{1}{2} - \varepsilon} \left(\tfrac{x}{hT}\right)^{\sigma} d\sigma \\
                                                                                                                       & \ll \left(h^{\frac{1}{2} + \theta + \varepsilon}T^{\frac{1}{2} + \varepsilon} + h^{1 + \varepsilon} T^{\varepsilon}\right) \left(\left(\tfrac{x}{hT}\right)^{\frac{1}{2} - \varepsilon} + \left(\tfrac{x}{hT}\right)^{\varepsilon}\right) \\
                                                                                                                       & = h^{\theta + 2\varepsilon}x^{\frac{1}{2} - \varepsilon}T^{2\varepsilon} + h^{\frac{1}{2} + \theta}x^{\varepsilon}T^{\frac{1}{2}} + h^{\frac{1}{2} + 2\varepsilon}x^{\frac{1}{2} - \varepsilon}T^{-\frac{1}{2} + 2 \varepsilon} + hx^{\varepsilon}.
   \end{align*} 
     
     \pagebreak
     
     Secondly
     \begin{align*}
      \int\limits_{\frac{1}{2} + \varepsilon \pm iT}^{1 + \varepsilon \pm iT} L_h(s,\phi) \frac{x^s}{s} ds &\ll \int\limits_{\frac{1}{2} + \varepsilon}^{1 + \varepsilon} \left( h^{(2\theta + \varepsilon) (1 - \sigma + \varepsilon)} \,T^{2 (1 - \sigma + \varepsilon)} + h^{1 - \sigma + \varepsilon} T^{1 - \sigma + \varepsilon} \right)
                                                                                                             \cdot  \frac{x^{\sigma}}{T}\; d\sigma\\
                                                                                                           & = h^{(2\theta + \varepsilon) (1 + \varepsilon)}T^{1 + 2\varepsilon} \int\limits_{\frac{1}{2} + \varepsilon}^{1 + \varepsilon} \left(\tfrac{x}{h^{2\theta + \varepsilon}T^2}\right)^{\sigma} d\sigma + h^{1 + \varepsilon}T^{\varepsilon} \int\limits_{\frac{1}{2} + \varepsilon}^{1 + \varepsilon} \left(\tfrac{x}{hT}\right)^{\sigma} d\sigma \\
                                                                                                           & \ll h^{(2\theta + \varepsilon) (1 + \varepsilon)}T^{1 + 2\varepsilon} \left(\left(\tfrac{x}{h^{2\theta + \varepsilon}T^2}\right)^{1 + \varepsilon} + \left(\tfrac{x}{h^{2\theta + \varepsilon}T^2}\right)^{\frac{1}{2} + \varepsilon}\right) \\ & \hskip 180pt + h^{1 + \varepsilon}T^{\varepsilon} \left(\left(\tfrac{x}{hT}\right)^{1 + \varepsilon} + \left(\tfrac{x}{hT}\right)^{\frac{1}{2} + \varepsilon}\right) \\
                                                                                                           & = x^{1 + \varepsilon}T^{-1} + h^{\theta + \frac{1}{2}\varepsilon}x^{\frac{1}{2} + \varepsilon} + x^{1 + \varepsilon}T^{-1} + h^{\frac{1}{2}}x^{\frac{1}{2} + \varepsilon}T^{-1}.
      \end{align*}   
  \end{proof}

    We now bound the second integral in (\ref{HorizIntegrals}). As before, we need to choose appropriate $T \rightarrow \infty$ so that the horizontal lines do not pass too close to any of the poles $\frac{1}{2} \pm ir_k$ of $L_h(s,\phi)$. In particular, among the values of $r_k$ with $T \leq r_k \leq T + 1$, there is a gap of length $\gg T^{-1}$, so by an appropriate choice of $T$ we can ensure that $|r_k - T| \gg T^{-1}$ for all $k$. From now on we assume that $T$ is chosen in this way.
   
    \begin{lemma} For $k=1,2,3,\ldots$ let $\tfrac12+ir_k$ denote the poles of  $L_h(s,\phi)$. Assume $T\to\infty$ is chosen so that   $|r_k - T| \gg T^{-1}$ for all $k$. Then
    \[
      \frac{1}{2\pi i} \int\limits_{\frac{1}{2} - \varepsilon \pm iT}^{\frac{1}{2} + \varepsilon \pm iT} L_h(s,\phi) \frac{x^s}{s} ds \ll h^{\theta + \varepsilon}x^{\frac{1}{2} + \varepsilon}T^{\frac{1}{2} + \varepsilon} + h^{\frac{1}{2} + \varepsilon}x^{\frac{1}{2} + \varepsilon}T^{-\frac{1}{2} + \varepsilon}.
    \]
    \end{lemma}
    \begin{proof}
    To see this, we exactly follow the proof of Lemma \ref{FifthLemma} but with the integrand having $x^s$ rather than $T^s$ in the numerator and $s$ rather than $s^{\frac{5}{2} + \varepsilon}$ in the denominator; this does not affect the validity of any of the arguments or make the computations substantially different.
  \end{proof}

\vskip 10pt\noindent
$\underline{\text{\bf Step 2: Bounding the sum of residues.}}$
\vskip 5pt
\begin{lemma} The bound for the absolute value of the sum of residues after shifting the line of integration in $\mathcal I_{T,x}(s)$ to ${\rm Re}(s) = -\varepsilon$ is  $\ll h^{\theta + \varepsilon}x^{\frac{1}{2}}$.
\end{lemma}
\begin{proof}
  The computation of an upper bound for the absolute value of the sum of the residues is precisely analogous to the computation in the previous section so that the absolute value of that sum is less than a constant times $h^{\theta + \varepsilon}x^{\frac{1}{2}}$.
  \end{proof}

\pagebreak

\vskip 10pt\noindent
$\underline{\text{\bf Step 3: Combining all the previous lemmas.}}$
\vskip 5pt
  If we choose $x \rightarrow \infty$ so that the conditions on $x$ and $T$ in the above lemmas are simultaneously satisfied, then we have
  \[
    \underset{n \neq 0,-h}{\sum_{\sqrt{|n (n + h)|} < x}} c(n) c(n + h) \ll x^{1 + 2\theta + \varepsilon}T^{-1} + h^{\frac{1}{2} + \theta}x^{\varepsilon}T^{\frac{3}{2}} + hx^{\varepsilon}\textup{log}T + h^{\theta + \varepsilon}x^{\frac{1}{2} + \varepsilon}T^{\frac{1}{2} + \varepsilon}
  \]
  for every $\varepsilon > 0$. To minimize this bound, we let $\boxed{T = h^{-\left(\frac{2}{3}\theta + \varepsilon\right)}x^{\frac{1}{3} + \frac{4}{3}\theta},}$ yielding the bound
  \[
    \underset{n \neq 0,-h}{\sum_{\sqrt{|n (n + h)|} < x}} c(n) c(n + h) \ll h^{\frac{2}{3}\theta + \varepsilon}x^{\frac{2}{3} (1 + \theta) + \varepsilon} + h^{\frac{1}{2} + \varepsilon}x^{\frac{1}{2} + 2\theta + \varepsilon} + h^{1 + \varepsilon}x^{\varepsilon}
  \]
  for every $\varepsilon > 0$. Note that this choice of $T$ is compatible with the conditions $h<x^{\frac12-\varepsilon}$ and $1<T\ll \frac{x^{1+\varepsilon}}{h}$ required in Lemma \ref{UnsmoothedPerronIntegral} as well as the lower bounds needed for Lemma \ref{LowerBoundLemma} and that the condition $h < x^{\frac{1}{2} - \varepsilon}$ allows us to remove the third term in the final bound.

\end{proof}

\section{\large \bf The different spectral parameter case}\label{Different}

We now discuss the differences that arise in the case in which $\mathfrak r_1 \neq \mathfrak r_2$ and show that, with additional difficulty, we are able to obtain results that are precisely analogous to those obtained in the $\mathfrak r_1 = \mathfrak r_2$ case discussed above. The key difficulty that makes this case more complicated is that the integral
\[
  \int\limits_0^{\infty} K_a(my) K_b(ny) e^{-y} y^{2s} \frac{dy}{y}
\]
cannot be expressed in as simple a form as when $a = b$. We find a more complicated expression for this integral in Lemma \ref{Identity}. This introduces a hypergeometric twist that we extract by taking the power series expansion, yielding a term precisely analogous to the $\mathfrak r_1 = \mathfrak r_2$ case (from the constant term) plus an infinite series (from the higher-degree terms) that we then show can be bounded by a term that is smaller than the error term that already arises from the first part. We thus obtain the same meromorphic continuation and bound as in the $\mathfrak r_1 = \mathfrak r_2$ case and are able to obtain precisely analogous results for the main theorems.

\begin{definition}
  Define the functions $L_h(s,\Phi_1,\Phi_2)$ and $L_h^{\#}(s,\Phi_1,\Phi_2)$ by the Dirichlet series
  \[
    L_h(s,\Phi_1,\Phi_2) = \sum_{n \neq -h,0} C_1(n) C_2(n + h) |n|^{-\frac{1}{2} s} |n + h|^{-\frac{1}{2} s + i (\mathfrak r_1 - \mathfrak r_2)}
  \]
  and
  \[
    L_h^{\#}(s,\Phi_1,\Phi_2) = \sum_{n \neq -h,0} C_1(n) C_2(n + h) |n|^{-\frac{1}{2} s} |n + h|^{-\frac{1}{2} s + i (\mathfrak r_1 - \mathfrak r_2)} \mathcal F_{+,\mathfrak r_1,\mathfrak r_2,\frac{n}{n + h}}\left(-\tfrac{1}{2} s\right),
  \]
  where
  \begin{align*}
    & \mathcal F_{+,\mathfrak r_1,\mathfrak r_2,a}(w) = B(-w + i\mathfrak r_2,-w - i\mathfrak r_1) \\ & \hskip 150pt \cdot F(-i (\mathfrak r_1 - \mathfrak r_2),-w + i\mathfrak r_2;-2w - i (\mathfrak r_1 - \mathfrak r_2);1 - a).
  \end{align*}
  In the last definition, $B$ is the beta function and $F$ is the Gauss hypergeometric function. Note that the Dirichlet series defining $L_h(s,\Phi_1,\Phi_2)$ and $L_h^{\#}(s,\Phi_1,\Phi_2)$ converge for $\sigma > 1$.
\end{definition}

We now prove the following proposition, which will in turn be used to prove Theorem \ref{TheoremPhi1Phi2}. 

\begin{proposition}\label{PropPhi1Phi2}
For $\sigma > 1 + 2\varepsilon$, we have
\[
  \boxed{\left\langle P_h(*,s),\overline{\Phi_1} \Phi_2 \right\rangle = \frac{\Gamma(s)}{2 (2\pi h)^s} \Big(\mathcal G_h(s,\Phi_1,\Phi_2) + \mathcal T_h(s,\Phi_1,\Phi_2)\Big)}
\]
where
\begin{align*}
  \mathcal G_h(s,\Phi_1,\Phi_2) & = \frac{1}{2\pi i} \int\limits_{\textup{Re}(w) = - \frac{1}{2} - \varepsilon} 2^w h^{-2w} \frac{\Gamma\left(\frac{1}{2} s + \frac{1}{2} + w\right) \Gamma\left(\frac{1}{2} s + w\right) \Gamma(-w)}{\Gamma\left(s + \frac{1}{2} + w\right)} \\ 
  & 
  \hskip 10pt \cdot \sum_{n < -h \textup{ or } n > 0} C_1(n) C_2(n + h) |n (n + h)|^w |n + h|^{i (\mathfrak r_1 - \mathfrak r_2)}  \cdot \mathcal F_{+,\mathfrak r_1,\mathfrak r_2,\frac{n}{n + h}}(w) dw,\\
  &\\   \mathcal T_h(s,\Phi_1,\Phi_2) & = \frac{1}{2\pi i} \int\limits_{\textup{Re}(w) = - \frac{1}{2} + \varepsilon} 2^w h^{-2w} \frac{\Gamma\left(\frac{1}{2} s + \frac{1}{2} + w\right) \Gamma\left(\frac{1}{2} s + w\right) \Gamma(-w)}{\Gamma\left(s + \frac{1}{2} + w\right)} \\
   &
    \hskip 10pt \cdot \sum_{-h < n < 0} C_1(n) C_2(n + h) |n (n + h)|^w |n + h|^{i (\mathfrak r_1 - \mathfrak r_2)} \cdot \mathcal F_{-,\mathfrak r_1,\mathfrak r_2,\frac{n}{n + h}}(w) dw.
    \end{align*}

\noindent
Here
\[
  \mathcal F_{\pm,\mathfrak r_1,\mathfrak r_2,a}(w) = 2^{-w} \left(\mathcal M_{\pm,\mathfrak r_1,\mathfrak r_2,a}(w) + \mathcal E_{\mathfrak r_1,\mathfrak r_2,a}(w)\right),
\]
\begin{align*}
  \mathcal M_{+,\mathfrak r_1,\mathfrak r_2,a}(w) & = B(-w + i\mathfrak r_2,-w - i\mathfrak r_2) \cdot F(-i (\mathfrak r_1 - \mathfrak r_2),-w + i\mathfrak r_2;w + i\mathfrak r_2 + 1;a),
\end{align*}
\begin{align*}
  \mathcal M_{-,\mathfrak r_1,\mathfrak r_2,a}(w) & = (B(-w + i\mathfrak r_2,2w + 1) + B(-w - i\mathfrak r_2,2w + 1)) \\ & \hskip 135pt \cdot F(-i (\mathfrak r_1 - \mathfrak r_2),-w + i\mathfrak r_2;w + i\mathfrak r_2 + 1;a) \\
                              & = \frac{\textup{cos}(i\pi \mathfrak r_2)}{\textup{cos}(\pi w)} \mathcal M_{+,\mathfrak r_1,\mathfrak r_2,a}(w),
\end{align*}
and
\[
  \mathcal E_{\mathfrak r_1,\mathfrak r_2,a}(w) = |a|^{-w - i\mathfrak r_2} B(w + i\mathfrak r_2,-w - i\mathfrak r_1) F(-2w,-w - i\mathfrak r_1;-w - i\mathfrak r_2 + 1;a).
\]
We also have
\begin{align*}
  & \mathcal F_{+,\mathfrak r_1,\mathfrak r_2,a}(w) = B(-w + i\mathfrak r_2,-w - i\mathfrak r_1) \\ & \hskip 150pt \cdot F(-i (\mathfrak r_1 - \mathfrak r_2),-w + i\mathfrak r_2;-2w - i (\mathfrak r_1 - \mathfrak r_2);1 - a).
\end{align*}
\end{proposition}

\noindent
{\it Proof of Proposition \ref{PropPhi1Phi2}.}
We begin the proof of this proposition with the following lemma.

\begin{lemma}
  For $a,b \in \mathbb C$ and $x,y \in \mathbb R_{> 0}$, we have
  \[
    K_a(x) K_b(y) = \frac{1}{2} \int\limits_{-\infty}^{\infty} K_{a - b}\left(\left(x^2 + y^2 + 2xy \cdot \textup{cosh}(u)\right)^{\frac{1}{2}}\right) e^{-\frac{a + b}{2}u} \left(\frac{xe^u + y}{ye^u + x}\right)^{\frac{a - b}{2}} du.
  \]
\end{lemma}

\begin{proof}
  From the identity
  \[
    K_r(z) = \int\limits_0^{\infty} e^{-z \cdot \textup{cosh}(t)} \cdot \textup{cosh}(rt)\, dt \;=\; \frac{1}{2} \int\limits_{-\infty}^{\infty} e^{-rt - z \cdot \textup{cosh}(t)} dt,
  \]
  which holds for all $r \in \mathbb C$ and $z \in \mathbb C$ with $\textup{Re}(z) > 0$, we have
  \[
    K_a(x) K_b(y) = \frac{1}{4} \int\limits_{-\infty}^{\infty} \int\limits_{-\infty}^{\infty} e^{-at - bu - x \cdot \textup{cosh}(t) - y \cdot \textup{cosh}(u)} dt du.
  \]
  Letting $t = T + U$ and $u = T - U$ yields
  \[
    K_a(x) K_b(y) = \frac{1}{2} \int\limits_{-\infty}^{\infty} \int\limits_{-\infty}^{\infty} e^{-(a + b) T - (a - b) U - x \cdot \textup{cosh}(T + U) - y \cdot \textup{cosh}(T - U)}\, dT \,dU.
  \]
  By letting $v = \left(xe^T + ye^{-T}\right) e^U$, we have
  \[
    \int\limits_{-\infty}^{\infty} e^{-(a - b) U - x \cdot \textup{cosh}(T + U) - y \cdot \textup{cosh}(T - U)} dU = 
    \hskip-3pt
    \int\limits_0^{\infty} \left(\frac{xe^T + ye^{-T}}{v}\right)^{a - b}
 \hskip-5pt   
     e^{-\frac{1}{2} \left(v + \frac{x^2 + y^2 + 2xy \cdot \textup{cosh}(2T)}{v}\right)} \frac{dv}{v}.
  \]
  Thus
  \[
    K_a(x) K_b(y) = \frac{1}{2} \int\limits_{-\infty}^{\infty} \int\limits_0^{\infty} e^{-(a + b) T} \left(xe^T + ye^{-T}\right)^{a - b} e^{-\frac{1}{2} \left(v + \frac{x^2 + y^2 + 2xy \cdot \textup{cosh}(2T)}{v}\right)} \frac{dv}{v^{(a - b) + 1}}\, dT.
  \]
  From the identity
  \[
    K_r(z) = \frac{1}{2} \left(\frac{1}{2}z\right)^r \int\limits_0^{\infty} e^{-t - \frac{z^2}{4t}} \frac{dt}{t^{r + 1}},
  \]
  which holds for all $r \in \mathbb C$ and $z \in \mathbb C$ with $\textup{Re}\left(z^2\right) > 0$, we have
  \pagebreak
  \begin{align*}
    & K_{a - b}\left(\left(x^2 + y^2 + 2xy \cdot \textup{cosh}(2T)\right)^{\frac{1}{2}}\right) \\ & \hskip 50pt = \frac{1}{2} \left(\frac{1}{2} \left(x^2 + y^2 + 2xy \cdot \textup{cosh}(2T)\right)^{\frac{1}{2}}\right)^{a - b} \int\limits_0^{\infty} e^{-v - \frac{x^2 + y^2 + 2xy \cdot \textup{cosh}(2T)}{4v}} \frac{dv}{v^{(a - b) + 1}} \\
    & \hskip 50pt = \left(x^2 + y^2 + 2xy \cdot \textup{cosh}(2T)\right)^{\frac{a - b}{2}} \int\limits_0^{\infty} e^{-\frac{1}{2} \left(v + \frac{x^2 + y^2 + 2xy \cdot \textup{cosh}(2T)}{v}\right)} \frac{dv}{v^{(a - b) + 1}},
  \end{align*}
  so
  \begin{align*}
    K_a(x) K_b(y) & = \frac{1}{2} \int\limits_{-\infty}^{\infty} K_{a - b}\left(\left(x^2 + y^2 + 2xy \cdot \textup{cosh}(2T)\right)^{\frac{1}{2}}\right) e^{-(a + b) T} \\ & \hskip 50pt \cdot \left(\frac{xe^T + ye^{-T}}{\left(x^2 + y^2 + 2xy \cdot \textup{cosh}(2T)\right)^{\frac{1}{2}}}\right)^{a - b} dT.
  \end{align*}
  Note that
  \[
    \left(xe^T + ye^{-T}\right) \left(ye^T + xe^{-T}\right) = x^2 + y^2 + xy \left(e^{2T} + e^{-2T}\right) = x^2 + y^2 + 2xy \cdot \textup{cosh}(2T).
  \]
  Thus
  \[
    \left(\frac{xe^T + ye^{-T}}{\left(x^2 + y^2 + 2xy \cdot \textup{cosh}(2T)\right)^{\frac{1}{2}}}\right)^{a - b} = \left(\frac{xe^T + ye^{-T}}{ye^T + xe^{-T}}\right)^{\frac{a - b}{2}} = \left(\frac{xe^{2T} + y}{ye^{2T} + x}\right)^{\frac{a - b}{2}}.
  \]
\end{proof}

We also need the following lemma for the proof of Proposition \ref{PropPhi1Phi2}.

\begin{lemma}\label{Identity}
  For $a,b \in \mathbb C$, $m,n \in \mathbb R_{> 0}$, and $s \in \mathbb C$ with $\textup{Re}(a) = \textup{Re}(b) = 0$ and $\textup{Re}(s) > 0$, we have
  \begin{align*}
    & \int\limits_0^{\infty} K_a(my) K_b(ny) e^{-y} y^{2s} \frac{dy}{y} = \pi^{\frac{1}{2}} 2^{-2s - 1} \frac{\Gamma(2s + (a - b)) \Gamma(2s - (a - b))}{\Gamma\left(s + \frac{1}{2}\right)} \\ & \hskip 50pt \cdot \int\limits_{-\infty}^{\infty} F\left(s + \tfrac{a - b}{2} + \tfrac{1}{2},s + \tfrac{a - b}{2};2s + \tfrac{1}{2};1 - \left(m^2 + n^2 + 2mn \cdot \textup{cosh}(u)\right)\right) \\ & \hskip 100pt \cdot e^{-\frac{a + b}{2}u} \left(me^{\frac{1}{2}u} + ne^{-\frac{1}{2}u}\right)^{a - b} du.
  \end{align*}
\end{lemma}

\begin{proof}
  From the preceding result, we have
  \begin{align*}
    & \int\limits_0^{\infty} K_a(my) K_b(ny) e^{-y} y^{2s} \frac{dy}{y} \\ & = \frac{1}{2} \int\limits_0^{\infty} \int\limits_{-\infty}^{\infty} K_{a - b}\left(\left(m^2 + n^2 + 2mn \cdot \textup{cosh}(u)\right)^{\frac{1}{2}} y\right) e^{-\frac{a + b}{2}u} \left(\frac{me^u + n}{ne^u + m}\right)^{\frac{a - b}{2}} du e^{-y} y^{2s} \frac{dy}{y} \\
    & = \frac{1}{2} \int\limits_{-\infty}^{\infty} \int\limits_0^{\infty} K_{a - b}\left(\left(m^2 + n^2 + 2mn \cdot \textup{cosh}(u)\right)^{\frac{1}{2}} y\right) e^{-y} y^{2s} \frac{dy}{y} e^{-\frac{a + b}{2}u} \left(\frac{me^u + n}{ne^u + m}\right)^{\frac{a - b}{2}} du.
  \end{align*}
  We must verify that the integral is absolutely convergent to justify the last step. The integral of the absolute value of the integrand in that case is
  \[
    \int\limits_0^{\infty} \int\limits_{-\infty}^{\infty} \left|K_{a - b}\left(\left(m^2 + n^2 + 2mn \cdot \textup{cosh}(u)\right)^{\frac{1}{2}} y\right)\right| e^{-y} y^{2\sigma - 1} du dy.
  \]
  For $r \in \mathbb C$ fixed and $x \in \mathbb R$, $K_r(x) \ll e^{-x}$, so the preceding integral is less than a constant times
  \[
    \int\limits_0^{\infty} \int\limits_{-\infty}^{\infty} e^{-\left(m^2 + n^2 + 2mn \cdot \textup{cosh}(u)\right)^{\frac{1}{2}} y} e^{-y} y^{2\sigma - 1} du dy.
  \]
  Because the integrand is nonnegative for all values of $u$ and $y$, we can first evaluate the integral over $y$; because $\left(m^2 + n^2 + 2mn \cdot \textup{cosh}(u)\right)^{\frac{1}{2}} > 0$ and $2\sigma > 0$, it converges and
  \[
    \int\limits_0^{\infty} e^{-\left(m^2 + n^2 + 2mn \cdot \textup{cosh}(u)\right)^{\frac{1}{2}} y} e^{-y} y^{2\sigma - 1} dy = \left(m^2 + n^2 + 2mn \cdot \textup{cosh}(u)\right)^{-\sigma} \Gamma(2\sigma).
  \]
  It thus remains to check
  \[
    \int\limits_{-\infty}^{\infty} \left(m^2 + n^2 + 2mn \cdot \textup{cosh}(u)\right)^{-\sigma} \Gamma(2\sigma) du.
  \]
  Because the integrand is defined for all $u$ and decays exponentially as $|u| \rightarrow \infty$ (since $\sigma$ is a positive constant and $\textup{cosh}(u)$ grows exponentially as $|u| \rightarrow \infty$), the preceding integral conrerges. Thus the original integral under consideration converges absolutely, justifying the interchange of the order of integration.

  We now use the Mellin transform
  \[
    \int\limits_0^{\infty} e^{-ay} K_{\nu}(by) y^{s - 1} dy = \frac{\sqrt{\pi}}{(2a)^s} \left(\frac{b}{a}\right)^{\nu} \frac{\Gamma(s + \nu) \Gamma(s - \nu)}{\Gamma\left(s + \frac{1}{2}\right)} F\left(\tfrac{s + \nu + 1}{2},\tfrac{s + \nu}{2};s + \tfrac{1}{2};1 - \left(\tfrac{b}{a}\right)^2\right),
  \]
  which holds for any $\nu,a,b,s \in \mathbb C$ with $\textup{Re}(s) > |\textup{Re}(\nu)|$ and $\textup{Re}(a + b) > 0$ (see page 331 in \cite{MR0061695}). In our situation, this yields
  \begin{align*}
    & \int\limits_0^{\infty} K_{a - b}\left(\left(m^2 + n^2 + 2mn \cdot \textup{cosh}(u)\right)^{\frac{1}{2}} y\right) e^{-y} y^{2s} \frac{dy}{y} \\ & \hskip 50pt = \frac{\pi^{\frac{1}{2}}}{2^{2s}} \left(m^2 + n^2 + 2mn \cdot \textup{cosh}(u)\right)^{\frac{a - b}{2}} \frac{\Gamma(2s + (a - b)) \Gamma(2s - (a - b))}{\Gamma\left(s + \frac{1}{2}\right)} \\ & \hskip 100pt \cdot F\left(s + \tfrac{a - b}{2} + \tfrac{1}{2},s + \tfrac{a - b}{2};2s + \tfrac{1}{2};1 - \left(m^2 + n^2 + 2mn \cdot \textup{cosh}(u)\right)\right),
  \end{align*}
  so
  \begin{align*}
    & \int\limits_0^{\infty} K_a(my) K_b(ny) e^{-y} y^{2s} \frac{dy}{y} = \pi^{\frac{1}{2}} 2^{-2s - 1} \frac{\Gamma(2s + (a - b)) \Gamma(2s - (a - b))}{\Gamma\left(s + \frac{1}{2}\right)} \\ & \hskip 50pt \cdot \int\limits_{-\infty}^{\infty} F\left(s + \tfrac{a - b}{2} + \tfrac{1}{2},s + \tfrac{a - b}{2};2s + \tfrac{1}{2};1 - \left(m^2 + n^2 + 2mn \cdot \textup{cosh}(u)\right)\right) \\ & \hskip 100pt \cdot e^{-\frac{a + b}{2}u} \left(me^{\frac{1}{2}u} + ne^{-\frac{1}{2}u}\right)^{a - b} du.
  \end{align*}
\end{proof}

By precisely analogous arguments to the $\mathfrak r_1 = \mathfrak r_2$ case, for $s \in \mathbb C$ with $\sigma > 1$, we have
\[
  \left\langle P_h(*,s),\Phi_1 \overline{\Phi_2} \right\rangle = \sum_{n \neq -h,0} C_1(n) C_2(n + h) \mathcal I_{\mathfrak r_1,\mathfrak r_2,h}(n,s),
\]
where
\begin{align*}
  \mathcal I_{\mathfrak r_1,\mathfrak r_2,h}(n,s) & = \int\limits_0^{\infty} y^s e^{-2\pi hy} K_{i\mathfrak r_1}(2\pi |n| y) K_{i\mathfrak r_2}(2\pi |n + h| y) \frac{dy}{y} \\
                              & = (2\pi h)^{-s} \int\limits_0^{\infty} y^s e^{-y} K_{i\mathfrak r_1}\left(\tfrac{|n|}{h} y\right) K_{i\mathfrak r_2}\left(\tfrac{|n + h|}{h} y\right) \frac{dy}{y} \\
                              & = \frac{\Gamma(s)}{2 (2\pi h)^s} \int\limits_0^{\infty} \int\limits_{\textup{Re}(w) = -\frac{1}{2} - \varepsilon} \frac{\Gamma\left(\frac{1}{2} s + \frac{1}{2} + w\right) \Gamma\left(\frac{1}{2} s + w\right) \Gamma(-w)}{2\pi i \Gamma\left(s + \frac{1}{2} + w\right)} \\ & \hskip 100pt \cdot \left(\frac{2 |n (n + h)|}{h^2}\right)^w (a_h(n) + \textup{cosh}(u))^w \\ & \hskip 100pt \cdot e^{-i\frac{\mathfrak r_1 + \mathfrak r_2}{2}u} \left(|n + h| e^{\frac{1}{2}u} + |n| e^{-\frac{1}{2}u}\right)^{i (\mathfrak r_1 - \mathfrak r_2)} dw du,
\end{align*}
where the last step uses the earlier lemma and the Barnes integral and sets
\[
  a_h(n) = \frac{|n|^2 + |n + h|^2 - h^2}{2 |n| |n + h|} = \textup{sgn}(n (n + h)).
\]

We now consider
\[
  \int\limits_{-\infty}^{\infty} (a_h(n) + \textup{cosh}(u))^w e^{-i\frac{\mathfrak r_1 + \mathfrak r_2}{2}u} \left(|n + h| e^{\frac{1}{2}u} + |n| e^{-\frac{1}{2}u}\right)^{i (\mathfrak r_1 - \mathfrak r_2)} du.
\]
Using the substitution $t = e^{-u}$ yields
\[
  \int\limits_0^{\infty} \left(a_h(n) + \frac{t + t^{-1}}{2}\right)^w t^{i\frac{\mathfrak r_1 + \mathfrak r_2}{2} - 1} \left(|n + h| t^{-\frac{1}{2}} + |n| t^{\frac{1}{2}}\right)^{i (\mathfrak r_1 - \mathfrak r_2)} dt.
\]
If $a_h(n) = 1$ (i.e. $n < -h$ or $n > 0$), then we have
\begin{align*}
  & \int\limits_0^{\infty} \left(a_h(n) + \frac{t + t^{-1}}{2}\right)^w t^{i\frac{\mathfrak r_1 + \mathfrak r_2}{2} - 1} \left(|n + h| t^{-\frac{1}{2}} + |n| t^{\frac{1}{2}}\right)^{i (\mathfrak r_1 - \mathfrak r_2)} dt \\ & \hskip 50pt = 2^{-w} \int\limits_0^{\infty} t^{i\frac{\mathfrak r_1 + \mathfrak r_2}{2} - 1} \left(t^{-\frac{1}{2}} + t^{\frac{1}{2}}\right)^{2w} \left(|n + h| t^{-\frac{1}{2}} + |n| t^{\frac{1}{2}}\right)^{i (\mathfrak r_1 - \mathfrak r_2)} dt \\ & \hskip 50pt = 2^{-w} |n + h|^{i (\mathfrak r_1 - \mathfrak r_2)} \int\limits_0^{\infty} t^{-w + i\mathfrak r_2 - 1} (1 + t)^{2w} \left(1 + \frac{|n|}{|n + h|} t\right)^{i (\mathfrak r_1 - \mathfrak r_2)} dt.
\end{align*}
We now use the identity
\[
  \int\limits_0^{\infty} x^{\lambda - 1} (1 + x)^{\nu} (1 + \alpha x)^{\mu} dx \; = \; B(\lambda,-\mu - \nu - \lambda) \,F(-\mu,\lambda;-\mu - \nu;1 - \alpha),
\]
which holds for $|\textup{arg}(\alpha)| < \pi$ and $-\textup{Re}(\mu + \nu) > \textup{Re}(\lambda) > 0$ (3.197 equation 5 in \cite{MR1773820}). Letting $\lambda = -w + i\mathfrak r_2$, $\mu = i (\mathfrak r_1 - \mathfrak r_2)$, $\nu = 2w$, and $\alpha = \frac{|n|}{|n + h|}$ (and noting that the hypotheses of the identity are satisfied when $\textup{Re}(w) < 0$, which is the case that we are interested in), we obtain
\begin{align*}
  & \int\limits_0^{\infty} t^{-w + i\mathfrak r_2 - 1} (1 + t)^{2w} \left(1 + \tfrac{|n|}{|n + h|} t\right)^{i (\mathfrak r_1 - \mathfrak r_2)} dt \\ & \hskip 25pt = B(-w + i\mathfrak r_2,-w - i\mathfrak r_1)  \cdot F\left(-i (\mathfrak r_1 - \mathfrak r_2),-w + i\mathfrak r_2;-2w - i (\mathfrak r_1 - \mathfrak r_2);1 - \tfrac{|n|}{|n + h|}\right).
\end{align*}
If $a_h(n) = -1$ (i.e. $-h < n < 0$), then we have
\begin{align*}
  & \int\limits_0^{\infty} \left(a_h(n) + \frac{t + t^{-1}}{2}\right)^w t^{i\frac{\mathfrak r_1 + \mathfrak r_2}{2} - 1} \left(|n + h| t^{-\frac{1}{2}} + |n| t^{\frac{1}{2}}\right)^{i (\mathfrak r_1 - \mathfrak r_2)} dt \\ & \hskip 50pt = 2^{-w} \int\limits_0^1 t^{i\frac{\mathfrak r_1 + \mathfrak r_2}{2} - 1} \left(t^{-\frac{1}{2}} - t^{\frac{1}{2}}\right)^{2w} \left(|n + h| t^{-\frac{1}{2}} + |n| t^{\frac{1}{2}}\right)^{i (\mathfrak r_1 - \mathfrak r_2)} dt \\ & \hskip 100pt + 2^{-w} \int\limits_0^1 t^{-i\frac{\mathfrak r_1 + \mathfrak r_2}{2} - 1} \left(t^{-\frac{1}{2}} - t^{\frac{1}{2}}\right)^{2w} \left(|n| t^{-\frac{1}{2}} + |n + h| t^{\frac{1}{2}}\right)^{i (\mathfrak r_1 - \mathfrak r_2)} dt \\ & \hskip 50pt = 2^{-w} |n + h|^{i (\mathfrak r_1 - \mathfrak r_2)} \int\limits_0^1 t^{-w + i\mathfrak r_2 - 1} (1 - t)^{2w} \left(1 + \tfrac{|n|}{|n + h|} t\right)^{i (\mathfrak r_1 - \mathfrak r_2)} dt \\ & \hskip 100pt + 2^{-w} |n|^{i (\mathfrak r_1 - \mathfrak r_2)} \int\limits_0^1 t^{-w - i\mathfrak r_1 - 1} (1 - t)^{2w} \left(1 + \tfrac{|n + h|}{|n|} t\right)^{i (\mathfrak r_1 - \mathfrak r_2)} dt.
\end{align*}
We now use the integral representation
\[
  \int\limits_0^1 x^{b - 1} (1 - x)^{c - b - 1} (1 - zx)^{-a} dx = B(b,c - b) F(a,b;c;z),
\]
which holds if $\textup{Re}(c) > \textup{Re}(b) > 0$ and $z$ is not a real number greater than or equal to 1. If $\textup{Re}(w) > -\frac{1}{2}$ (which is why we have to move the line of integration from $\textup{Re}(w) = -\frac{1}{2} - \varepsilon$ to $\textup{Re}(w) = -\frac{1}{2} + \varepsilon$ in the $-h < n < 0$ case), then we have
\begin{align*}
  & \int\limits_0^1 t^{-w + i\mathfrak r_2 - 1} (1 - t)^{2w} \left(t + \tfrac{|n + h|}{|n|}\right)^{i (\mathfrak r_1 - \mathfrak r_2)} dt \\ & \hskip 50pt = B(-w + i\mathfrak r_2,2w + 1) \\ & \hskip 100pt \cdot F\left(-i (\mathfrak r_1 - \mathfrak r_2),-w + i\mathfrak r_2;w + i\mathfrak r_2 + 1;-\tfrac{|n|}{|n + h|}\right)
\end{align*}
and
\begin{align*}
  & \int\limits_0^1 t^{-w - i\mathfrak r_1 - 1} (1 - t)^{2w} \left(t + \tfrac{|n|}{|n + h|}\right)^{i (\mathfrak r_1 - \mathfrak r_2)} dt \\ & \hskip 50pt = B(-w - i\mathfrak r_1,2w + 1) \\ & \hskip 100pt \cdot F\left(-i (\mathfrak r_1 - \mathfrak r_2),-w - i\mathfrak r_1;w - i\mathfrak r_1 + 1;-\tfrac{|n + h|}{|n|}\right).
\end{align*}
Thus if $a_h(n) = 1$ (i.e. $n < -h$ or $n > 0$), then we have
\begin{align*}
  & \int\limits_{-\infty}^{\infty} (a_h(n) + \textup{cosh}(u))^w e^{-i\frac{\mathfrak r_1 + \mathfrak r_2}{2}u} \left(|n + h| e^{\frac{1}{2}u} + |n| e^{-\frac{1}{2}u}\right)^{i (\mathfrak r_1 - \mathfrak r_2)} du \\ & \hskip 50pt = 2^{-w} |n + h|^{i (\mathfrak r_1 - \mathfrak r_2)} B(-w + i\mathfrak r_2,-w - i\mathfrak r_1) \\ & \hskip 100pt \cdot F\left(-i (\mathfrak r_1 - \mathfrak r_2),-w + i\mathfrak r_2;-2w - i (\mathfrak r_1 - \mathfrak r_2);1 - \tfrac{|n|}{|n + h|}\right) \\
  & \hskip 50pt = 2^{-w} |n + h|^{i (\mathfrak r_1 - \mathfrak r_2)} B(-w + i\mathfrak r_2,-w - i\mathfrak r_1) \\ & \hskip 100pt \cdot F\left(-i (\mathfrak r_1 - \mathfrak r_2),-w + i\mathfrak r_2;-2w - i (\mathfrak r_1 - \mathfrak r_2);\tfrac{h}{n + h}\right),
\end{align*}
and if $a_h(n) = -1$ (i.e. $-h < n < 0$), then we have
\begin{align*}
  & \int\limits_{-\infty}^{\infty} (a_h(n) + \textup{cosh}(u))^w e^{-i\frac{\mathfrak r_1 + \mathfrak r_2}{2}u} \left(|n + h| e^{\frac{1}{2}u} + |n| e^{-\frac{1}{2}u}\right)^{i (\mathfrak r_1 - \mathfrak r_2)} du \\ & \hskip 50pt = 2^{-w} |n + h|^{i (\mathfrak r_1 - \mathfrak r_2)} B(-w + i\mathfrak r_2,2w + 1) \\ & \hskip 100pt \cdot F\left(-i (\mathfrak r_1 - \mathfrak r_2),-w + i\mathfrak r_2;w + i\mathfrak r_2 + 1;-\tfrac{|n|}{|n + h|}\right) \\ & \hskip 75pt + 2^{-w} |n|^{i (\mathfrak r_1 - \mathfrak r_2)} B(-w - i\mathfrak r_1,2w + 1) \\ &\hskip 100pt \cdot F\left(-i (\mathfrak r_1 - \mathfrak r_2),-w - i\mathfrak r_1;w - i\mathfrak r_1 + 1;-\tfrac{|n + h|}{|n|}\right) \\
  & \hskip 50pt = 2^{-w} |n + h|^{i (\mathfrak r_1 - \mathfrak r_2)} B(-w + i\mathfrak r_2,2w + 1) \\ & \hskip 100pt \cdot F\left(-i (\mathfrak r_1 - \mathfrak r_2),-w + i\mathfrak r_2;w + i\mathfrak r_2 + 1;\tfrac{n}{n + h}\right) \\ & \hskip 75pt + 2^{-w} |n|^{i (\mathfrak r_1 - \mathfrak r_2)} B(-w - i\mathfrak r_1,2w + 1) \\ &\hskip 100pt \cdot F\left(-i (\mathfrak r_1 - \mathfrak r_2),-w - i\mathfrak r_1;w - i\mathfrak r_1 + 1;\tfrac{n + h}{n}\right).
\end{align*}
In the latter case, applying the identity
\begin{align*}
  & F(a,b;c;z) = \frac{\Gamma(b - a) \Gamma(c)}{\Gamma(c - a) \Gamma(b)} (-z)^{-a} F\left(a,a - c + 1;a - b + 1;\tfrac{1}{z}\right) \\ & \hskip 100pt + \frac{\Gamma(a - b) \Gamma(c)}{\Gamma(c - b) \Gamma(a)} (-z)^{-b} F\left(b - c + 1,b;-a + b + 1;\tfrac{1}{z}\right),
\end{align*}
which holds if $z \not\in (0,1)$, yields
\begin{align*}
  & \int\limits_{-\infty}^{\infty} (a_h(n) + \textup{cosh}(u))^w e^{-i\frac{\mathfrak r_1 + \mathfrak r_2}{2}u} \left(|n + h| e^{\frac{1}{2}u} + |n| e^{-\frac{1}{2}u}\right)^{i (\mathfrak r_1 - \mathfrak r_2)} du \\ & \hskip 50pt = 2^{-w} |n + h|^{i (\mathfrak r_1 - \mathfrak r_2)} (B(-w + i\mathfrak r_2,2w + 1) + B(-w - i\mathfrak r_2,2w + 1)) \\ & \hskip 100pt \cdot F\left(-i (\mathfrak r_1 - \mathfrak r_2),-w + i\mathfrak r_2;w + i\mathfrak r_2 + 1;\tfrac{n}{n + h}\right) \\ & \hskip 75pt + 2^{-w} |n + h|^{i (\mathfrak r_1 - \mathfrak r_2)} \left|\tfrac{n}{n + h}\right|^{-w - i\mathfrak r_2} B(w + i\mathfrak r_2,-w - i\mathfrak r_1) \\ & \hskip 100pt \cdot F\left(-2w,-w - i\mathfrak r_1;-w - i\mathfrak r_2 + 1;\tfrac{n}{n + h}\right).
\end{align*}
In the former case, we now use the identity
\begin{align*}
  F(a,b;c;z) & = \frac{\Gamma(c) \Gamma(c - a - b)}{\Gamma(c - a) \Gamma(c - b)} F(a,b;a + b + 1 - c;1 - z) \\ & \hskip 40pt + \frac{\Gamma(c) \Gamma(a + b - c)}{\Gamma(a) \Gamma(b)} (1 - z)^{c - a - b} F(c - a,c - b; 1 + c - a - b;1 - z),
\end{align*}
which holds for $c - a - b \not\in \mathbb Z$, to obtain
\begin{align*}
  & F\left(-i (\mathfrak r_1 - \mathfrak r_2),-w + i\mathfrak r_2;-2w - i (\mathfrak r_1 - \mathfrak r_2);\tfrac{h}{n + h}\right) \\ & \hskip 50pt = \frac{\Gamma(-2w - i (\mathfrak r_1 - \mathfrak r_2)) \Gamma(-w - i\mathfrak r_2)}{\Gamma(-2w) \Gamma(-w - i\mathfrak r_1)} \\ & \hskip 100pt \cdot F\left(-i (\mathfrak r_1 - \mathfrak r_2),-w + i\mathfrak r_2;w + i\mathfrak r_2 + 1;\tfrac{n}{n + h}\right) \\ & \hskip 75pt + \frac{\Gamma(-2w - i (\mathfrak r_1 - \mathfrak r_2)) \Gamma(w + i\mathfrak r_2)}{\Gamma(-i (\mathfrak r_1 - \mathfrak r_2)) \Gamma(-w + i\mathfrak r_2)} \left(\tfrac{n}{n + h}\right)^{-w - i\mathfrak r_2} \\ & \hskip 100pt \cdot F\left(-2w,-w - i\mathfrak r_1;-w - i\mathfrak r_2 + 1;\tfrac{n}{n + h}\right).
\end{align*}
Thus
\begin{align*}
  & B(-w + i\mathfrak r_2,-w - i\mathfrak r_1) F\left(-i (\mathfrak r_1 - \mathfrak r_2),-w + i\mathfrak r_2;-2w - i (\mathfrak r_1 - \mathfrak r_2);\tfrac{h}{n + h}\right) \\ & \hskip 30pt = B(-w + i\mathfrak r_2,-w - i\mathfrak r_2)\cdot F\left(-i (\mathfrak r_1 - \mathfrak r_2),-w + i\mathfrak r_2;w + i\mathfrak r_2 + 1;\tfrac{n}{n + h}\right) \\ & \hskip 60pt + \left(\tfrac{n}{n + h}\right)^{-w - i\mathfrak r_2} B(w + i\mathfrak r_2,-w - i\mathfrak r_1)
    \cdot F\left(-2w,-w - i\mathfrak r_1;-w - i\mathfrak r_2 + 1;\tfrac{n}{n + h}\right).
\end{align*}

Plugging the above computations into the earlier sum immediately completes the proof of Proposition \ref{PropPhi1Phi2}.
\qed

\vskip 10pt
We now use the preceding proposition to prove the following theorem.

\begin{theorem}\label{TheoremPhi1Phi2}
  The function $L_h^{\#}(s,\Phi_1,\Phi_2)$ has meromorphic continuation to\linebreak ${\rm Re}(s)=\sigma > 0$ with possible simple poles at $s = \frac{1}{2} \pm ir_k$ and no other poles in that region and satisfies the bound
  \[\boxed{
    L_h^{\#}(s,\Phi_1,\Phi_2) \ll h^{\frac{1}{2} - \sigma + \theta + \varepsilon} |t|^{\textup{max}\left(1 - \sigma,\frac{1}{2}\right)} + h^{1 - \sigma + \varepsilon} |t|^{-\frac{1}{2} + \varepsilon}}
  \]
  for $|t| \rightarrow \infty$ with $\varepsilon < \sigma < 1 + \varepsilon$ fixed and $|s - \rho| > \varepsilon$ for every pole $\rho$.
\end{theorem}

\vskip 10pt

\noindent{\it Proof of Theorem \ref{TheoremPhi1Phi2}.}
Assume that $1 + 2\varepsilon < \sigma < 1 + 4\varepsilon$. We can then evaluate $\mathcal G_h(s,\Phi_1,\Phi_2)$ by shifting the line of integration to $\textup{Re}(w) = -\frac{1}{2} - 2\varepsilon$. This crosses over a simple pole at $w = -\frac{1}{2} s$ and no other poles. Thus
\[
  \mathcal G_h(s,\Phi_1,\Phi_2) = \mathcal R_h(s,\Phi_1,\Phi_2) + \mathcal G_{h,-\frac{1}{2} - 2\varepsilon}(s,\Phi_1,\Phi_2),
\]
where $\mathcal G_{h,-\frac{1}{2} - 2\varepsilon}(s,\Phi_1,\Phi_2)$ is $\mathcal G_h(s,\Phi_1,\Phi_2)$ with the line of integration that is shifted to $\textup{Re}(w) = -\frac{1}{2} - 2\varepsilon$. Furthermore, the residue from the pole at $w = -\frac{1}{2} s$ is given by
\begin{align*}
  \mathcal R_h(s,\Phi_1,\Phi_2) & =  \frac{\pi^{\frac{1}{2}} 2^{-\frac{1}{2} s} h^s\,\Gamma\left(\frac{1}{2} s\right)}{\Gamma\left(\frac{1}{2} s + \frac{1}{2}\right)} 
  \hskip-5pt
  \sum_{n < -h \textup{ or } n > 0}
  \hskip-8pt
  \frac{C_1(n) C_2(n + h)}{|n (n + h)|^{\frac{1}{2} s} }\, |n + h|^{i (\mathfrak r_1 - \mathfrak r_2)} \,  \mathcal F_{+,\mathfrak r_1,\mathfrak r_2,\frac{n}{n + h}} \left(\tfrac{-s}{2}\right) \\
                                &\hskip-40pt =  \frac{\pi^{\frac{1}{2}} h^s\,\Gamma\left(\frac{1}{2} s\right)}{\Gamma\left(\frac{1}{2} s + \frac{1}{2}\right)} \sum_{n < -h \textup{ or } n > 0} C_1(n) C_2(n + h) |n|^{-\frac{1}{2} s} |n + h|^{-\frac{1}{2} s + i (\mathfrak r_1 - \mathfrak r_2)}
                                 \\ &
                                  \hskip 160pt
              \cdot \left(\mathcal M_{+,\mathfrak r_1,\mathfrak r_2,\frac{n}{n + h}}\left(-\tfrac{1}{2} s\right) + \mathcal E_{\mathfrak r_1,\mathfrak r_2,\frac{n}{n + h}}\left(-\tfrac{1}{2} s\right)\right) \\
                                &\hskip-40pt = \pi^{\frac{1}{2}} h^s \frac{\Gamma\left(\frac{1}{2} s\right)}{\Gamma\left(\frac{1}{2} s + \frac{1}{2}\right)} \Bigg(L_h^{\#}(s,\Phi_1,\Phi_2) - \sum_{-h < n < 0} C_1(n) C_2(n + h) |n|^{-\frac{1}{2} s} |n + h|^{-\frac{1}{2} s + i (\mathfrak r_1 - \mathfrak r_2)} \\ & \hskip 160pt \cdot \left(\mathcal M_{+,\mathfrak r_1,\mathfrak r_2,\frac{n}{n + h}}\left(-\tfrac{1}{2} s\right) + \mathcal E_{\mathfrak r_1,\mathfrak r_2,\frac{n}{n + h}}\left(-\tfrac{1}{2} s\right)\right)\Bigg),
\end{align*}
where
\begin{align*}
  L_h^{\#}(s,\Phi_1,\Phi_2) & = \sum_{n \neq -h,0} 
  \frac{C_1(n) C_2(n + h)}{|n|^{\frac{1}{2} s} |n + h|^{\frac{1}{2} s-i (\mathfrak r_1 - \mathfrak r_2)}}
  \left(\mathcal M_{+,\mathfrak r_1,\mathfrak r_2,\frac{n}{n + h}}\left(-\tfrac{1}{2} s\right) + \mathcal E_{\mathfrak r_1,\mathfrak r_2,\frac{n}{n + h}}\left(-\tfrac{1}{2} s\right)\right).
\end{align*}

Therefore, we have
\begin{align*}
  \left\langle P_h(*,s),\overline{\Phi_1} \Phi_2 \right\rangle & = \frac{ \Gamma\left(\tfrac{1}{2} s\right)^2 }{4\pi^s} \Bigg(L_h^{\#}(s,\Phi_1,\Phi_2) 
   - \hskip-5pt\sum_{-h < n < 0}\hskip-5pt \frac{C_1(n) C_2(n + h)\cdot \mathcal E_{\mathfrak r_1,\mathfrak r_2,\frac{n}{n + h}}\left(-\tfrac{1}{2} s\right)}{|n|^{\frac{1}{2} s} |n + h|^{\frac{1}{2} s-i (\mathfrak r_1 - \mathfrak r_2)}}
  \\ 
  & 
  \hskip 75pt 
  - \sum_{-h < n < 0} \frac{C_1(n) C_2(n + h)}{|n|^{\frac{1}{2} s} |n + h|^{\frac{1}{2} s-i (\mathfrak r_1 - \mathfrak r_2)}} \mathcal M_{+,\mathfrak r_1,\mathfrak r_2,\frac{n}{n + h}}\left(-\tfrac{1}{2} s\right) \Bigg)
   \\ 
   & 
   \hskip -55pt 
   +\frac{\Gamma(s)}{2^{s +1} \pi^{s} h^{s}}\Gamma(s) \left(\frac{1}{2\pi i} \int\limits_{\textup{Re}(w) = -\frac{1}{2} - 2\varepsilon} \hskip-10pt h^{-2w} \right.\frac{\Gamma\left(\frac{1}{2} s + \frac{1}{2} + w\right) \Gamma\left(\frac{1}{2} s + w\right) \Gamma(-w)}{\Gamma\left(s + \frac{1}{2} + w\right)} 
   \\ 
   &
    \hskip -3pt
     \cdot \sum_{n < -h \textup{ or } n > 0}
      C_1(n) C_2(n + h) |n|^w |n + h|^{w + i (\mathfrak r_1 - \mathfrak r_2)} 
       \cdot \mathcal M_{+,\mathfrak r_1,\mathfrak r_2,\frac{n}{n + h}}(w) dw 
     \\ 
     & 
     \hskip -55pt 
     + \frac{1}{2\pi i} \int\limits_{\textup{Re}(w) = -\frac{1}{2} - 2\varepsilon} h^{-2w} \frac{\Gamma\left(\frac{1}{2} s + \frac{1}{2} + w\right) \Gamma\left(\frac{1}{2} s + w\right) \Gamma(-w)}{\Gamma\left(s + \frac{1}{2} + w\right)} 
     \\ 
     & 
     \hskip 15pt 
     \cdot \sum_{n < -h \textup{ or } n > 0} C_1(n) C_2(n + h) |n|^w |n + h|^{w + i (\mathfrak r_1 - \mathfrak r_2)} \cdot \mathcal E_{\mathfrak r_1,\mathfrak r_2,\frac{n}{n + h}}(w) dw 
     \\ 
     &
      \hskip -55pt 
      + \frac{1}{2\pi i} \int\limits_{\textup{Re}(w) = -\frac{1}{2} + \varepsilon} h^{-2w} \frac{\Gamma\left(\frac{1}{2} s + \frac{1}{2} + w\right) \Gamma\left(\frac{1}{2} s + w\right) \Gamma(-w)}{\Gamma\left(s + \frac{1}{2} + w\right)} 
      \\ 
      &
       \hskip 20pt \cdot \sum_{-h < n < 0} C_1(n) C_2(n + h) |n|^w |n + h|^{w + i (\mathfrak r_1 - \mathfrak r_2)}\cdot \mathcal M_{-,\mathfrak r_1,\mathfrak r_2,\frac{n}{n + h}}(w) dw 
       \\ 
       & 
       \hskip -55pt 
       + \frac{1}{2\pi i} \int\limits_{\textup{Re}(w) = -\frac{1}{2} + \varepsilon} h^{-2w} \frac{\Gamma\left(\frac{1}{2} s + \frac{1}{2} + w\right) \Gamma\left(\frac{1}{2} s + w\right) \Gamma(-w)}{\Gamma\left(s + \frac{1}{2} + w\right)} 
       \\ 
       & 
       \hskip 20pt \cdot \sum_{-h < n < 0} C_1(n) C_2(n + h) |n|^w |n + h|^{w + i (\mathfrak r_1 - \mathfrak r_2)} \cdot \mathcal E_{\mathfrak r_1,\mathfrak r_2,\frac{n}{n + h}}(w)\, dw \hskip-62pt\left. \phantom{\int\limits_{\textup{Re}(w) = -\frac{1}{2} - 2\varepsilon}} \right).
\end{align*}

We will now find meromorphic continuation and bounds for $5\varepsilon < \sigma \leq 1 + 3\varepsilon$ for all terms in this identity other than $L_h^{\#}(s,\Phi_1,\Phi_2)$.

Recall the following result.

\begin{theorem}
  The inner product $\left\langle P_h(*,s),\overline{\Phi_1} \Phi_2 \right\rangle$, which is defined for $\sigma > 1$, has meromorphic continuation to $\sigma > 0$, with possible simple poles at $s = \frac{1}{2} \pm ir_k$ and no other poles in that region. For $\sigma$ fixed and $|t| \rightarrow \infty$, we have the bound
  \[
    \left\langle P_h(*,s),\overline{\Phi_1} \Phi_2 \right\rangle \ll h^{\frac{1}{2} - \sigma + \theta + \varepsilon} |t|^{\textup{max}\left(\sigma - \frac{1}{2},0\right)} e^{-\frac{\pi}{2} |t|}
  \]
  for $\sigma > \varepsilon$ and $|s - \rho| > \varepsilon$ for every pole $\rho$.
\end{theorem}

To continue the proof of Theorem \ref{TheoremPhi1Phi2}
let $\textup{Re}(w) = -\eta$ and $\textup{Im}(w) = y$.
\vskip 8pt
 For $a < 0$ and  $0 < \rho < \eta$
 we have
\begin{align*}
  \mathcal E_{\mathfrak r_1,\mathfrak r_2,a}(w) & = |a|^{-w - i\mathfrak r_2} B(w + i\mathfrak r_2,-w - i\mathfrak r_1) F(-2w,-w - i\mathfrak r_1;-w - i\mathfrak r_2 + 1;a) 
  \\ 
                            & = |a|^{-w - i\mathfrak r_2} \frac{\Gamma(w + i\mathfrak r_2) \Gamma(-w - i\mathfrak r_1)}{\Gamma(-i (\mathfrak r_1 - \mathfrak r_2))} \cdot \frac{\Gamma(-w - i\mathfrak r_2 + 1)}{\Gamma(-2w) \Gamma(-w - i\mathfrak r_1)} \\ & \hskip 100pt \cdot \frac{1}{2\pi i} \int\limits_{\textup{Re}(v) = -\rho} \frac{\Gamma(-2w + v) \Gamma(-w - i\mathfrak r_1 + v) \Gamma(-v)}{\Gamma(-w - i\mathfrak r_2 + 1 + v)} |a|^v dv \\
                            & = \frac{1}{\Gamma(-i (\mathfrak r_1 - \mathfrak r_2))} |a|^{-w - i\mathfrak r_2} \frac{\pi}{\textup{sin}(\pi (w + i\mathfrak r_2))} \cdot \frac{1}{\Gamma(-2w)} \\ & \hskip 80pt \cdot \frac{1}{2\pi i} \int\limits_{\textup{Re}(v) = -\rho} \frac{\Gamma(-2w + v) \Gamma(-w - i\mathfrak r_1 + v) \Gamma(-v)}{\Gamma(-w - i\mathfrak r_2 + 1 + v)} |a|^v dv \\
                            & \ll |a|^{\eta - \rho} (1 + |y|)^{-2\eta + \frac{1}{2}} \int\limits_{-\infty}^{\infty} \frac{(1 + |\hskip-3pt-2y + u|)^{2\eta - \rho - \frac{1}{2}}\cdot (1 + |u|)^{\rho - \frac{1}{2}} }{(1 + |\hskip-3pt-y + u|)  \cdot  e^{\frac{\pi}{2} (|-2y + u| + |u|)}         } \;  du \\                             & \ll |a|^{\eta - \rho} (1 + |y|)^{-\frac{1}{2}} \textup{log}(1 + |y|) e^{-\pi |y|}.
\end{align*}
\vskip 10pt
For $0 < a < 1$, we have
\begin{align*}
  F(-2w,-w - i\mathfrak r_1;-w - i\mathfrak r_2 + 1;a) &= \frac{F\left(w - i\mathfrak r_2 + 1,-w - i\mathfrak r_1;-w - i\mathfrak r_2 + 1;-\tfrac{a}{1 - a}\right)}{(1 - a)^{-w - i\mathfrak r_1}   }\end{align*}.
  
 \vskip 10pt 
   
It follows that
\begin{align*}
  \mathcal E_{\mathfrak r_1,\mathfrak r_2,a}(w) & = |a|^{-w - i\mathfrak r_2} B(w + i\mathfrak r_2,-w - i\mathfrak r_1) F(-2w,-w - i\mathfrak r_1;-w - i\mathfrak r_2 + 1;a)
  &\\
   \\
                            &\hskip-30pt = a^{-w - i\mathfrak r_2}  B(w + i\mathfrak r_2,-w - i\mathfrak r_1) 
                             \cdot \frac{F\left(w - i\mathfrak r_2 + 1,-w - i\mathfrak r_1;-w - i\mathfrak r_2 + 1;-\tfrac{a}{1 - a}\right)}{(1 - a)^{-w - i\mathfrak r_1}  } \\
                             & \\
                            &\hskip-30pt = a^{-w - i\mathfrak r_2} (1 - a)^{w + i\mathfrak r_1} \frac{\Gamma(w + i\mathfrak r_2) \Gamma(-w - i\mathfrak r_1)}{\Gamma(-i (\mathfrak r_1 - \mathfrak r_2))} \cdot \frac{\Gamma(-w - i\mathfrak r_2 + 1)}{\Gamma(w - i\mathfrak r_2 + 1) \Gamma(-w - i\mathfrak r_1)} 
                            \\ 
                            & 
                            \hskip 20pt
                             \cdot \left(\frac{1}{2\pi i} \int\limits_{\textup{Re}(v) = -\rho} \right. \frac{\Gamma(w - i\mathfrak r_2 + 1 + v) \Gamma(-w - i\mathfrak r_1 + v) \Gamma(-v)}{\Gamma(-w - i\mathfrak r_2 + 1 + v)} \left(\tfrac{a}{1 - a}\right)^v dv 
                             \\ 
                             & 
                             \hskip 62pt 
                             + D(\eta,\rho) \left(\tfrac{a}{1 - a}\right)^{-w + i\mathfrak r_2 - 1} \frac{\Gamma(-2w - i (\mathfrak r_1 - \mathfrak r_2) - 1) \Gamma(w - i\mathfrak r_2 + 1)}{\Gamma(-2w)}\hskip-45pt\left. \phantom{\int\limits_{\textup{Re}(v) = -\rho}} \right) \\
                            &\hskip-30pt = \frac{a^{-w - i\mathfrak r_2} (1 - a)^{w + i\mathfrak r_1}}{\Gamma(-i (\mathfrak r_1 - \mathfrak r_2))}  \frac{\pi}{\textup{sin}(\pi (w + i\mathfrak r_2))} \cdot \frac{1}{\Gamma(w - i\mathfrak r_2 + 1)} 
                            \\ 
                            & 
                            \hskip 30pt
                             \cdot \left(\frac{1}{2\pi i} \int\limits_{\textup{Re}(v) = -\rho}\right.
                              \frac{\Gamma(w - i\mathfrak r_2 + 1 + v) \Gamma(-w - i\mathfrak r_1 + v) \Gamma(-v)}{\Gamma(-w - i\mathfrak r_2 + 1 + v)} \left(\tfrac{a}{1 - a}\right)^v dv 
                            \\
                             &
                              \hskip 60pt
                              + D(\eta,\rho) \left(\tfrac{a}{1 - a}\right)^{-w + i\mathfrak r_2 - 1} \frac{\Gamma(-2w - i (\mathfrak r_1 - \mathfrak r_2) - 1) \Gamma(w - i\mathfrak r_2 + 1)}{\Gamma(-2w)}\hskip-41pt \left.\phantom{\int\limits_{\textup{Re}(v) = -\rho}}\right) \\
                          &\hskip-30pt \ll \; \left(\tfrac{a}{1 - a}\right)^{\eta - \rho} (1 + |y|)^{\eta - \frac{1}{2}} e^{-\frac{\pi}{2} |y|} 
                            \\ 
                            &
                             \hskip 
                             10pt \cdot \left(\;\int\limits_{-\infty}^{\infty} \right.(1 + |-y + u|)^{-1} (1 + |u|)^{\rho - \frac{1}{2}} (1 + |y + u|)^{-\eta - \rho + \frac{1}{2}} e^{-\frac{\pi}{2} (|u| + |y + u|)} du \\ & \hskip 210pt + D(\eta,\rho) \left(\tfrac{a}{1 - a}\right)^{\eta - 1} (1 + |y|)^{-\eta} e^{-\frac{\pi}{2}}\hskip-18pt\left.\phantom{\int\limits_{-\infty}^{\infty}} \right) \\
                            &\hskip-30pt \ll \;\left(\tfrac{a}{1 - a}\right)^{\eta - \rho} \left(1 + D(\eta,\rho) \left(\tfrac{a}{1 - a}\right)^{\eta - 1}\right)\big (1 + |y|\big)^{-\frac{1}{2}} \,\textup{log}(1 + |y|)\, e^{-\pi |y|},
\end{align*}
where $0 < \rho < \eta$, $\rho \neq 1 - \eta$, and $D(\eta,\rho)$ is 1 if $\rho > 1 - \eta$ and 0 otherwise.

\vskip 10pt
For $a > 1$, we have
\begin{align*}
  & F(-2w,-w - i\mathfrak r_1;-w - i\mathfrak r_2 + 1;a) 
  \\ 
  & 
  \hskip 20pt  
  = (1 - a)^{2w + i (\mathfrak r_1 - \mathfrak r_2) + 1} \frac{\Gamma(-w - i\mathfrak r_2 + 1) \Gamma(-2w - i (\mathfrak r_1 - \mathfrak r_2) - 1)}{\Gamma(-2w) \Gamma(-w - i\mathfrak r_1)} 
  \\ 
  & 
  \hskip 150pt 
  \cdot F(w - i\mathfrak r_2 + 1,i (\mathfrak r_1 - \mathfrak r_2) + 1;2w + i (\mathfrak r_1 - \mathfrak r_2) + 2;1 - a) \\ & \hskip 75pt + \frac{\Gamma(-w - i\mathfrak r_2 + 1) \Gamma(2w + i (\mathfrak r_1 - \mathfrak r_2) + 1)}{\Gamma(w - i\mathfrak r_2 + 1) \Gamma(i (\mathfrak r_1 - \mathfrak r_2) + 1)} \\ & \hskip 180pt \cdot F(-2w,-w - i\mathfrak r_1;-2w - i (\mathfrak r_1 - \mathfrak r_2);1 - a) \\
  &
  \\
  & \hskip 20pt = \frac{\pi}{\textup{sin}(\pi (2w + i (\mathfrak r_1 - \mathfrak r_2)))} \cdot \frac{\Gamma(-w - i\mathfrak r_2 + 1)}{\Gamma(-2w) \Gamma(-w - i\mathfrak r_1) \Gamma(w - i\mathfrak r_2 + 1) \Gamma(i (\mathfrak r_1 - \mathfrak r_2) + 1)}
   \\ 
   &
    \hskip 15pt   
    \cdot \left(  \frac{(1 - a)^{2w + i (\mathfrak r_1 - \mathfrak r_2) + 1}}{2\pi i}\hskip-10pt \int\limits_{\textup{Re}(v) = -\rho} \right.
    \hskip-10pt\frac{\Gamma(w - i\mathfrak r_2 + 1 + v) \Gamma(i (\mathfrak r_1 - \mathfrak r_2) + 1 + v) \Gamma(-v)}{\Gamma(2w + i (\mathfrak r_1 - \mathfrak r_2) + 2 + v)} (a - 1)^v dv \\ & \hskip 126pt - \frac{1}{2\pi i} \int\limits_{\textup{Re}(v) = -\rho} \frac{\Gamma(-2w + v) \Gamma(-w - i\mathfrak r_1 + v) \Gamma(-v)}{\Gamma(-2w - i (\mathfrak r_1 - \mathfrak r_2) + v)} (a - 1)^v dv \hskip-42pt\left.\phantom{\int\limits_{\textup{Re}(v) = -\rho}}\right) \\
  & \hskip 20pt \ll (a - 1)^{\textup{max}(1 - 2\eta,0) - \rho} (1 + |y|)^{\frac{1}{2}} \textup{log}(1 + |y|),
\end{align*}

so
\begin{align*}
  \mathcal E_{\mathfrak r_1,\mathfrak r_2,a}(w) & = |a|^{-w - i\mathfrak r_2} B(w + i\mathfrak r_2,-w - i\mathfrak r_1) F(-2w,-w - i\mathfrak r_1;-w - i\mathfrak r_2 + 1;a) \\
                            & \ll a^{\eta} (a - 1)^{\textup{max}(1 - 2\eta,0) - \rho} (1 + |y|)^{-\frac{1}{2}} \textup{log}(1 + |y|) e^{-\pi |y|} \\
                            & \ll (a - 1)^{\textup{max}(1 - \eta,\eta) - \rho} (1 + |y|)^{-\frac{1}{2}} \textup{log}(1 + |y|) e^{-\pi |y|},
\end{align*}
where $0 < \rho < \textup{min}(\eta,1 - \eta)$.

\pagebreak

By analogous arguments, we have
\[
  \mathcal M_{+,\mathfrak r_1,\mathfrak r_2,a}(w) \ll \left(|a|^{-\rho} + 1\right) (1 + |y|)^{-\frac{1}{2}} \textup{log}(1 + |y|)
\]
for $a < 0$, where $0 < \rho < \eta$;
\[
  \mathcal M_{+,\mathfrak r_1,\mathfrak r_2,a}(w) \ll \left(\left(\tfrac{a}{1 - a}\right)^{-\rho} + 1\right) (1 + |y|)^{-\frac{1}{2}} \textup{log}(1 + |y|)
\]
for $0 < a < 1$, where $0 < \rho < 1 - \eta$; and
\[
  \mathcal M_{+,\mathfrak r_1,\mathfrak r_2,a}(w) \ll (a - 1)^{\textup{max}(1 - 2\eta,0)} \left((a - 1)^{-\rho} + 1\right) (1 + |y|)^{-\frac{1}{2}} \textup{log}(1 + |y|)
\]
for $a > 1$, where $0 < \rho < \textup{min}(\eta,1 - \eta)$.

\vskip 10pt
Now we have
\begin{align*}
  & \frac{1}{2\pi i} \int\limits_{\textup{Re}(w) = -\frac{1}{2} + \varepsilon} h^{-2w} \frac{\Gamma\left(\frac{1}{2} s + \frac{1}{2} + w\right) \Gamma\left(\frac{1}{2} s + w\right) \Gamma(-w)}{\Gamma\left(s + \frac{1}{2} + w\right)} 
   \cdot \sum_{-h < n < 0} C_1(n) C_2(n + h)
   \\ 
  &
  \hskip 225pt
  \cdot |n|^w |n + h|^{w + i (\mathfrak r_1 - \mathfrak r_2)} \mathcal E_{\mathfrak r_1,\mathfrak r_2,\frac{n}{n + h}}(w) dw 
  \\ 
  &
   \hskip 30pt
    \ll h^{1 - 2\varepsilon} \sum_{-h < n < 0} \left|C_1(n) C_2(n + h)\right| |n|^{-\frac{1}{2} + \varepsilon} (n + h)^{-\frac{1}{2} + \varepsilon} \left(-1 + \tfrac{h}{|n|}\right)^{\rho} \\ & \hskip 100pt \cdot \int\limits_{-\infty}^{\infty} (1 + |t + y|)^{-\sigma + \frac{1}{2} - \varepsilon} \left(1 + \left|\tfrac{1}{2} t + y\right|\right)^{\sigma - \frac{3}{2} + 2\varepsilon} (1 + |y|)^{-\frac{1}{2} - \varepsilon + \delta} \\ & \hskip 275pt \cdot e^{-\frac{\pi}{2} \left(-|t + y| + 2 \left|\frac{1}{2} t + y\right| + 3 |y|\right)} dy \\
  & \hskip 30pt \ll h^{1 - 4\varepsilon + \delta} |t|^{-1 + \varepsilon}.
\end{align*}
\vskip 10pt\noindent
Furthermore
\begin{align*}
  & \frac{1}{2\pi i} \int\limits_{\textup{Re}(w) = -\frac{1}{2} + \varepsilon} h^{-2w} \frac{\Gamma\left(\frac{1}{2} s + \frac{1}{2} + w\right) \Gamma\left(\frac{1}{2} s + w\right) \Gamma(-w)}{\Gamma\left(s + \frac{1}{2} + w\right)} \\ & \hskip 100pt \cdot \sum_{-h < n < 0} C_1(n) C_2(n + h) |n|^w |n + h|^{w + i (\mathfrak r_1 - \mathfrak r_2)} \mathcal M_{-,\mathfrak r_1,\mathfrak r_2,\frac{n}{n + h}}(w) dw
  \end{align*}
 
 \pagebreak

  \begin{align*} \\ & \hskip 50pt \ll h^{1 - 2\varepsilon} \sum_{-h < n < 0} \left|C_1(n) C_2(n + h)\right| |n|^{-\frac{1}{2} + \varepsilon} (n + h)^{-\frac{1}{2} + \varepsilon} \left(\left(-1 + \tfrac{h}{|n|}\right)^{\rho} + 1\right) \\ & \hskip 100pt \cdot \int\limits_{-\infty}^{\infty} (1 + |t + y|)^{-\sigma + \frac{1}{2} - \varepsilon} \left(1 + \left|\tfrac{1}{2} t + y\right|\right)^{\sigma - \frac{3}{2} + 2\varepsilon} (1 + |y|)^{-\frac{1}{2} - \varepsilon + \delta} \\ & \hskip 275pt \cdot e^{-\frac{\pi}{2} \left(-|t + y| + 2 \left|\frac{1}{2} t + y\right| + 3 |y|\right)} dy \\
  & \hskip 50pt \ll h^{1 - 4\varepsilon + \delta} |t|^{-1 + \varepsilon}
  \end{align*}
\vskip 10pt\noindent
and
\begin{align*}
  & \frac{1}{2\pi i} \int\limits_{\textup{Re}(w) = -\frac{1}{2} - 2\varepsilon} h^{-2w} \frac{\Gamma\left(\frac{1}{2} s + \frac{1}{2} + w\right) \Gamma\left(\frac{1}{2} s + w\right) \Gamma(-w)}{\Gamma\left(s + \frac{1}{2} + w\right)} \\ & \hskip 100pt \cdot \sum_{n < -h \textup{ or } n > 0} C_1(n) C_2(n + h) |n|^w |n + h|^{w + i (\mathfrak r_1 - \mathfrak r_2)} \mathcal E_{\mathfrak r_1,\mathfrak r_2,\frac{n}{n + h}}(w) dw \\ & \hskip 50pt \ll h^{1 + 4\varepsilon} \sum_{n > 0} \left|C_1(n) C_2(n + h)\right| n^{-\frac{1}{2} - 2\varepsilon} (n + h)^{-\frac{1}{2} - 2\varepsilon} \left(\left(\tfrac{n}{h}\right)^{\frac{1}{2} + 2\varepsilon - \rho} + \left(\tfrac{n}{h}\right)^{4\varepsilon - \rho}\right) \\ & \hskip 100pt \cdot \int\limits_{-\infty}^{\infty} (1 + |t + y|)^{-\sigma + \frac{1}{2} - \varepsilon} \left(1 + \left|\tfrac{1}{2} t + y\right|\right)^{\sigma - \frac{3}{2} + 2\varepsilon} (1 + |y|)^{-\frac{1}{2} - \varepsilon + \delta} \\ & \hskip 275pt \cdot e^{-\frac{\pi}{2} \left(-|t + y| + 2 \left|\frac{1}{2} t + y\right| + 3 |y|\right)} dy \\
  & \hskip 50pt \ll h^{1 + 4\varepsilon + \delta} |t|^{-1 + \varepsilon}.
\end{align*}
\vskip 10pt\noindent
Furthermore
\begin{align*}
  & \frac{1}{2\pi i} \int\limits_{\textup{Re}(w) = -\frac{1}{2} - 2\varepsilon} h^{-2w} \frac{\Gamma\left(\frac{1}{2} s + \frac{1}{2} + w\right) \Gamma\left(\frac{1}{2} s + w\right) \Gamma(-w)}{\Gamma\left(s + \frac{1}{2} + w\right)} 
  \\ 
  & 
  \hskip 90pt
   \cdot \sum_{n < -h \textup{ or } n > 0} C_1(n) C_2(n + h) |n|^w |n + h|^{w + i (\mathfrak r_1 - \mathfrak r_2)} \mathcal M_{+,\mathfrak r_1,\mathfrak r_2,\frac{n}{n + h}}(w) dw 
  \\ 
  & 
  \hskip 10pt
   \ll h^{1 + 4\varepsilon} \sum_{n > 0} \frac{\left|C_1(n) C_2(n + h)\right|}{ n^{\frac{1}{2}+ 2\varepsilon} (n + h)^{\frac{1}{2} + 2\varepsilon}} \left(\left(\tfrac{h}{n}\right)^{\rho} + 1\right) 
      \int\limits_{-\infty}^{\infty} (1 + |t + y|)^{-\sigma + \frac{1}{2} - \varepsilon} \left(1 + \left|\tfrac{1}{2} t + y\right|\right)^{\sigma - \frac{3}{2} + 2\varepsilon} \\ & \hskip 230pt \cdot (1 + |y|)^{-\frac{1}{2} - \varepsilon + \delta} e^{-\frac{\pi}{2} \left(-|t + y| + 2 \left|\frac{1}{2} t + y\right| + |y|\right)} dy \\
  & \hskip 10pt \ll h^{1 + 4\varepsilon + \delta} |t|^{-1 + \varepsilon}.
\end{align*}
\vskip 10pt
\noindent
Next
\begin{align*}
  & \sum_{-h < n < 0} C_1(n) C_2(n + h) |n|^{-\frac{1}{2} s} |n + h|^{-\frac{1}{2} s + i (\mathfrak r_1 - \mathfrak r_2)} \mathcal E_{\mathfrak r_1,\mathfrak r_2,\frac{n}{n + h}}\left(-\tfrac{1}{2} s\right) \\ & \hskip 30pt \ll \sum_{-h < n < 0} \left|C_1(n) C_2(n + h)\right| |n|^{-\frac{1}{2} \sigma} (n + h)^{-\frac{1}{2} \sigma} \left(-1 + \tfrac{h}{|n|}\right)^{\frac{1}{2} \sigma - \rho} |t|^{-\frac{1}{2}} \textup{log}|t| e^{-\pi |t|} \\
  & \hskip 30pt \ll h^{\delta} |t|^{-\frac{1}{2}} \textup{log}|t| e^{-\pi |t|},
\end{align*}
and
\begin{align*}
  & \sum_{-h < n < 0} C_1(n) C_2(n + h) |n|^{-\frac{1}{2} s} |n + h|^{-\frac{1}{2} s + i (\mathfrak r_1 - \mathfrak r_2)} \mathcal M_{+,\mathfrak r_1,\mathfrak r_2,\frac{n}{n + h}}\left(-\tfrac{1}{2} s\right) \\ & \hskip 30pt \ll \sum_{-h < n < 0} \left|C_1(n) C_2(n + h)\right| |n|^{-\frac{1}{2} \sigma} (n + h)^{-\frac{1}{2} \sigma} \left(\left(-1 + \tfrac{h}{|n|}\right)^{\rho} + 1\right) |t|^{-\frac{1}{2}} \textup{log}|t| \\
  & \hskip 30pt \ll h^{-\frac{1}{2} \sigma + \delta} |t|^{-\frac{1}{2}} \textup{log}|t|.
\end{align*}

Therefore, for $5\varepsilon < \sigma < 1 + 3\varepsilon$ and $t \in \mathbb R$ with $|t| \rightarrow \infty$, we have
\[
  \left\langle P_h(*,s),\overline{\Phi_1} \Phi_2 \right\rangle = 2^{-\frac{1}{2} s - 2} \pi^{-s} \Gamma\left(\tfrac{1}{2} s\right)^2 L_h^{\#}(s,\Phi_1,\Phi_2) + \mathcal O\left(h^{1 - \sigma + \varepsilon} |t|^{\sigma - \frac{3}{2} + \varepsilon} e^{-\frac{\pi}{2} |t|}\right).
\]
The error term in the asymptotic formula is holomorphic. Because $\left\langle P_h(*,s),\overline{\Phi_1} \Phi_2 \right\rangle$ has meromorphic continuation to $\sigma > \varepsilon$ with possible simple poles at $s = \frac{1}{2} \pm ir_k$ and no other poles in that region and satisfies the bound
\[
  \left\langle P_h(*,s),\overline{\Phi_1} \Phi_2 \right\rangle \ll h^{\frac{1}{2} - \sigma + \theta + \varepsilon} |t|^{\textup{max}\left(\sigma - \frac{1}{2},0\right)} e^{-\frac{\pi}{2} |t|}
\]
in that region away from its poles, it immediately follows that $L_h^{\#}(s,\Phi_1,\Phi_2)$ has meromorphic continuation to $\sigma > 0$ with possible simple poles at $s = \frac{1}{2} \pm ir_k$ (with precisely the same of these possible poles as for $\left\langle P_h(*,s),\overline{\Phi_1} \Phi_2 \right\rangle$) and no other poles in that region and satisfies the bound
\[
  L_h^{\#}(s,\Phi_1,\Phi_2) \ll h^{\frac{1}{2} - \sigma + \theta + \varepsilon} |t|^{\textup{max}\left(1 - \sigma,\frac{1}{2}\right)} + h^{1 - \sigma + \varepsilon} |t|^{-\frac{1}{2} + \varepsilon}
\]
for $|t| \rightarrow \infty$ with $\varepsilon < \sigma < 1 + \varepsilon$ fixed and $|s - \rho| > \varepsilon$ for every pole $\rho$. This completes the proof of Theorem \ref{TheoremPhi1Phi2}. \qed

\vskip 15pt
We now obtain the following meromorphic continuation and bounds, which are precisely analogous to those obtained for the $\mathfrak r_1 = \mathfrak r_2$ case, except that the second term in the first two bounds has a $|t|^{\varepsilon}$ factor that was not present in that case.

\begin{theorem}
  The function $L_h(s,\Phi_1,\Phi_2)$ has meromorphic continuation to $\sigma > 0$ with possible simple poles at $s = \frac{1}{2} \pm ir_k$ and no other poles in that region and satisfies the bound
  \[
    L_h(s,\Phi_1,\Phi_2) \ll h^{\frac{1}{2} - \sigma + \theta + \varepsilon} |t|^{\textup{max}\left(\frac{3}{2} - \sigma,1\right)} + h^{1 - \sigma + \varepsilon} |t|^{\varepsilon}
  \]
  for $|t| \rightarrow \infty$ with $\varepsilon < \sigma < 1 + \varepsilon$ fixed and $|s - \rho| > \varepsilon$ for every pole $\rho$. By using a convexity argument for $\frac{1}{2} < \sigma < 1 + \varepsilon$ (on which $L_h(s,\Phi_1,\Phi_2)$ is holomorphic), we obtain the improved bound
  \[
    L_h(s,\Phi_1,\Phi_2) \ll \begin{cases}
                               h^{\frac{1}{2} - \sigma + \theta + \varepsilon} |t|^{\frac{3}{2} - \sigma + \varepsilon} + h^{1 - \sigma + \varepsilon} |t|^{\varepsilon} & \varepsilon < \sigma \leq \frac{1}{2} \\
                               h^{(2\theta + \varepsilon) (1 - \sigma + \varepsilon)} |t|^{2 (1 - \sigma + \varepsilon)} + h^{1 - \sigma + \varepsilon} |t|^{\varepsilon} & \frac{1}{2} \leq \sigma \leq 1 + \varepsilon \\
                               1 & \sigma \geq 1 + \varepsilon.
                             \end{cases}
  \]
\end{theorem}

\begin{proof}
  For $n < -2h$ or $n > 0$, we have the power series expansion
  \begin{align*}
    F\left(-i (\mathfrak r_1 - \mathfrak r_2),\tfrac{1}{2} s + i\mathfrak r_2;s - i (\mathfrak r_1 - \mathfrak r_2);\tfrac{h}{n + h}\right) & = \sum_{k = 0}^{\infty} \frac{(-i (\mathfrak r_1 - \mathfrak r_2))_k \left(\frac{1}{2} s + i\mathfrak r_2\right)_k}{(s - i (\mathfrak r_1 - \mathfrak r_2))_k} \frac{1}{k!} \left(\tfrac{h}{n + h}\right)^k \\
                                                                                          & = 1 + \sum_{k = 1}^{\infty} a_{\mathfrak r_1,\mathfrak r_2,k}(s) \left(\tfrac{h}{n + h}\right)^k.
  \end{align*}

  Thus for $\sigma > 1$, we have
  \begin{align*}
    & \sum_{n \neq -h,0} \frac{C_1(n) C_2(n + h)}{ |n|^{\frac{1}{2} s} |n + h|^{\frac{1}{2} s - i (\mathfrak r_1 - \mathfrak r_2)}} F\left(-i (\mathfrak r_1 - \mathfrak r_2),\tfrac{1}{2} s + i\mathfrak r_2;s - i (\mathfrak r_1 - \mathfrak r_2);\tfrac{h}{n + h}\right)
     \\ 
     & 
     \hskip 20pt = \sum_{n \neq -h,0} \frac{C_1(n) C_2(n + h)}{ |n|^{\frac{1}{2} s} |n + h|^{\frac{1}{2} s - i (\mathfrak r_1 - \mathfrak r_2)}} + \sum_{k = 1}^{\infty} a_{\mathfrak r_1,\mathfrak r_2,k}(s) h^k 
     \\
     &
     \hskip98pt
     \cdot
     \Bigg( \sum_{n < -2h} 
     \frac{(-1)^k\, C_1(n) C_2(n + h)}{ |n|^{\frac{1}{2} s} |n + h|^{\frac{1}{2} s + k - i (\mathfrak r_1 - \mathfrak r_2)}}      
      + \sum_{n > 0} \frac{C_1(n) C_2(n + h)}{ |n|^{\frac{1}{2} s} |n + h|^{\frac{1}{2} s + k - i (\mathfrak r_1 - \mathfrak r_2)}}
      \Bigg)
       \\ 
       &
        \hskip 80pt + \sum_{-2h \leq n < -h \textup{ or } -h < n < 0} \frac{C_1(n) C_2(n + h)}{ |n|^{\frac{1}{2} s} |n + h|^{\frac{1}{2} s - i (\mathfrak r_1 - \mathfrak r_2)}} \\ & \hskip 170pt \cdot \left(F\left(-i (\mathfrak r_1 - \mathfrak r_2),\tfrac{1}{2} s + i\mathfrak r_2;s - i (\mathfrak r_1 - \mathfrak r_2);\tfrac{h}{n + h}\right) - 1\right).
  \end{align*}
  Note that the sum on the left side of the above equation equals $$\frac{1}{B\left(\frac{1}{2} s + i\mathfrak r_2,\frac{1}{2} s - i\mathfrak r_1\right)} L_h^{\#}(s,\Phi_1,\Phi_2)$$while the first sum on the right side equals $L_h(s,\Phi_1,\Phi_2)$. The last sum on the right side is holomorphic for $\sigma > 0$ (as it is a finite sum of functions that are holomorphic for $\sigma > 0$) and from analogous computations to earlier is less than a constant times $h^{-\frac{1}{2} \sigma + \varepsilon} \textup{log}|t| + h^{\varepsilon} \textup{log}|t| e^{-\pi |t|}$. Specifically, for $-2h \leq n < -h$, we have
  \begin{align*}
    & F\left(-i (\mathfrak r_1 - \mathfrak r_2),\tfrac{1}{2} s + i\mathfrak r_2;s - i (\mathfrak r_1 - \mathfrak r_2);\tfrac{h}{n + h}\right) \\ & \hskip 50pt \ll \left|\tfrac{h}{n + h}\right|^{\textup{max}(1 - \sigma,0)} \left(\left|\tfrac{h}{n + h}\right|^{-\rho} + 1\right) \textup{log}|t| + \left|\tfrac{h}{n + h}\right|^{\textup{max}\left (1 - \frac{1}{2} \sigma,\frac{1}{2} \sigma\right) - \rho} \textup{log}|t| e^{-\pi |t|} \\
    & \hskip 50pt \ll \left|\tfrac{h}{n + h}\right|^{\textup{max}(1 - \sigma,0)} \textup{log}|t| + \left|\tfrac{h}{n + h}\right|^{\textup{max}\left(1 - \frac{1}{2} \sigma,\frac{1}{2} \sigma\right) - \rho} \textup{log}|t| e^{-\pi |t|},
  \end{align*}
  where $0 < \rho < \textup{min}\left(\tfrac{1}{2} \sigma,1 - \tfrac{1}{2} \sigma\right)$, and for $-h < n < 0$, we have
  \pagebreak

  \begin{align*}
    & F\left(-i (\mathfrak r_1 - \mathfrak r_2),\tfrac{1}{2} s + i\mathfrak r_2;s - i (\mathfrak r_1 - \mathfrak r_2);\tfrac{h}{n + h}\right) 
    \\ 
    & 
    \hskip 15pt 
    = \frac{\Gamma(s - i (\mathfrak r_1 - \mathfrak r_2)) \Gamma\left(\frac{1}{2} s - i\mathfrak r_2\right)}{\Gamma(s) \Gamma\left(\frac{1}{2} s  - i\mathfrak r_1\right)} F\left(-i (\mathfrak r_1 - \mathfrak r_2),\tfrac{1}{2} s + i\mathfrak r_2;-\tfrac{1}{2} s + i\mathfrak r_2 + 1;\tfrac{n}{n + h}\right) 
    \\ 
    & 
    \hskip 40pt
     + \frac{\Gamma(s - i (\mathfrak r_1 - \mathfrak r_2)) \Gamma\left(-\frac{1}{2} s + i\mathfrak r_2\right)}{\Gamma(-i (\mathfrak r_1 - \mathfrak r_2)) \Gamma\left(\frac{1}{2} s + i\mathfrak r_2\right)} \left(\tfrac{n}{n + h}\right)^{\frac{1}{2} s - i\mathfrak r_2}\cdot F\left(s, \tfrac{1}{2} s - i\mathfrak r_1;\tfrac{1}{2} s - i\mathfrak r_2 + 1;\tfrac{n}{n + h}\right)
     \\
    &\\&
     \hskip 15pt
      \ll \left(\left|\tfrac{n}{n + h}\right|^{-\rho} + 1\right) \textup{log}|t| + \left|\tfrac{n}{n + h}\right|^{\textup{max}\left(1 - \frac{1}{2} \sigma,\frac{1}{2} \sigma\right) - \rho} \textup{log}|t| e^{-\pi |t|} \\
      &\\
    & \hskip 15pt \ll \textup{log}|t| + \left|\tfrac{n}{n + h}\right|^{\textup{max}\left(1 - \frac{1}{2} \sigma,\frac{1}{2} \sigma\right) - \rho} \textup{log}|t| e^{-\pi |t|},
  \end{align*}
  where $0 < \rho < \tfrac{1}{2} \sigma$, allowing us to estimate the sums over $-2h \leq n < -h$ and $-h < n < 0$ analogously to before. It thus remains to analyze
  \[
    \sum_{n < -2h} C_1(n) C_2(n + h) |n|^{-\frac{1}{2} s} |n + h|^{-\frac{1}{2} s - k + i (\mathfrak r_1 - \mathfrak r_2)}
  \]
  and
  \[
    \sum_{n > 0} C_1(n) C_2(n + h) |n|^{-\frac{1}{2} s} |n + h|^{-\frac{1}{2} s - k + i (\mathfrak r_1 - \mathfrak r_2)}
  \]
  for $k \in \mathbb Z_{> 0}$. Because the arguments are precisely analogous, we explicitly discuss only the latter series. For each individual $k$, because the real part of the total power of $n$ in the summand is $-\sigma - k$, the sum is holomorphic for $\sigma > 0$ and bounded for $\sigma > \varepsilon$. Additionally, as $k \rightarrow \infty$, the sum of the absolute value is bounded in $k$. To determine the bound's dependence on $h$, note that
  \begin{align*}
    \sum_{n > 0} C_1(n) C_2(n + h) n^{-\frac{1}{2} s} (n + h)^{-\frac{1}{2} s - k + i (\mathfrak r_1 - \mathfrak r_2)} & \ll h^{-\frac{1}{2} \sigma - k} \sum_{0 < n \leq h} \left|C_1(n) C_2(n + h)\right| n^{-\frac{1}{2} \sigma} \\
                                                                                                              & \ll h^{-\frac{1}{2} \sigma - k + \textup{max}\left(1 - \frac{1}{2} \sigma + \varepsilon,0\right)}.
  \end{align*}
  It thus follows that the $h$-dependence of the bound for the overall sum over $k$, provided that the sum converges when the inner sums are replaced with the preceding bound, is $h^{\textup{max}\left(1 - \sigma + \varepsilon,-\frac{1}{2} \sigma\right)}$. Note that each $a_{\mathfrak r_1,\mathfrak r_2,k}(s)$ is holomorphic and bounded in $s$ for $\sigma > 0$. Furthermore, we have
  \begin{align*}
    a_{\mathfrak r_1,\mathfrak r_2,k}(s) & = \frac{(-i (\mathfrak r_1 - \mathfrak r_2))_k \left(\tfrac{1}{2} s + i\mathfrak r_2\right)_k}{(s - i (\mathfrak r_1 - \mathfrak r_2))_k} \frac{1}{k!} \\
                     & = \frac{\Gamma(-i (\mathfrak r_1 - \mathfrak r_2) + k) \Gamma\left(\frac{1}{2} s + i\mathfrak r_2 + k\right) \Gamma(s - i (\mathfrak r_1 - \mathfrak r_2))}{\Gamma(-i (\mathfrak r_1 - \mathfrak r_2)) \Gamma\left(\frac{1}{2} s + i\mathfrak r_2\right) \Gamma(s - i (\mathfrak r_1 - \mathfrak r_2) + k) \Gamma(k + 1)}.
  \end{align*}
  For $x \rightarrow \infty$ and $\alpha \in \mathbb C$, we have the asymptotic formula
  \[
    \Gamma(x + \alpha) \sim \Gamma(x) x^{\alpha},
  \]
  so for $k \rightarrow \infty$ with $s$ fixed, we have
  \[
    \frac{\Gamma(-i (\mathfrak r_1 - \mathfrak r_2) + k) \Gamma\left(\frac{1}{2} s + i\mathfrak r_2 + k\right)}{\Gamma(s - i (\mathfrak r_1 - \mathfrak r_2) + k) \Gamma(k + 1)} \sim \frac{k^{-i (\mathfrak r_1 - \mathfrak r_2)} k^{\frac{1}{2} s + i\mathfrak r_2}}{k^{s - i (\mathfrak r_1 - \mathfrak r_2)} k} = k^{-\frac{1}{2} s - 1 + i\mathfrak r_2}.
  \]
  For $\sigma > 0$, the exponent of the final expression in the above asymptotic formula has real part strictly less than $-1$, so the sum over $k$ still converges after replacing the inner sums with the uniform bounds chosen earlier. Thus for $\sigma > 0$, the original sum over $k$ is holomorphic and less than a constant times $h^{\textup{max}\left(1 - \sigma + \varepsilon,-\frac{1}{2} \sigma\right)}$.

  Combining the analytic continuations and bounds found above with the meromorphic continuation and bounds for $L_h^{\#}(s,\Phi_1,\Phi_2)$ and Stirling's formula yields the claimed meromorphic continuation and bounds for $L_h(s,\Phi_1,\Phi_2)$.
\end{proof}

We now compute the residues of $L_h(s,\Phi_1,\Phi_2)$ at the poles $s = \frac{1}{2} \pm ir_k$.

\begin{theorem}
  The residues of $L_h(s,\Phi_1,\Phi_2)$ at its poles $s = \frac{1}{2} \pm ir_k$ are given by
  \begin{align*}
    \underset{s = \frac{1}{2} \pm ir_k}{\textup{Res}}L_h(s,\Phi_1,\Phi_2) & = 2\pi^{\frac{1}{2}} h^{\mp ir_k} \frac{\Gamma(\pm ir_k) \Gamma\left(\frac{1}{2} \pm ir_k - i (\mathfrak r_1 - \mathfrak r_2)\right)\cdot c_k(h) \left\langle \phi_k,\overline{\Phi_1} \Phi_2 \right\rangle}{\Gamma\left(\frac{1}{4} \pm \frac{1}{2} ir_k\right)^2 \Gamma\left(\frac{1}{4} \pm \frac{1}{2} ir_k - i\mathfrak r_1\right) \Gamma\left(\frac{1}{4} \pm \frac{1}{2} ir_k + i\mathfrak r_2\right)} 
    .
  \end{align*}
\end{theorem}

\begin{proof}
  Analogously to the $\mathfrak r_1 = \mathfrak r_2$ case, we have
  \begin{align*}
    \underset{s = \frac{1}{2} \pm ir_k}{\textup{Res}}L_h(s,\Phi_1,\Phi_2) & = \frac{\underset{s = \frac{1}{2} \pm ir_k}{\textup{Res}} 4\pi^{\frac{1}{2} \pm ir_k}\langle P_h(*,s),\phi_k\rangle \left\langle \phi_k,\overline{\Phi_1} \Phi_2 \right\rangle}{\Gamma\left(\frac{1}{2} \left(\frac{1}{2} \pm ir_k\right)\right)^2 B\left(\frac{1}{2} \left(\frac{1}{2} \pm ir_k\right) + i\mathfrak r_2,\frac{1}{2} \left(\frac{1}{2} \pm ir_k\right) - i\mathfrak r_1\right)} 
    \\
  &
  \\                                                                          &\hskip-30pt = 4\pi^{\frac{1}{2} \pm ir_k} \frac{\Gamma\left(\frac{1}{2} \pm ir_k - i (\mathfrak r_1 - \mathfrak r_2)\right)\cdot \frac{1}{2} (\pi h)^{\mp ir_k} \Gamma(\pm ir_k) c_k(h) \left\langle \phi_k,\overline{\Phi_1} \Phi_2 \right\rangle}{\Gamma\left(\frac{1}{4} \pm \frac{1}{2} ir_k\right)^2 \Gamma\left(\frac{1}{4} \pm \frac{1}{2} ir_k - i\mathfrak r_1\right) \Gamma\left(\frac{1}{4} \pm \frac{1}{2} ir_k + i\mathfrak r_2\right)}
  \\ &  \\
                                                                            &\hskip-30pt = 2\pi^{\frac{1}{2}} h^{\mp ir_k} \frac{\Gamma(\pm ir_k) \Gamma\left(\frac{1}{2} \pm ir_k - i (\mathfrak r_1 - \mathfrak r_2)\right)\cdot c_k(h) \left\langle \phi_k,\overline{\Phi_1} \Phi_2 \right\rangle}{\Gamma\left(\frac{1}{4} \pm \frac{1}{2} ir_k\right)^2 \Gamma\left(\frac{1}{4} \pm \frac{1}{2} ir_k - i\mathfrak r_1\right) \Gamma\left(\frac{1}{4} \pm \frac{1}{2} ir_k + i\mathfrak r_2\right)}.
  \end{align*}
\end{proof}
\vskip -5pt\noindent
We now obtain the following results by arguments analogous  to the $\mathfrak r_1 = \mathfrak r_2$ case.

\begin{theorem} {\bf (Asymptotic formula for smoothed SCS)}
  Fix $0 < \varepsilon < \tfrac12$. Let $h$ be a positive integer.  Then for $T \rightarrow \infty$, we have
  \begin{align*}
   & \underset{n \neq 0,-h}{\sum_{\sqrt{|n (n + h)|}<T}} \hskip-12pt
   C_1(n) C_2(n + h) \left(\textup{log} \tfrac{T}{\sqrt{|n (n + h)|}}\, \right)^{\frac{3}{2} + \varepsilon} \hskip-9pt =f_{\mathfrak r_1,\mathfrak r_2,h,\varepsilon}(T) T^{\frac{1}{2}} +  \mathcal O\Big(h^{1-\varepsilon}\, T^{\varepsilon}+ h^{1+\varepsilon}\, T^{-1 - \varepsilon}\Big).
  \end{align*}
  Here $f_{\mathfrak r_1,\mathfrak r_2,h,\varepsilon}(T) \ll h^{\theta + \varepsilon}$, and more precisely,
  
  \pagebreak
  
  \begin{align*}
    &f_{\mathfrak r_1,\mathfrak r_2,h,\varepsilon}(T) = 4\pi^{\frac{1}{2}} \Gamma\left(\tfrac{5}{2} + \varepsilon\right) \sum_{k = 1}^{\infty} c_k(h) \left\langle \phi_k,\overline{\Phi_1} \Phi_2 \right\rangle
    \\
    & 
  \hskip 50pt
       \cdot \textup{Re}\left(  \frac{(T/h)^{i r_k}}{\left(\frac12+ir_k\right)^{\frac{5}{2}+\varepsilon}}\;
      \cdot \frac{\Gamma(ir_k) \Gamma\left(\frac{1}{2} + ir_k - i (\mathfrak r_1 - \mathfrak r_2)\right)}{\Gamma\left(\frac{1}{4} + \frac{1}{2} ir_k\right)^2 \Gamma\left(\frac{1}{4} + \frac{1}{2}ir_k - i\mathfrak r_1\right) \Gamma\left(\frac{1}{4} + \frac{1}{2}ir_k + i\mathfrak r_2\right)}\right)  \end{align*}
  which converges for all $T$ and satisfies $f_{\mathfrak r_1,\mathfrak r_2,h,\varepsilon}(T) \ll h^{\theta + \varepsilon}$. 
\end{theorem}
\begin{remark} It seems likely that
for any fixed positive integer $h$ and $0<\varepsilon<\tfrac12$, the function $f_{\mathfrak r_1,\mathfrak r_2,h,\varepsilon}(T)$  is never identically zero.
\end{remark}
\begin{theorem} {\bf (Upper bound for unsmoothed SCS)}
 Fix $\,0<\varepsilon<\tfrac12.$ Let $x\to\infty.$  Then for any positive integer $h<x^{\frac12-\varepsilon}$ we have
  \[
    \sum_{\sqrt{|n (n + h)|} < x} C_1(n) C_2(n + h) \ll h^{\frac{2}{3}\theta + \varepsilon}x^{\frac{2}{3} (1 + \theta) + \varepsilon} + h^{\frac{1}{2} + \varepsilon}x^{\frac{1}{2} + 2\theta + \varepsilon}.
  \]
\end{theorem}

That is, we have now proved the main theorems in the $\mathfrak r_1 \neq \mathfrak r_2$ case as well, thus completing the proofs of the overall results.

\subsection*{Acknowledgement.} The authors would like to thank the referee for their helpful comments. We would also like to thank Aditya Ghosh for providing additional comments and corrections that are reflected in this revised version of the paper.

\subsection*{Conflict of interest statement.} The authors have no conflicts of interest to declare.

\subsection*{Data availability statement.} Not applicable.

\bibliographystyle{amsalpha}

\bibliography{SingleShiftBiblio}

\end{document}